 \documentclass[final,3p,times]{elsarticle}

\usepackage{times}
\usepackage[ansinew]{inputenc}
\usepackage{amssymb,amsmath,amsthm}
\usepackage{url}
\usepackage{tabularx}
\usepackage{calc}
\usepackage{enumitem}
\usepackage{wrapfig}
\usepackage{verbatim}

\usepackage{tikz}

\usepackage{booktabs}
\usepackage{xspace}
\usepackage{amsmath}
\usepackage{lineno}
\usepackage[linesnumbered,ruled,resetcount]{algorithm2e}

\newtheorem{corollary}{Corollary}

\newtheorem{theorem}{Theorem}

\newtheorem{conjecture}{Conjecture}

\newtheorem{lemma}{Lemma}

\newcommand{\shortcite}{\cite}
\newcommand{\trelation}{\succ}
\newcommand{\mymapsto}{\rightarrow}

\newcommand{\teq}{$\tau$}
\newcommand{\mathteq}{\tau}

\newcommand{\schwaz}{Schwartz's Conjecture}

\newcommand{\captain}{captain}
\newcommand{\capgraph}[1]{Dom{(#1)}}
\newcommand{\captaingraph}{domination graph}

\newcommand{\sibling}{sibling} 

\newcommand{\myvspace}[1]{\vspace*{0pt}}
\newcommand{\myfig}[1]{{Figure}~\ref{#1}}
\newcommand{\myfigs}{{Figures}}
\newcommand{\mytable}[1]{{Table}~\ref{#1}}
\newcommand{\broom}[3]{(#1; #2; #3)}
\newcommand{\edge}[2]{(#1,#2)}
\newcommand{\dcapgraph}[1]{\mathfrak{Dom}{(#1)}}

\newcommand{\caseinproofskip}{\medskip}
\newcommand{\comments}[1]{}
\newcommand{\condition}[1]{Condition~(#1)}
\newcommand{\conditions}[2]{Conditions~(#1)-(#2)}
\newcommand{\assumecontradiction}{Assume for the sake of contradiction that $R_1$ and $R_2$ are two minimal {\teq}-retentive sets of $T$ that satisfy the above two conditions in the lemma.}
\newcommand{\casebegin}{$\bullet$\hspace{1em}}
\newcommand{\mymod}[2]{#1\mod #2}

\begin{document}
\begin{frontmatter}

\title{A Further Step Towards an Understanding of the Tournament Equilibrium Set}
\author{Yongjie Yang
}
\ead{yyongjie@mmci.uni-saarland.de}
\address{Universit{\"{a}}t des Saarlandes, Saarbr\"{u}cken, Germany}

\begin{abstract}
We study some problems pertaining to the tournament equilibrium set ({\teq} for short).
A tournament $H$ is a {\teq}-retentive tournament if there is a tournament $T$ which has a minimal {\teq}-retentive set $R$ such that $T[R]$ is isomorphic to $H$. We study {\teq}-retentive tournaments and achieve many significant results.
In particular, we prove that there are no {\teq}-retentive tournaments of size 4, only 2 non-isomorphic {\teq}-retentive tournaments of sizes 5 and 6, respectively, and 26 non-isomorphic {\teq}-retentive tournaments of size 7.
For three tournaments $H_1, H_2$ and $T$, we say $T$ is a $(H_1,H_2)$-{\teq}-retentive tournament if $T$ has two minimal {\teq}-retentive sets $R_1$ and $R_2$ such that $T[R_1]$ and $T[R_2]$ are isomorphic to $H_1$ and $H_2$, respectively. We show that there are no $(H_1,H_2)$-retentive tournaments for $H_1$ and $H_2$ being small tournaments. Our results imply that Schwartz's Conjecture holds in all tournaments of size at most 14.
Finally, we study Schwartz's Conjecture in several classes of tournaments.
To achieve these results, we study the relation between (directed) domination graphs of tournaments and {\teq}-retentive sets, and derive a number of properties on minimal {\teq}-retentive sets.
\end{abstract}

\end{frontmatter}

\section{Introduction}
\label{sec:introduction}
Tournaments play a significant role in decision making~\cite{DBLP:conf/ijcai/AltmanPT09,DBLP:journals/jair/AzizBFHLS15,DBLP:conf/aaai/BrandtBH14,DBLP:conf/ijcai/BrandtBS11,
DBLP:journals/scw/BrandtCKLNSST13,DBLP:conf/atal/BrandtHKS13,DBLP:conf/atal/BrandtHS14,DBLP:conf/aldt/YangG13}. For instance, a group of autonomous agents may jointly decide on a course of action based on the relation of majority preference, which prescribes that an alternative dominates another alternative if a majority of agents prefer the former to the latter. If there is no tie, the relation of majority preference gives rise to a tournament---a complete and antisymmetric binary relation over the alternatives.

When tournaments are used for joint decision making, the problem of determining which alternatives should be selected as the winners is of particular importance. If there is an alternative that dominates every other alternative, then this alternative is widely recognized as the winner, the so-called {\it{Condorcet winner}}.
However, the relation of majority preference may result in tournaments where no Condorcet winner exists.
For instance, consider the preferences $a\succ b\succ c, b\succ c\succ a, c\succ a\succ b$ of three agents. 
In this case there is no straightforward notion of a ``best"
alternative. To address the problem, researchers proposed several prominent tournament solutions, which are functions that map a given tournament to a nonempty subset of alternatives.
In particular, Thomas Schwartz~\shortcite{Schwartz1990} proposed to select the {\it{tournament equilibrium set}} ($\tau$ for short) as the winning set. The tournament equilibrium set of a tournament $T$, denoted by $\tau(T)$, is recursively defined as the union of all minimal {\teq}-retentive sets of $T$. By and large, a {\teq}-retentive set of $T$ is a nonempty alternative subset $R\subseteq V(T)$ such that for every $v\in R$, $\tau(H)\subseteq R$, where $H$ is the subtournament induced by all inneighbors of $v$ (see Section~\ref{sec:preliminaries} for further details). 

Since the work of Thomas Schwartz~\shortcite{Schwartz1990}, many researchers have made much effort to investigate the properties of the tournament equilibrium set~\cite{Brandt2010a,Brandt2010,Dutta1990,Houy2009,Laffond1993,MSYangIJCAI2015}. Nevertheless, less work has focused on structural properties of subtournaments induced by minimal {\teq}-retentive sets. In particular, questions such as, ``What structures are forbidden, necessary or sufficient for a set of alternatives to form a minimal {\teq}-retentive set? What structures are forbidden for two sets of alternatives to form two minimal {\teq}-retentive sets in the same tournament?", which are of particular importance for people to comprehensively understand the tournament equilibrium set, have not been extensively explored. This paper is aimed to partially answer these questions and propose some new related questions that we believe to be important for a comprehensive understanding of the tournament equilibrium set. In the following, the size of a tournament refers to the number of alternatives of the tournament.

First, we study tournaments $H$ such that there exists a tournament $T$ which has a minimal {\teq}-retentive set $R$ such that $T[R]$ is isomorphic to $H$ (we refer to Section~\ref{sec:preliminaries} for the definitions and notations that are used in the following discussion).
We call such tournaments $H$ {\it{{\teq}-retentive tournaments}}. Let $\beta_n$ be the number of all non-isomorphic {\teq}-retentive tournaments of size $n$. Mnich, Shrestha and Yang~\shortcite{MSYangIJCAI2015} showed that all minimal {\teq}-retentive sets induce irreducible subtournaments, which implies that $\beta_2=0$ (since all tournaments of size $2$ are reducible). It is fairly easy to check that $\beta_1=\beta_3=1$. In this paper, we study $\beta_n$ for some other small values of $n$. Our results are somewhat surprising. In particular, we show that $\beta_4=0, \beta_5=\beta_6=2$ and $\beta_7=26$.
Recall that there are $4,12,56,456$ non-isomorphic tournaments of sizes $4,5,6,7$, respectively.
Our results reveal that small {\teq}-retentive tournaments are considerably rare.
In other words, in order for a small set of alternatives to form a tournament equilibrium set in a tournament, the alternatives must first fulfil several  restrictive structural properties inside themselves.
In addition, we show the upper bounds for $\beta_n$ for $n=8,9,10$. These findings provide theoretical evidence for the results of the experimental research in~\cite{BrandtS2014ORDiscrimitiveTournament}, where the authors examined the distribution of the number of alternatives returned by common tournament solutions for empirical data as well as data generated according to stochastic preference models. They showed that under various stochastic preference models, the average absolute size of the tournament equilibrium set is less than $3$, when the number of alternatives is $5$. According to our findings there are only 2 $\tau$-retentive tournaments of size 5 and no {\teq}-retentive tournaments of sizes 2 and 4, implying very small possibility of a returned ${\tau}$ of size 5, and leaving many of sizes 1 or 3.

Schwartz conjectured~\cite{Schwartz1990} that every tournament has a unique minimal $\tau$-retentive set. This conjecture is of particular importance since it is equivalent that $\tau$ having any one of a set of desirable fairness properties, including  monotonicity, independence of unchosen alternatives, weak superset and strong superset (see~\cite{Brandt2010a} for further details).
Unfortunately, {\schwaz} has been disproved very recently~\cite{DBLP:journals/scw/BrandtCKLNSST13,Brandt2013b}, indicating that $\tau$ does not satisfy any of the above mentioned fairness properties in general. In particular, a counterexample of size~24 is derived by a exhaustive search~\cite{Brandt2013b}. In spite of these negative results for $\tau$ in general, $\tau$ is still recognized by researchers as one of the most attractive tournament solutions, since {\teq} satisfies the properties for all practical purposes. On the one hand, researchers proved that Schwartz's Conjecture holds in many classes of tournaments~\cite{Brandt2010,MSYangIJCAI2015}. In particular, by enumerating all non-isomorphic tournaments of size at most~12 and checking Schwartz's Conjecture on each of them, Brandt et al.~\shortcite{Brandt2010} proved that Schwartz's Conjecture holds in all tournaments of size at most 12. These positive results are important since many real-world tournaments are small tournaments (if there are too many alternatives, it is unlikely to be able to get a tournament in many cases due to huge information involved).
On the other hand, counterexamples of large size to Schwartz's Conjecture seem extremely rare, according to an experimental research by Brandt et al.~\cite{Brandt2010}.

In this paper, we study $(H_1,H_2)$-{\teq}-retentive tournaments which are defined as tournaments that have two distinct minimal {\teq}-retentive sets $R_1$ and $R_2$ such that $T[R_1]$ and $T[R_2]$ are isomorphic to the tournaments $H_1$ and $H_2$, respectively.
Schwartz's Conjecture is equivalent to saying that there are no $(H_1,H_2)$-retentive tournaments for all possible tournaments $H_1$ and $H_2$. The counterexample of size 24 found by Brandt and Seedig~\cite{Brandt2013b} has two minimal {\teq}-retentive sets each of size 12. This means that there are $(H_1,H_2)$-{\teq}-retentive tournaments such that $|H_1|=|H_2|=12$. In this paper, we prove that there are no $(H_1, H_2)$-{\teq}-retentive tournaments for $H_1$ and $H_2$ being small tournaments. See \myfig{fig:summary} for a demonstration. A significant consequence of our results is that Schwartz's Conjecture holds in all tournaments of size at most 14. We remark that we do not enumerate tournaments of sizes~13 and~14 (in fact, enumerating all non-isomorphic tournaments of sizes 13 and 14 is an extremely time-consuming task), but instead, we derive numerous properties on minimal {\teq}-retentive sets, and then based on these properties, we show the non-existence of $(H_1,H_2)$-{\teq}-retentive tournaments for small tournaments $H_1$ and $H_2$.

Finally, based on the derived properties on minimal {\teq}-retentive sets, we examine Schwartz's Conjecture in several classes of tournaments, and prove that Schwartz's Conjecture holds on these tournaments.

To study the problems above, we exploit directed domination graphs of tournaments. The {\it{directed domination graph}} of a tournament $T$ is a directed graph with the same vertex (alternative) set as $T$. Moreover, there is an arc from a vertex $v$ to another vertex $u$, if $v$ dominates $u$ and all the other inneighbors of $u$. (Directed) domination graphs were first studied by Fisher et al.~\shortcite{DBLP:journals/gc/FisherGLMR01,DBLP:journals/jgt/FisherLMR98}.
We first introduce directed domination graphs into the study of the tournament equilibrium set. To achieve the results of this paper, we derive numerous intriguing properties on {\teq}-retentive sets. For instance, we show both necessary and sufficient conditions for a set of three vertices to be a minimal {\teq}-retentive set, which further implies an efficient algorithm to determine whether a given three vertices is a minimal {\teq}-retentive set in a tournament. We believe that these properties are helpful for people to have a comprehensive understanding of the tournament equilibrium set. Moreover, these properties are useful for further studies on the tournament equilibrium set. 

\begin{figure}[h!]
\begin{center}
\includegraphics[width=\textwidth]{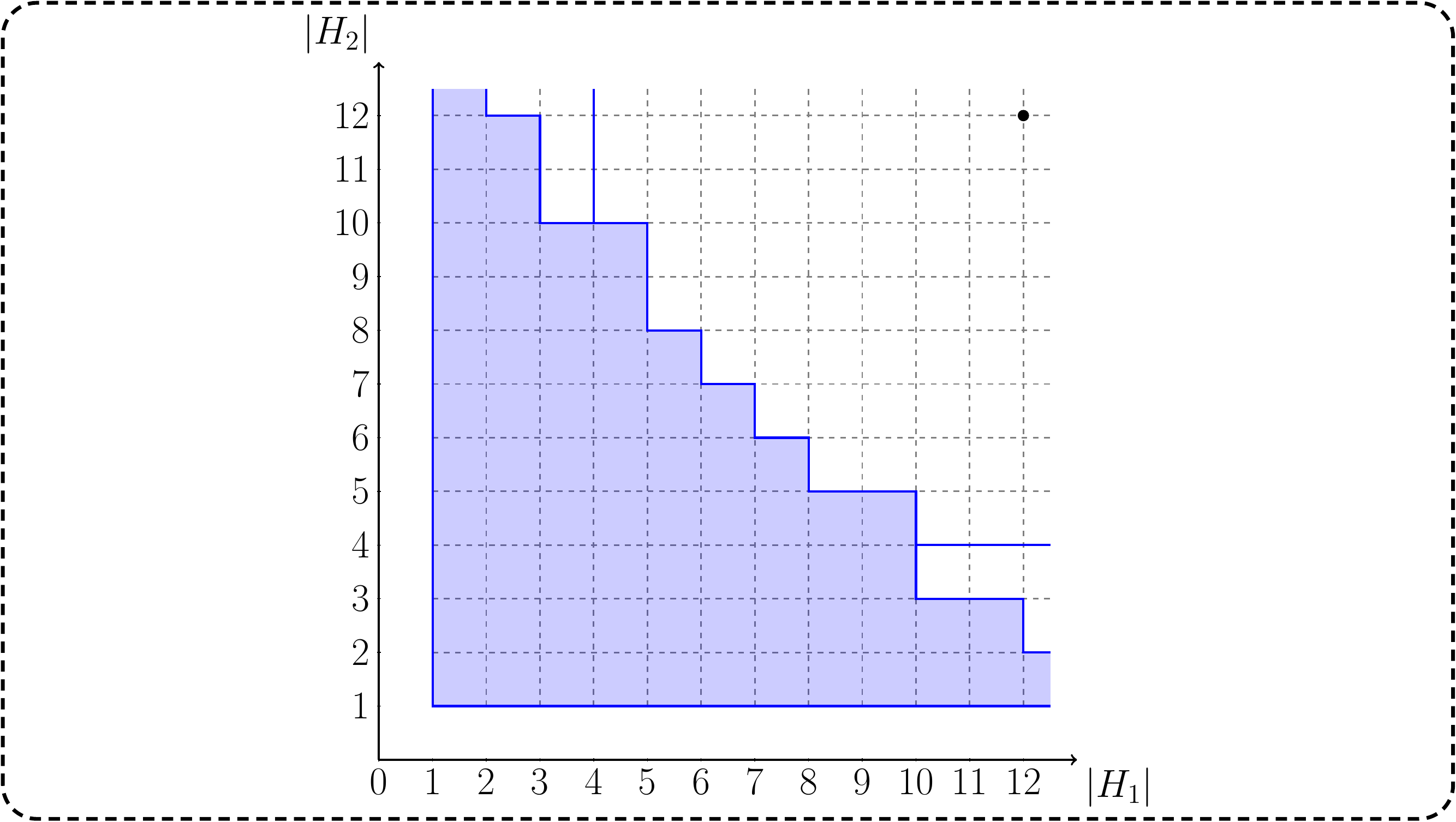}
\end{center}
\caption{This figure shows the status of $(H_1,H_2)$-retentive tournaments. A coordinate $(x,y)$ in the blue area, including the boundaries, means that there are no $(H_1,H_2)$-retentive tournaments such that $|H_1|=x$ and $|H_2|=y$. The counterexample to Schwartz's Conjecture discovered by Brandt and Seedig~\cite{Brandt2013b} is a tournament with two minimal {\teq}-retentive sets each of size 12. Thus, there are $(H_1,H_2)$-retentive tournaments such that $|H_1|=|H_2|=12$, as indicated by the dark circle on the coordinate $(12,12)$. Except the coordinate $(12,12)$, for every other positive integer coordinate $(x,y)$ that are not in the blue area, it remains open whether there exists $(H_1,H_2)$-{retentive} tournaments such that $|H_1|=x$ and $|H_2|=y$.}
\label{fig:summary}
\end{figure}

\section{Preliminaries}
\label{sec:preliminaries}
{\textbf{Tournament.}} 
A \emph{tournament} $T$ is a pair $(V(T), \trelation)$, where $V(T)$ is a set of {\it{alternatives}} and $\trelation$ is an asymmetric and complete
binary relation on $V(T)$. 
For two alternative sets $X$ and $Y$, $X \trelation Y$ means that $x\trelation y$ for every $x\in X$ and every $y\in Y$.
For ease of exposition, we use {\it{directed graphs}} to represent tournaments. Precisely, in this paper a tournament $T=(V(T), \trelation)$ is considered as a directed graph where $V(T)$ is considered as the vertex set and $\trelation$ the arc set. We refer to the textbook by West~\shortcite{Douglas2000} for readers who are not familiar with graph theory. We use the term ``vertex" for ``alternative" hereinafter.

The {\it{source}} of a tournament $T=(V(T),\trelation)$ is the vertex  $a$ so that $a\trelation b$ for every vertex $b\in V(T)\setminus \{a\}$.
From a social choice point of view, the source is called the {\it{Condorcet winner}} of the tournament. 

For a vertex $v\in V(T)$, let $N^-_{T}(v)$ denote the set of inneighbors of $v$ in $T$ and $N^{+}_{T}(v)$ the set of outneighbors of $v$, that is,  $N^-_{T}(v)=\{u\in V(T)\mid u\trelation v\}$ and $N^+_{T}(v)=\{u\in V(T)\mid v\trelation u\}$.
For a subset $B\subseteq V(T)$, $T[B]$ is the {\it{subtournament}} induced by $B$, that is, $T[B]=(B,\trelation')$ where for every $a,b\in B$, $a\trelation' b$ if and only if $a\trelation b$. A {\it{directed triangle}} is a tournament with three vertices so that each vertex has exactly one inneighbor.

A tournament $T=(V(T),\trelation)$ is {\it{isomorphic}} to another tournament $T'=(V(T'),\trelation')$ if there is a bijection $f: V(T)\mymapsto V(T')$ such that $v\trelation u$ if and only if $f(v)\trelation' f(u)$. Throughout this paper, let $T_n$ denote the set of all non-isomorphic tournaments of $n$ vertices.

A tournament $T$ is {\it{reducible}} if there is a partition $(A,B)$ of $V(T)$ such that $A\trelation B$. Otherwise, it is {\it{irreducible}}.

\smallskip

{\textbf{Tournament Equilibrium Set ({\teq} for short).} A {\it{tournament solution}} $S$ is a function that maps every tournament $T$ to a nonempty subset $S(T)\subseteq V(T)$. For a tournament solution $S$ and a tournament $T$, a nonempty subset $A\subseteq V(T)$ is an $S$-{\it{retentive set}} if for all $v\in A$ with $N^-_{T}(v)\neq \emptyset$, it holds that $S(T[N^-_{T}(v)])\subseteq A$. An $S$-retentive set $A$ is {\it{minimal}} if there is no other $S$-retentive set $B$ such that $B\subset A$. Since the set $V(T)$ of
all vertices is trivially an $S$-retentive set of $T$, $S$-retentive
sets are guaranteed to exist.
 The tournament equilibrium set $\mathteq{(T)}$ of a tournament $T$ is defined as the union of all minimal $\tau$-retentive sets of $T$~\cite{Schwartz1990}. This recursion is well-defined since $|N^-_T(v)|$ is strictly smaller than $|V(T)|$ for every $v\in V(T)$.  See~\myfig{fig:exampleteq} for some small tournaments with their tournament equilibrium sets.  


\begin{figure}[h!]
\begin{center}
\includegraphics[width=\textwidth]{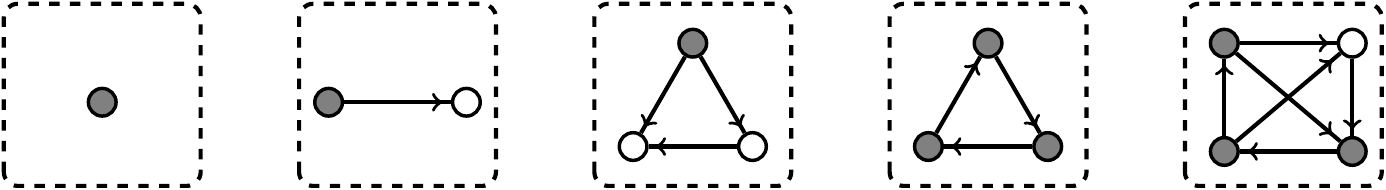}
\end{center}
\caption{The tournament equilibrium set of each tournament consists of all the gray vertices.}
\label{fig:exampleteq}
\end{figure}

{\textbf{{\teq}-Retentive Tournament.}} A tournament $H$ is a {\it{\teq}-retentive tournament} if there is a tournament $T$ that has a minimal {\teq}-retentive set $R$ so that $T[R]$ is isomorphic to $H$. 

{\textbf{{($H_1,H_2$)-{\teq}-Retentive Tournament.}}} For three tournaments $T, H_1$ and $H_2$, we say that $T$ is a {\it{$(H_1, H_2)$-{\teq}-retentive tournament}}, if $T$ has two distinct minimal {\teq}-retentive sets $R_1$ and $R_2$ such that $T[R_1]$ is isomorphic to $H_1$ and $T[R_2]$ is isomorphic to $H_2$.
For two classes $\mathcal{H}_1$ and $\mathcal{H}_2$ of tournaments, we say a tournament $T$ is a {\it{$(\mathcal{H}_1, \mathcal{H}_2)$-{\teq}-retentive tournament}} if there are $H_1\in \mathcal{H}_1$ and $H_2\in \mathcal{H}_2$ such that $T$ is a $(H_1, H_2)$-{\teq}-retentive tournament. 

\smallskip



\begin{figure}[h!]
\begin{center}
\includegraphics[width=\textwidth]{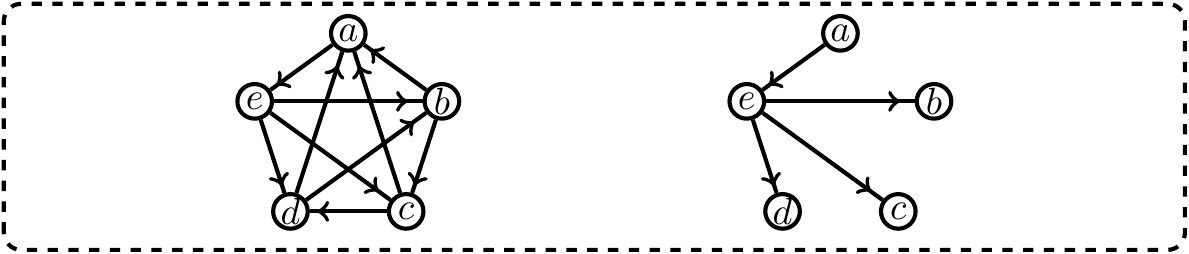}
\end{center}
\caption{A tournament (left) and its directed {\captaingraph} (right).}
\label{fig:Tfiveoneandtiscaptain}
\end{figure}


{\textbf{Domination Graph.}} Let $T$ be a tournament. A vertex $u$ is the {\it{\captain}} of another vertex $v$ if $u$ dominates $v$ and all the other inneighbors of $v$, that is, $u\trelation v$ and $\{u\}\trelation N^-_T(v)\setminus \{u\}$.  We say that $v$ is a {\it{slave}} of $u$. It is clear that every vertex has at most one captain, but may have more than one slave. Moreover, if a vertex $u$ is the captain of a vertex $v$, then $u$ is the source of $T[N^-_T(v)]$. A vertex is a {\it{captain vertex}} (resp. {\it{slave vertex}}) if $u$ is the captain (resp. a slave) of another vertex in $T$.
The {\it{\captaingraph}} (resp. {\it{directed {\captaingraph}}}) of $T$, denoted by ${\capgraph{T}}$ (resp. ${\dcapgraph{T}}$) is an undirected (resp. a directed) graph such that
(1) $V({\capgraph{T}})=V(T)$ (resp. $V({\dcapgraph{T}})=V(T)$); and
(2) there is an edge $\edge{u}{v}$ (resp. an arc $u\trelation v$) in ${\capgraph{T}}$ (resp. ${\dcapgraph{T}}$) if and only if $u$ is the {\captain} of $v$ in $T$. So the domination graph of a tournament is the restriction of the directed domination graph of the tournament without the orientation of arcs.
A vertex in $\capgraph{T}$ (resp. $\dcapgraph{T}$) is an {\it{isolated}} vertex if it is neither a captain vertex nor a slave vertex in $T$. A domination graph (resp. directed domination graph) is {\it{empty}} if it contains no edge (resp. arc), that is, it consists of only isolated vertices; otherwise it is {\it{non-empty}}.
See \myfig{fig:teqtournamentofsizesix} for an illustration.

Domination graphs were initially studied by Fisher~et~al.~\shortcite{DBLP:journals/jgt/FisherLMR98,DBLP:journals/gc/FisherGLMR01}~\footnote{Captain vertices and slave vertices are not distinguished in the definition of domination graphs by Fisher~et~al.~\shortcite{DBLP:journals/jgt/FisherLMR98,DBLP:journals/gc/FisherGLMR01}. In particular, in their definition, two vertices $v,u$ dominate a tournament $T$ if every vertex $w\in V(T)\setminus \{v,u\}$ is dominated by at least one of $\{v,u\}$. Then, the domination graph of $T$ is defined as the graph $\capgraph{T}=(V(T),E)$, where there is an edge between two vertices $v,u$ in $E$ if they dominate the tournament $T$. It is easy to verify that if two vertices $v,u$ with $u\trelation v$ dominate a tournament, then $u$ is the captain of $v$. On the other hand, if $u$ is the captain of a vertex $v$, then $v,u$ dominate the tournament. Therefore, these two definitions lead to the same {\captaingraph}.
Our adoption of the definition is due to the following reasons. First, we were not aware of the papers by Fisher~et~al.~\cite{DBLP:journals/jgt/FisherLMR98,DBLP:journals/gc/FisherGLMR01} as we initiated the study. Second, the distinguish between captain vertices and slave vertices plays important role in our study of tournament equilibrium set.
}.
So far, domination graphs were purely studied from the graph theory perspective. In this paper, we first extend the application of {\captaingraph}s to the study of the tournament equilibrium set.
Fisher~et~al.~\shortcite{DBLP:journals/jgt/FisherLMR98,DBLP:journals/gc/FisherGLMR01} derived a property of domination graphs that is very useful for our study (see Lemma~\ref{lem:dominationgraphproperty} below).
A {\it{cycle}} (resp. {\it{directed cycle}}) is a sequence of vertices $(v_0,v_1,...,v_{k-1})$ such that $v_i\neq v_j$ for all $0\leq i\neq j\leq k-1$, and there is an edge $\edge{v_i}{v_{\mymod{(i+1)}{k}}}$ (resp. an arc $v_i\trelation v_{\mymod{(i+1)}{k}}$) for all $i=0,1,...,k-1$.
The set of vertices in a cycle $C$ is denoted by $V(C)$.
An {\it{odd cycle}} is a cycle $C$ such that $|V(C)|$ is odd.
A graph is {\it{acyclic}} if there is no cycle in the graph.
A {\it{tree}} is a connected acyclic graph.
A {\it{forest}} is a collection of vertex-disjoint trees.
A path $P$ is a sequence of distinct vertices $(v_0,v_1,...,v_{k-1})$ such that there is an edge $\edge{v_i}{v_{i+1}}$ for all $i=0,...,k-2$.
We say $P$ is a path between $v_1$ and $v_k$.
The {\it{length of a path}} is the number of vertices in the path.
The {\it{distance}} between two vertices $v$ and $u$ is the length of a shortest path between them minus one.
For example, the length of a path $(v,u,w,v)$ is 3, since we have three distinct vertices $v,u,w$ in the path.
A {\it{caterpillar}} is a tree such that the removal of all degree-1 vertices yields a path.
A {\it{spiked cycle}} is a connected graph such that the removal of all degree-1 vertices yields a cycle.
\myfig{fig:dominationgraphs} illustrates the following Lemma.


\begin{lemma}\cite{DBLP:journals/jgt/FisherLMR98}
\label{lem:dominationgraphproperty}
Let $T$ be a tournament. Then ${\capgraph{T}}$ is either a spiked odd cycle, with or without isolated vertices, or a forest of caterpillars.
\end{lemma}

See \myfig{fig:dominationgraphs} for an illustration of Lemma~\ref{lem:dominationgraphproperty}.

\begin{figure}
\begin{center}
\includegraphics[width=\textwidth]{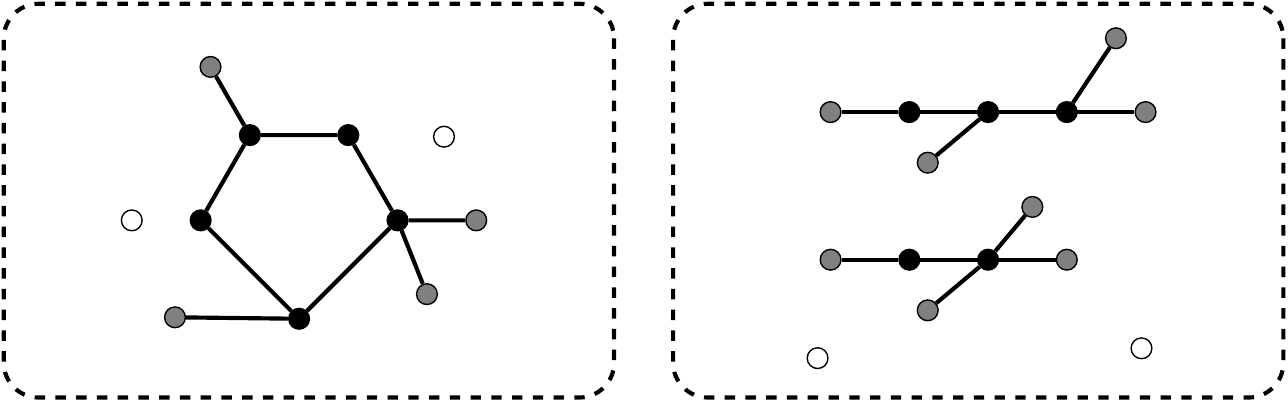}
\end{center}
\caption{An illustration of Lemma~\ref{lem:dominationgraphproperty}. Degree-1 vertices are in gray, isolated vertices in white, and vertices in an odd cycle or a path in black. Every domination graph is either a spiked odd cycle, with or without isolated vertices (as the graph on the left side), or a forest of caterpillars (as the graph on the right side).
}
\label{fig:dominationgraphs}
\myvspace{-20pt}
\end{figure}

\section{{\teq}-Retentive Tournaments}\label{sec:numberofteqretentivetournaments}
In this section, we study {\teq}-retentive tournaments of small size. In particular, we focus on the following questions:

``How many non-isomorphic {\teq}-retentive tournaments of size $n$, and what structural properties do they fulfill?"

In what follows, let $\beta_n$ denote the number of all non-isomorphic {\teq}-retentive tournaments of size $n$. It is clear that $\beta_1=1$. Mnich, Shrestha and Yang~\shortcite{MSYangIJCAI2015} studied the following lemma.

\begin{lemma}
\label{lem:teqirreducible}
Every minimal {\teq}-retentive set of a tournament $T$ induces an irreducible subtournament of $T$.
\end{lemma}

Since tournaments with two vertices are reducible, Lemma~\ref{lem:teqirreducible} implies that 
$\beta_2=0$.
It is clear that a tournament with three vertices is irreducible only if the three vertices form a directed triangle. Since a directed triangle is a {\teq}-retentive tournament, we have that $\beta_3=1$.
Now we investigate $\beta_4$. Let's first recall and study some useful properties of {\teq}-retentive sets. 

\begin{lemma}\cite{MSYangIJCAI2015}
\label{lem:teqsourceuniqueminial}
A tournament $T$ has a unique minimal {\teq}-retentive set $R$ consisting of only one vertex $v$, if and only if $v$ is the source of $T$.
\end{lemma}

The following lemmas 
connect directed {\captaingraph}s with minimal {\teq}-retentive sets.
%
%
\begin{lemma}
\label{lem:captainistheuniqueteqofslavery}
Let $T$ be a tournament. A vertex $v\in V(T)$ is the captain of another vertex $u\in V(T)$ in $T$ if and only if $\{v\}$ is the unique minimal {\teq}-retentive set of $T[N^-_T(u)]$. Moreover, if $v$ is the captain of $u$, and $u$ is in a minimal {\teq}-retentive set $R$ of $T$, then $v$ is also in $R$.
\end{lemma}
\begin{proof}
We begin the proof with the first part.
If $v$ is the captain of $u$ in $T$, then $v$ dominates $u$ and all the other inneighbors of $u$. Due to Lemma~\ref{lem:teqsourceuniqueminial}, $\{v\}$ is the unique minimal {\teq}-retentive set of $T[N^-_T(u)]$. On the other hand, if $\{v\}$ is the unique minimal {\teq}-retentive set of $T[N^-_T(u)]$, according to Lemma~\ref{lem:teqsourceuniqueminial}, $v$ is the source of $T[N^-_T(u)]$; thus $v$ is the captain of $u$ in $T$.

Now we prove the second part. Due to the first part, if $v$ is the captain of $u$ in $T$, then $\{v\}=\mathteq{(T[N^-_T(u)])}$. Then, according to the definition of {\teq}-retentive sets, if $u$ is included in a minimal {\teq}-retentive set $R$ of $T$, so is $v$.
\end{proof}

\begin{lemma}
\label{lem:inneighboronesouce}
Let $T$ be a tournament and $R$ be a minimal {\teq}-retentive set of $T$. Let $v$ and $u$ be two vertices in $R$. Then, $v$ is the captain of $u$ in $T[R]$ if and only if $v$ is the captain of $u$ in $T$.
\end{lemma}
\begin{proof}
It is clear that if $v$ is the captain of $u$ in $T$, then $v$ is the captain of $u$ in $T[R]$. We prove the other direction.
Since $R$ is a minimal {\teq}-retentive set of $T$ and $u\in R$, we have that $\mathteq ({T[N^-_T(u)]})\subseteq N^-_{T[R]}(u)$. Since $v$ is the captain of $u$ in $T[R]$, $v$ is the source of $T[N^-_{T[R]}(u)]$. According to Lemma~\ref{lem:teqirreducible}, every minimal {\teq}-retentive set induces an irreducible subtournament; and thus, it holds that $\mathteq{(T[N^-_T(u)])}=\{v\}$. Then, according to Lemma~\ref{lem:captainistheuniqueteqofslavery}, $v$ is the captain of $u$ in $T$.
\end{proof}

Lemma~\ref{lem:inneighboronesouce} directly implies that the domination graph of $T[R]$ is a subgraph of the domination graph of $T$, for every minimal {\teq}-retentive set $R$ of $T$. Now we study a forbidden structure of minimal {\teq}-retentive sets, as summarized in the following lemma.

\begin{lemma}
\label{lem:captaingraphcirclenoteq}
Let $T$ be a tournament and $R$ be a vertex subset. If $\dcapgraph{T[R]}$ has a directed cycle $C$ such that $R\setminus V(C)\neq \emptyset$, then $R$ cannot be a minimal {\teq}-retentive set of $T$.
\end{lemma}
\begin{proof}
Assume for the sake of contradiction that $R$ is a minimal {\teq}-retentive set of $T$ and $C=(c_0,c_1,...,c_{k-1})$ is a directed cycle in $\dcapgraph{T[R]}$.
Therefore, for all $i=0,1,...,k-1$, $c_i$ is the captain of $c_{\mymod{(i+1)}{k}}$ in ${T[R]}$. According to Lemma~\ref{lem:inneighboronesouce}, $c_i$ is the captain of $c_{\mymod{(i+1)}{k}}$ in $T$. Then, according to Lemma~\ref{lem:captainistheuniqueteqofslavery}, $\{c_{i}\}=\mathteq(T[N^-_T(c_{\mymod{(i+1)}{k}})])$ for every $i=0,1,...,k-1$, which implies that $V(C)$ is also a {\teq}-retentive set of $T$. Since $R\setminus V(C)\neq \emptyset$, this contradicts with the assumption that $R$ is a minimal {\teq}-retentive set. 
\end{proof}


Now we are ready to show our result concerning the number of {\teq}-retentive tournaments of size $4$.

\begin{figure}[h!]
\begin{center}
\includegraphics[width=0.9\textwidth]{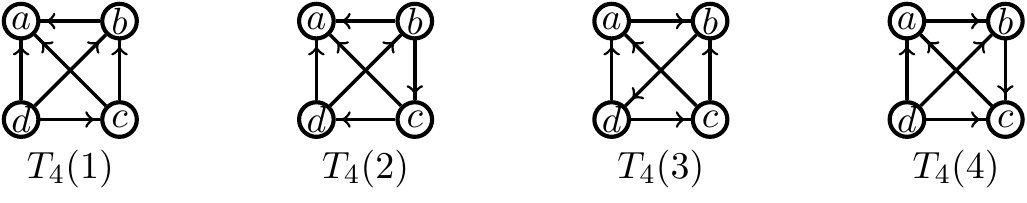}
\end{center}
\caption{All non-isomorphic tournaments of size 4.}
\label{fig:alltfour}
\myvspace{-16pt}
\end{figure}

\begin{theorem}
\label{lem:noretentivesetofsizefour}
There are no {\teq}-retentive tournaments of size~4.
\end{theorem}
\myvspace{-8pt}
\begin{proof}
\myfig{fig:alltfour} shows all non-isomorphic tournaments of size~4. Since $T_4(1), T_4(2)$ and $T_4(4)$ are all reducible, according to Lemma~\ref{lem:teqirreducible}, none of them is a {\teq}-retentive tournament. Now let's consider $T_4(3)$. It is easy to check that $(b,d,c)$ is a directed cycle (triangle) in $\dcapgraph{T_4(3)}$. Then, according to Lemma~\ref{lem:captaingraphcirclenoteq}, there is no tournament $T$ which has a minimal {\teq}-retentive set $R$ such that $T[R]$ is isomorphic to $T_4(3)$. Thus, $T_4(3)$ cannot be a {\teq}-retentive tournament either.
\end{proof}

In the following, we investigate {\teq}-retentive tournaments of size $5$.
According to Lemma~\ref{lem:teqirreducible}, we need only to study irreducible tournaments.
All non-isomorphic and irreducible tournaments of size 5 are shown in \myfig{fig:tfive}.


\begin{figure}[h!]
\begin{center}
\includegraphics[width=\textwidth]{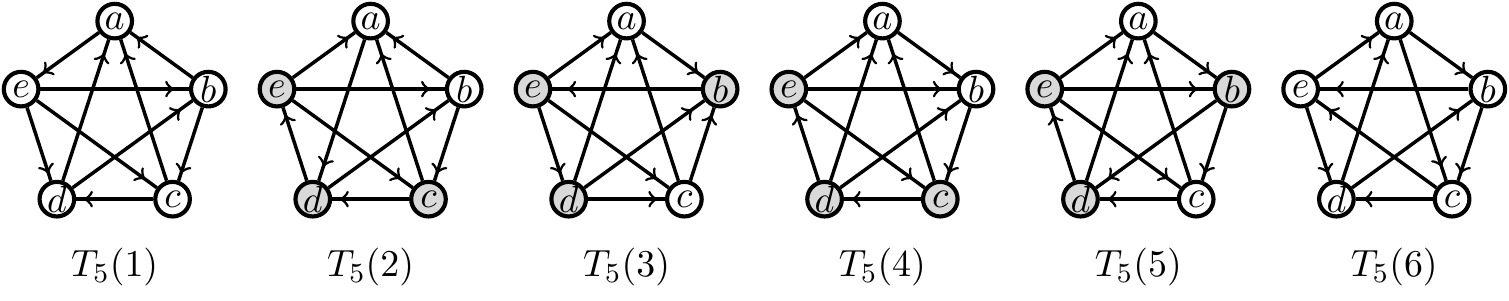}
\end{center}
\caption{All non-isomorphic and irreducible tournaments of size $5$. In each $T_5(i)$, $i\in \{2,3,4,5\}$, the three vertices in gray form a directed triangle in $\dcapgraph{T_5(i)}$.}
\label{fig:tfive}
\end{figure}

\begin{theorem}
\label{thm:teqisomorphctot5}
Every {\teq}-retentive tournament of size 5 is either isomorphic to $T_5(1)$ or to $T_5(6)$ in~\myfig{fig:tfive}.
\end{theorem}
\myvspace{-8pt}
\begin{proof}
See \myfig{fig:tfive} for all non-isomorphic and irreducible tournaments of size $5$. Since there is a directed cycle in each $\dcapgraph{T_5(i)}$ for all $i=2,3,4,5$ (see the caption of \myfig{fig:tfive} for further details), according to Lemma~\ref{lem:captaingraphcirclenoteq}, none of $\{T_5(2),T_5(3),T_5(4),T_5(5)\}$ is a {\teq}-retentive tournament.
On the other hand, both $T_5(1)$ and $T_5(6)$ have a unique minimal {\teq}-retentive set consisting of all the five vertices; and thus, they are {\teq}-retentive tournaments.
\end{proof}

Now we study $\beta_n$ for $n=6,7,8,9,10$.
We first study some important lemmas which help us filter out a large number of non-{\teq}-retentive tournaments.
We say that a {\teq}-retentive set $R$ of a tournament $T$ is $c$-{\it{locally bounded}} for some constant $c$, if for every $v\in R$, it holds that $|\mathteq{(T[N^-_T(v)])}|\leq c$.

The following lemma characterizes both the sufficient and necessary conditions for minimal {\teq}-retentive sets of size $3$.
Roughly, it states that three vertices $a,b,c$ form a minimal {\teq}-retentive set if and only if they cyclically dominate each other and, moreover, a majority of them dominate every other vertex.

\begin{lemma}
\label{lem:t3minimalretentiveset}
Let $T$ be a tournament and $R=\{a,b,c\}\subseteq V(T)$ be a 3-subset of $V(T)$. Then, $R$ is a minimal {\teq}-retentive set of $T$ if and only if

(1) $T[R]$ is a directed triangle; and

(2) no vertex in $T$ dominates two vertices in $R$.
\end{lemma}

\begin{proof}
We first prove that if $R$ is a minimal {\teq}-retentive set, then the above two conditions hold.

$(\Rightarrow:)$ Suppose that $\{a,b,c\}$ is a minimal {\teq}-retentive set of $T$. According to Lemma~\ref{lem:teqirreducible}, we know that $T[R]$ is irreducible. This implies that $T[R]$ is a directed triangle; thus {\condition{1}} holds. Without loss of generality, assume that $a\trelation b, b\trelation c$ and $c\trelation a$ in $T$. Since $\mathteq{(T[N^-_T(a)])}\neq \emptyset$ (according to the definition of {\teq}) and $\mathteq{(T[N^-_T(a)])} \subseteq N^-_{T[R]}(a)=\{c\}$, we have that $\mathteq{(T[N^-_T(a)])}=\{c\}$. Then according to Lemma~\ref{lem:captainistheuniqueteqofslavery}, $c$ is the captain of $a$ in $T$. Analogously, we can show that $a$ is the captain of $b$, and $b$ is the captain of $c$ in $T$. Therefore, any vertex that dominates $a$ is dominated by $c$, any vertex that dominates $c$ is dominated by $b$, and any vertex that dominates $b$ is dominated by $a$. This directly implies {\condition{2}}.

Now we prove the other direction.

$(\Leftarrow:)$ 
Without loss of generality, assume that $a\trelation b, b\trelation c$ and $c\trelation a$ in $T$. Since no vertex in $T$ dominates two of $\{a,b,c\}$ ({\condition{2}}), any vertex $v\in V(T)\setminus R$ which dominates $a$ must be dominated by $c$. Therefore, $c$ is the captain of $a$ in $T$. According to Lemma~\ref{lem:captainistheuniqueteqofslavery}, $\{c\}$ is the unique minimal {\teq}-retentive set of $T[N^-_T(a)]$. Analogously, 
$\{b\}$ and $\{a\}$ are the unique minimal {\teq}-retentive sets of $T[N^-_T(c)]$ and $T[N^-_T(b)]$, respectively. In summary, $\{a,b,c\}$ is a {\teq}-retentive set of $T$. Since there is no source in $T$ (directly implied by {\conditions{1}{2}}: since $a,b,c$ form a directed triangle, none of them can be the source of $T$; since no vertex in $T$ dominates two vertices in $R$, no vertex in $T\setminus R$ can be the source of $T$), according to Lemma~\ref{lem:teqsourceuniqueminial}, any minimal {\teq}-retentive set must contain at least two vertices. However, there is no minimal {\teq}-retentive set of size $2$ according to Lemma~\ref{lem:teqirreducible}. Thus, $R$ is a minimal {\teq}-retentive set of $T$.
\end{proof}

Based on Lemma~\ref{lem:t3minimalretentiveset}, we derive the following lemma.
For a vertex $v$ and three of its inneighbors $a,b,c$, we say $\{a,b,c\}$ is a {\it{tri-captain}} of $v$ if (1) the three vertices $a,b,c$ form a directed cycle; and (2) no other inneighbor of $v$ dominates at least two of $\{a,b,c\}$. In fact, due to Lemma~\ref{lem:t3minimalretentiveset}, if $\{a,b,c\}$ is a tri-captain of $v$ in a tournament $T$, then $\{a,b,c\}$ is a minimal {\teq}-retentive set of $T[N^-_T(v)]$. The proof of Lemma~\ref{lem:t3minimalretentiveset} in fact also implies that a subset $R$ of size $3$ is a minimal {\teq}-retentive set if and only if $\dcapgraph{T}[R]$ is a directed triangle.
This implication is used in the proof of the following lemma.

\begin{lemma}
\label{lem:eachtournamentonecaptaintriangle}
Every tournament has at most one minimal {\teq}-retentive set of size $3$.
\end{lemma}
\begin{proof}
We prove the lemma by contradiction. Assume that there is a tournament $T$ which has two distinct minimal {\teq}-retentive sets $\{a(1),b(1),c(1)\}$ and $\{a(2),b(2),c(2)\}$.
According to Lemma~\ref{lem:t3minimalretentiveset}, both $\{a(1),b(1),c(1)\}$ and $\{a(2),b(2),c(2)\}$ form directed triangles. Without loss of generality, assume that $a(i)\trelation b(i), b(i)\trelation c(i)$ and $c(i)\trelation a(i)$, for both $i=1,2$. Then, according to the proof of Lemma~\ref{lem:t3minimalretentiveset}, $a(i)$ is the captain of $b(i)$, $b(i)$ is the captain of $c(i)$, and $c(i)$ is the captain of $a(i)$ in $T$. Therefore, both $(a(1),b(1),c(1))$ and $(a(2),b(2),c(2))$ are cycles in ${\capgraph{T}}$. However, since $\{a(1),b(1),c(1)\}\neq \{a(2),b(2),c(2)\}$, this contradicts with Lemma~\ref{lem:dominationgraphproperty}.
\end{proof}

The following lemma states that if a vertex $v$ in a minimal {\teq}-retentive set $R$ has at most 5 inneighbors in $R$, and has a tri-captain $\{a,b,c\}$ in the subtournament $T[R]$, then $\{a,b,c\}$ is the unique minimal {\teq}-retentive set of the subtournament induced by all its inneighbors in $T$.

\begin{lemma}
\label{lem:localboundedsmallteqtournament}
Let $T$ be a tournament, $R$ be a minimal {\teq}-retentive set of $T$, and $v$ be a vertex in $R$. If $v$ has a tri-captain $\{a,b,c\}$ in $T[R]$ and $|N^-_{T[R]}(v)|\leq 5$,
then $\{a,b,c\}=\mathteq{(T[N^-_T(v)])}$.
\end{lemma}
\begin{proof}
The main idea of the proof is that we exclude the possibilities of $|\mathteq{(T[N^-_T(v)])}|=1$ and $|\mathteq{(T[N^-_T(v)])}|=5$ based on some properties of {\teq}-retentive sets. This leaves only the possibility that $|\mathteq{(T[N^-_T(v)])}|=3$. Then, based on Lemma~\ref{lem:t3minimalretentiveset}, we can show that $\{a,b,c\}=\mathteq{(T[N^-_T(v)])}$. 
The formal proof is as follows.

Let $R'$ be any minimal {\teq}-retentive set of $T[N^-_T(v)]$. We shall show that $R'=\{a,b,c\}$.
Since $R$ is a minimal {\teq}-retentive set of $T$ and $v\in R$, it holds that $\mathteq{(T[N^-_T(v)])}\subseteq N^-_{T[R]}(v)$. Thus, $R'\subseteq N^-_{T[R]}(v)$.
Since $|N^-_{T[R]}(v)|\leq 5$, it holds that $|R'|\leq 5$. According to Lemma~\ref{lem:teqirreducible} and Theorem~\ref{lem:noretentivesetofsizefour}, there are no minimal {\teq}-retentive sets of sizes $2$ and $4$. Therefore, either $|R'|=1$ or $|R'|=3$ or $|R'|=5$. In the following, we exclude the possibility of $|R'|=1$ and $|R'|=5$. Since $\{a,b,c\}$ is a tri-captain of $v$ in $T[R]$, the three vertices $a,b,c$ form a directed triangle. Moreover, no vertex in $N^-_{T[R]}(v)$ dominates two of $\{a,b,c\}$.
\caseinproofskip

Case~1. $|R'|=1$.

Since $a,b,c$ form a directed triangle, none of them is the source of $T[N^-_T(v)]$. Since no vertex in $N^-_{T[R]}(v)\setminus \{a,b,c\}$ dominates all of $\{a,b,c\}$, none of $N^-_{T[R]}(v)\setminus \{a,b,c\}$ is the source of $T[N^-_T(v)]$ either. Then, according to Lemma~\ref{lem:teqsourceuniqueminial}, it cannot be that $|R'|=1$.
\caseinproofskip

Case~2. $|R'|=5$.

Since $|N^-_{T[R]}(v)|\leq 5$ and $\{a,b,c\}\subseteq N^-_{T[R]}(v)$, it holds that $\{a,b,c\}\subset R'$. According to Theorem~\ref{thm:teqisomorphctot5}, $T[R']$ is either isomorphic to $T_5(1)$ or to $T_5(6)$ in \myfig{fig:tfive}. However, this cannot be the case. The reason is as follows:
there is no directed triangle in both $T_5(1)$ and $T_5(6)$ such that every vertex not in the triangle is dominated by at least two vertices in the triangle, but $T[R']$ contains such a directed triangle formed by $a,b,c$.

Now it leaves only the case that $|R'|=3$. According to Lemma~\ref{lem:t3minimalretentiveset}, $\{a,b,c\}$ is a minimal {\teq}-retentive set of $T[N^-_{T[R]}(v)]$.
According to Lemma~\ref{lem:eachtournamentonecaptaintriangle}, $\{a,b,c\}$ is the unique minimal {\teq}-retentive set of size $3$ of $T[N^-_{T[R]}(v)]$. That is, there exists no further directed triangle formed by $\{x,y,z\}\subseteq N^-_{T[R]}(v)$, such that $\{x,y,z\}\neq R'$ and no vertex in $N^-_{T[R]}(v)$ dominates two of $\{x,y,z\}$ (see Lemma~\ref{lem:t3minimalretentiveset}). Then, due to Lemma~\ref{lem:t3minimalretentiveset} and the fact that $R'\subseteq N^-_{T[R]}(v)$, it must hold that $R'=\{a,b,c\}$, implying that $\{a,b,c\}=\mathteq{(T[N^-_T(v)])}$.
\end{proof}

\begin{lemma}\label{lem:tauthreeequalretentive}
Let $T$ be a tournament. If $\mathteq(T)=\{a,b,c\}$, then $\{a,b,c\}$ is the unique minimal {\teq}-retentive set of $T$.
\end{lemma}
\begin{proof}
We prove the lemma as follows. Let $\mathteq(T)=\{a,b,c\}$. According to Lemma~\ref{lem:teqirreducible}, there are no minimal {\teq}-retentive sets of size two. Thus, either each of $\{\{a\},\{b\},\{c\}\}$ is a minimal {\teq}-retentive set of $T$, or $\{a,b,c\}$ is the unique minimal {\teq}-retentive set of $T$. According to Lemma~\ref{lem:teqsourceuniqueminial}, if it is the former case, then every of $\{a,b,c\}$ is a source of $T$. However, every tournament has at most one source; a contradiction. Thus, it must be the latter case.
\end{proof}

Now we explore another forbidden structure to {\teq}-retentive tournaments. By and large, this structure characterizes all tournaments which have a proper 3-locally bounded {\teq}-retentive set in which every vertex has bounded indegree 5.

\begin{lemma}
\label{lem:threeloallyboudnedminimalteqagainst}
Let $H$ be a tournament. If $H$ has a 3-locally bounded {\teq}-retentive set $R$ such that
(1) $V(H)\setminus R\neq\emptyset$; and
(2) for every vertex $v\in R$, $|N^-_H(v)|\leq 5$,
then $H$ cannot be a {\teq}-retentive tournament.
\end{lemma}
\begin{proof}
We prove the lemma by contradiction. The main idea of the proof is that if a tournament $T$ has a subtournament $H$ which has a 3-locally bounded {\teq}-retentive set $R$ as stated in the lemma, and $V(H)$ is a minimal {\teq}-retentive set of $T$, then due to Lemma~\ref{lem:localboundedsmallteqtournament} and some other properties of {\teq}-retentive sets, for every vertex $v\in R$ it holds that $\mathteq{(T[N^-_T(v)])}\subseteq R$, which implies that $R$ is also a {\teq}-retentive set of the tournament $T$; contradicting with the minimality of $V(H)$. The formal proof is as follows.

Assume that there is an $R\subseteq V(H)$ that satisfies the conditions in the above lemma but $H$ is a {\teq}-retentive tournament. Then there is a tournament $T$ which has a minimal {\teq}-retentive set $S$ such that $T[S]$ is isomorphic to $H$. For ease of exposition, let $H=T[S]$ (this does not affect the proof since $T[S]$ is isomorphic to $H$); thus $S=V(H)$. We claim that $R$ is a {\teq}-retentive set of $T$. Let $v$ be any vertex in $R$. Due to Lemma~\ref{lem:teqirreducible}, $H$ is irreducible; and thus, $N^-_H(v)\neq\emptyset$. This means that $|\mathteq{(H[N^-_H(v)])}|\neq 0$. Then, since $R$ is a 3-locally bounded minimal {\teq}-retentive set of $H$, we need to distinguish between the following two cases. We shall prove that in both cases $\mathteq{(T[N^-_T(v)])}\subseteq R$ holds.
\caseinproofskip

Case~1. $|\mathteq{(H[N^-_H(v)])}|=1$.

Let $\{u\}=\mathteq{(H[N^-_H(v)])}$. Since $R$ is a minimal {\teq}-retentive set of $H$ and $v\in R$, we have that $u\in R$. According to Lemma~\ref{lem:captainistheuniqueteqofslavery}, $u$ is the captain of $v$ in $H$. Since $V(H)$ is a minimal {\teq}-retentive set of $T$, according to Lemma~\ref{lem:inneighboronesouce}, $u$ is still the captain of $v$ in $T$. Then, according to Lemma~\ref{lem:captainistheuniqueteqofslavery}, $\mathteq{(T[N^-_T(v)])}=\{u\}\subseteq R$.
\caseinproofskip

Case~2. $|\mathteq{(H[N^-_H(v)])}|=3$.

Let $\{a,b,c\}=\mathteq{(H[N^-_H(v)])}$. Due to Lemma~\ref{lem:tauthreeequalretentive}, $\{a,b,c\}$ is the unique minimal {\teq}-retentive set of $H[N^-_H(v)]$. Then, according to Lemma~\ref{lem:t3minimalretentiveset}, $\{a,b,c\}$ is a tri-captain of $v$ in $H$. Since $R$ is a minimal {\teq}-retentive set of $H$ and $v\in R$, we have that $\{a,b,c\}\subseteq R$. Since $|N^-_H(v)|\leq 5$ and $V(H)$ is a minimal {\teq}-retentive set of $T$, according to Lemma~\ref{lem:localboundedsmallteqtournament}, $\mathteq{(T[N^-_T(v)])}=\{a,b,c\}\subseteq R$.

Notice that it cannot be $|\mathteq{(H[N^-_H(v)])}|=2$. The reason is as follows. Due to Lemma~\ref{lem:teqirreducible}, there is no minimal {\teq}-retentive set of size $2$; thus if $|\mathteq{(H[N^-_H(v)])}|=2$, it must be that $H[N^-_H(v)]$ has two minimal {\teq}-retentive sets each of size 1. This implies that $H[N^-_H(v)]$ has two distinct sources. However, every tournament has at most one source; a contradiction.

Due to the above analysis, for any vertex $v\in R$, it holds $\mathteq{(T[N^-_{T}(v)])}\subseteq R$; implying that $R$ is a {\teq}-retentive set of $T$. However, since $V(H)\setminus R\neq\emptyset$, this contradicts with the assumption that $V(H)$ is a minimal {\teq}-retentive set of $T$. 
\end{proof}

We write a C++ program to check whether there is a 3-locally bounded {\teq}-retentive set that satisfies the conditions in Lemma~\ref{lem:threeloallyboudnedminimalteqagainst} for each of the  non-isomorphic and irreducible tournaments of sizes $6,7,8,9,10$ (our tournaments are from https://cs.anu.edu.au/$\sim$bdm/data/digraphs.html,  maintained by Brendan McKay). In particular, for a tournament $H$ with $6\leq |V(H)|\leq 10$, we enumerate all proper subsets $S$ of $V(H)$ and check if $S$ is a 3-locally bounded {\teq}-retentive set that satisfies the conditions in Lemma~\ref{lem:threeloallyboudnedminimalteqagainst} (we discuss later how we check if a subset is a 3-locally bounded {\teq}-retentive set). After filtering out these tournaments, there remain $2, 26, 395, 30596$ and $3881175$ tournaments of sizes $6,7,8,9,10$, respectively. We then further check that all the $2$ and $26$ tournaments of sizes $6$ and $7$, respectively, are {\teq}-retentive tournaments (each has a unique minimal {\teq}-retentive set consisting of all vertices in the tournament). Thus, $\beta_6=2$ and $\beta_7=26$. \myfig{fig:teqtournamentofsizesix} shows the 2 non-isomorphic {\teq}-retentive tournaments of size $6$ along with their directed {\captaingraph}s. A full list of the 26 non-isomorphic {\teq}-retentive tournaments of size $7$, in terms of the upper triangles of their adjacent matrices, is given in \mytable{tab:teqretentivetournametnsofsizeseven}. See also Appendix for graph representations of these tournaments and their directed {\captaingraph}s. The following theorem directly follows from the above discussions.

\begin{table}
\begin{center}
\begin{tabular}{c c c c}
111000101101101011101 &
000001000101000100000 &
000001110001010110000 &
000011000101000100001 \\
000011100001100100010 &
000011101001100010101 &
000101110001010110101 &
000110101001010010011 \\
001010100000101001000 &
010001000101001100000 &
010001000101010100001 &
010001110001110110101 \\
010010000101100100011 &
010010010001100101101 &
010010100001101100010 &
011000000100110001001 \\
011000000101100011001 &
011000100000111100101 &
100100101101010010011 &
100100101101100010111 \\
100100110101010110011 &
100100110101100110111 &
100100111001010110111 &
101000110001000111101 \\
111000101101011101101 &
111000110011011111000 \\
\end{tabular}
\caption{A full list of all non-isomorphic {\teq}-retentive tournaments of size $7$. Each tournament is given as the upper triangle of the adjacency matrix in row order, on one line without spaces. Each row, except the last row, shows four tournaments.}
\label{tab:teqretentivetournametnsofsizeseven}
\end{center}
\end{table}

\begin{theorem}
\label{thm:betas}
$\beta_6=2, \beta_7=26, \beta_8\leq 395, \beta_9\leq 30596$ and $\beta_{10}\leq 3881175$.
\end{theorem}


\begin{figure}
\begin{center}
\includegraphics[width=\textwidth]{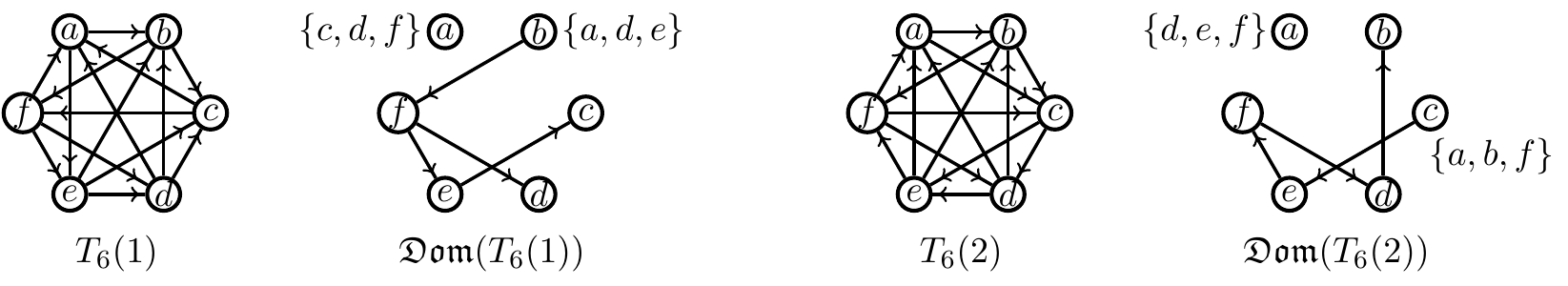}
\end{center}
\caption{The two {\teq}-retentive tournaments of size $6$ and their directed {\captaingraph}s. In the directed {\captaingraph}s, if a vertex has a tri-captain, then its tri-captain is labeled next to the vertex.}
\label{fig:teqtournamentofsizesix}
\end{figure}

Now we discuss how we check if a subset $S$ is a 3-locally bounded {\teq}-retentive set of a tournament $H$.
Due to Lemma~\ref{lem:t3minimalretentiveset}, it suffices to check for each vertex $v\in S$ which has no captain $u\in S$ in $H$ (if $v$ has a captain $u$, then according to Lemma~\ref{lem:captainistheuniqueteqofslavery},  $\mathteq{(T[N^-_H(v)])}=\{u\}$) whether there are three inneighbors $x,y,z$ of $v$ in $S$ such that (1) $x,y,z$ form a directed triangle; and (2) no vertex in $N^-_{T[H]}(v)\setminus \{x,y,z\}$ dominates two of $\{x,y,z\}$.

{\bf{Remark.}} We believe that there are still much space for the improvement of the upper bounds for $\beta_8,\beta_9$ and $\beta_{10}$. However, we need further properties of {\teq}-retentive sets to help us filter out more non-{\teq}-retentive tournaments. The drastic increase of the upper bounds for $\beta_9$ and $\beta_{10}$ is due to the fact that vertices which have tri-captains, as well as vertices whose indegree is bounded by 5 in tournaments of sizes $9$ and $10$ become relatively rare. Recall that in Lemma~\ref{lem:threeloallyboudnedminimalteqagainst}, every vertex in $R$ has indegree bounded by 5. Hence, tournaments of sizes $9$ and $10$ to which Lemma~\ref{lem:threeloallyboudnedminimalteqagainst} applies are relatively rare. Another major reason is that the number of non-isomorphic tournaments of size $n$ increases exponentially in $n$. 


\section{Tournaments with Two {\teq}-Retentive Sets}
In this section, we study $(H_1,H_2)$-{\teq}-retentive tournaments for small tournaments $H_1$ and $H_2$.
Schwartz's Conjecture is equivalent to saying that there are no $(H_1,H_2)$-{\teq}-retentive tournaments.
In spite of the fact that Schwartz's Conjecture is not true~\cite{DBLP:journals/scw/BrandtCKLNSST13,Brandt2013b}, researchers believe that tournaments that violate Schwartz's Conjecture are rare.  In this section, we focus on the following question:

``What structures are forbidden for a tournament to have two distinct minimal {\teq}-retentive sets?"

Before answering the above question, let's first study some useful properties.

\begin{lemma}
\label{lem:distinctminiteqs}
If a tournament $T$ has two distinct minimal {\teq}-retentive sets $R_1$ and $R_2$, then $R_1\cap R_2=\emptyset$.
\end{lemma}
\begin{proof}
We prove the lemma by contradiction. Assume that there is a tournament $T$ which has two distinct minimal {\teq}-retentive sets $R_1$ and $R_2$ such that $R_1\cap R_2=A\neq \emptyset$. Let $v\in A$ be any vertex in $A$. Since both $R_1$ and $R_2$ are minimal {\teq}-retentive sets of $T$, it holds that  $\mathteq{(T[N^-_T(v)])}\subseteq R_1$ and $\mathteq{(T[N^-_T(v)])}\subseteq R_2$; hence, $\mathteq{(T[N^-_T(v)])}\subseteq (R_1\cap R_2)=A$. This implies that $A$ is also a {\teq}-retentive set of $T$. Since $R_1$ and $R_2$ are distinct, $A$ is a proper subset of at least one of $R_1$ and $R_2$. However, this contradicts the minimality of $R_1$ and $R_2$.
\end{proof}

Due to the above lemma, any two distinct minimal {\teq}-retentive sets in the same tournament are disjoint. In the following, when we write ``two  minimal {\teq}-retentive sets" in a tournament, we mean `two disjoint minimal {\teq}-retentive sets".

Due to Lemma~\ref{lem:teqirreducible} and Theorem~\ref{lem:noretentivesetofsizefour}, there are no minimal {\teq}-retentive sets of sizes $2$ and $4$. Therefore, there are no $(T_2\cup T_4,T_k)$-{\teq}-retentive tournaments for every positive integer $k$. Let $T$ be a tournament. If $\{v\}\subseteq V(T)$ is a minimal {\teq}-retentive set of $T$, then due to the definition of {\teq}-retentive sets, it holds that $N^-_T(v)=\emptyset$; that is, $v$ is the source of $T$. This implies that $\{v\}$ is the unique minimal {\teq}-retentive set of $T$. Therefore, there are no $(T_1,T_k)$-{\teq}-retentive tournaments for every positive integer $k$. In the following, we study $(T_3,T_k)$-{\teq}-retentive tournaments. The following lemmas are useful.

\begin{lemma}
\label{lem:novertexdominatestwoadjacentverticesincapgraingraphs}
Let $T$ be a tournament and $R$ be a minimal {\teq}-retentive set of $T$ such that $\capgraph{T[R]}$ is not empty. Let $v\trelation u$ be an arc in $\dcapgraph{T[R]}$. Then, there is no vertex in $T$ that dominates both $v$ and $u$.
\end{lemma}
\begin{proof}
Since $v\trelation u$ in $\dcapgraph{T[R]}$, $v$ is the captain of $u$ in ${T[R]}$. Since $R$ is a minimal {\teq}-retentive set of $T$ and $v,u\in R$, according to Lemma~\ref{lem:inneighboronesouce}, $v$ is the captain of $u$ in $T$. Hence, all vertices which dominate $u$ in $T$ are dominated by $v$.
\end{proof}

Two vertices in a graph (resp. directed graph) are {\it{adjacent}} if there is an edge (resp. an arc) between them.
Lemma~\ref{lem:novertexdominatestwoadjacentverticesincapgraingraphs} is equivalent to saying that no vertex in a tournament $T$ dominates two adjacent vertices in the domination graph of a subtournament induced by a minimal {\teq}-retentive set of $T$.

A {\it{broom}} of order $k$ is a directed graph consisting of $k+2$ vertices and $k+1$ arcs such that there is a vertex $v$ which has exactly one inneighbor $u$ and exactly $k$ outneighbors $w_1,w_2,...,w_k$ in the directed graph. We call $u$ the {\it{head}}, $v$ the {\it{center}} and all $w_i$s the {\it{leaves}} of the broom. Moreover, we use $\broom{u}{v}{w_1,w_2,...,w_k}$ to denote this broom.

\begin{figure}
\begin{center}
\includegraphics[width=\textwidth]{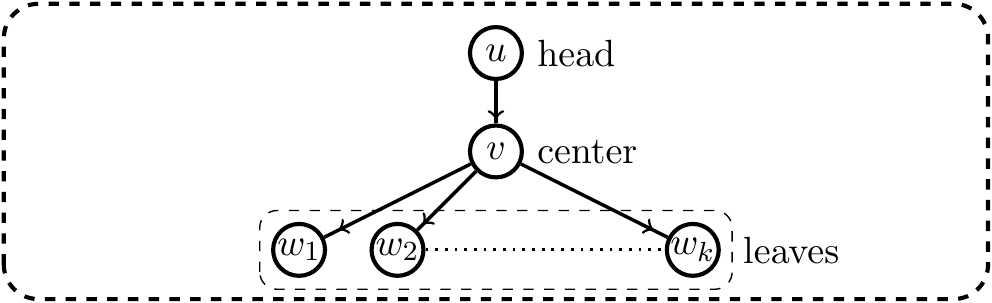}
\end{center}
\caption{A broom $\broom{u}{v}{w_1,...,w_k}$ of size $k$.}
\label{fig:broom}
\end{figure}



\begin{lemma}
\label{lem:teqtrianglecyclebroom}
Let $T$ be a tournament. If $T$ has a minimal {\teq}-retentive set $R$ of size $3$, it cannot have another minimal {\teq}-retentive set $R'$ which satisfies at least one of the following conditions.

(1) $\exists v\in R'$ and a minimal {\teq}-retentive set $S$ of $T[N^-_T(v)]$ such that there is a directed cycle in $\dcapgraph{T[S]}$; or

(2) $\exists v\in R'$ and a minimal {\teq}-retentive set $S$ of $T[N^-_T(v)]$ such that there is a broom $\broom{u}{v'}{w_1,w_2}$ in $\dcapgraph{T[S]}$ such that $w_1,w_2$ together with a third vertex $w'\in N^-_T(v)$ form a minimal {\teq}-retentive set of ${T[N^-_{T[N^-_T(v)]}(u)]}$.
\end{lemma}
\begin{proof}
We prove the lemma by contradiction. \myfig{fig:teqtrianglecyclebroom} is helpful for the readers to follow the proof. Assume that the lemma is not true. Let $R=\{a,b,c\}$ be a minimal {\teq}-retentive set of $T$. We first consider {\condition{1}}. Assume for the sake of contradiction that $T$ has another minimal {\teq}-retentive set $R'$, where there is a vertex $v\in R'$ and a minimal {\teq}-retentive set $S$ of $T[N_T^-(v)]$ such that there is a directed cycle $C=(u_0, u_1,..., u_{t-1})$ in $\dcapgraph{T[S]}$. According to Lemma~\ref{lem:dominationgraphproperty}, $t$ is odd.  Since $\{a,b,c\}$ is a  minimal {\teq}-retentive set of $T$, according to Lemma~\ref{lem:t3minimalretentiveset}, at least two of $\{a,b,c\}$ dominate $v$ in $T$. Without loss of generality, assume that $\{a,b\}\trelation \{v\}$. Then, since $S$ is a minimal {\teq}-retentive set of $T[N_T^-(v)]$, according to Lemma~\ref{lem:novertexdominatestwoadjacentverticesincapgraingraphs}, none of $\{a,b\}\subset N^-_T(v)$ dominates two adjacent vertices in $C$. Since $t$ is odd, this implies that for every $x\in \{a,b\}$, there are two adjacent vertices in $C$ that both dominate $x$. Let $u_i, u_{\mymod{(i+1)}{t}}$ be two adjacent vertices in $C$ that dominate $a$. 
Since $\{a,b,c\}$ is a minimal {\teq}-retentive set of $T$, according to~Lemma~\ref{lem:t3minimalretentiveset}, no vertex in $T$ dominates two of $\{a,b,c\}$. This implies that
$\{b\}\trelation \{u_i,u_{i+1}\}$. 
However, as discussed above, according to Lemma~\ref{lem:novertexdominatestwoadjacentverticesincapgraingraphs}, no vertex in $N^-_T(v)$, including $b$, dominates two adjacent vertices in $C$; a contradiction.

Now let's consider {\condition{2}}. Assume for the sake of contradiction that $T$ has another minimal {\teq}-retentive set $R'$ where there is a vertex $v$ and a minimal {\teq}-retentive set $S$ of $T[N^-_T(v)]$ as stated in {\condition{2}} in the lemma.
Similar to the argument for {\condition{1}},
we can infer that at least two of $\{a,b,c\}$ dominate $v$. Without loss of generality, assume that $\{a,b\}\trelation \{v\}$ in $T$. According to Lemma~\ref{lem:t3minimalretentiveset}, no vertex in $T$ dominates two of $\{a,b,c\}$; and thus, at least one of $\{a,b\}$ dominates $v'$ in $T$. If $a\trelation v'$, then, since $S$ is a minimal {\teq}-retentive set of $N^-_T(v)$ and $\{a,v',w_1,w_2\}\subset N^-_T(v)$, according to Lemma~\ref{lem:novertexdominatestwoadjacentverticesincapgraingraphs}, $\{w_1,w_2\}\trelation \{a\}$. Moreover, $u\trelation a$.
Then, according to Lemma~\ref{lem:t3minimalretentiveset}, it holds that $b\trelation u$.
Since $\{w_1,w_2,w'\}$ is a minimal {\teq}-retentive set of $T[N^-_{T[N^-_T(v)]}(u)]$ and $b\in N^-_{T[N^-_T(v)]}(u)$, according to Lemma~\ref{lem:t3minimalretentiveset}, at least one of $\{w_1, w_2\}$ dominates $b$. However, since $\{a,b,c\}$ is a minimal {\teq}-retentive set of $T$, according to  Lemma~\ref{lem:t3minimalretentiveset}, no vertex in $T$ dominates two of $\{a,b,c\}$ (but due to the above analysis, at least one of $\{w_1,w_2\}$ dominates both $a$ and $b$); a contradiction. 
The proof for the case $b\trelation v'$ can be obtained from the above proof for the case $a\trelation v'$ by exchanging all occurrences of ``$a$" and ``$b$".
\begin{figure}[h]
\begin{center}
\includegraphics[width=\textwidth]{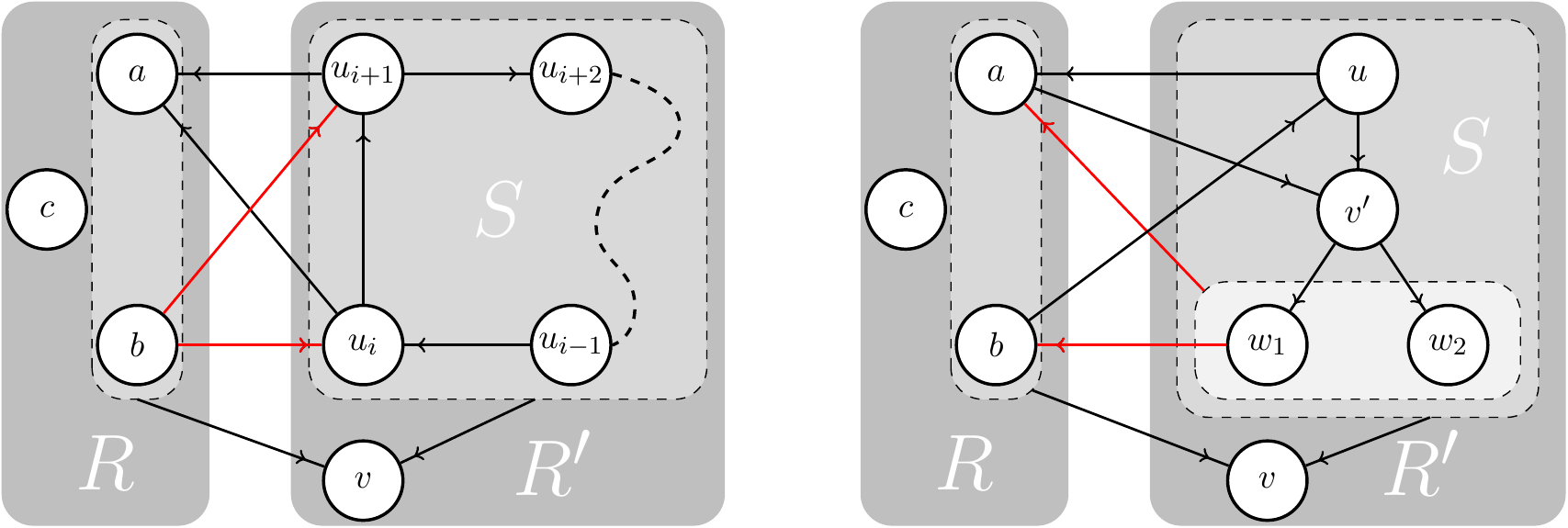}
\end{center}
\caption{This figure illustrates the proof of Lemma~\ref{lem:teqtrianglecyclebroom}. To keep the picture
clean, not all vertices and arcs of $T[R'\cup R]$ are drawn. The figure on the left side illustrates {\condition{1}} and the figure on the right side illustrates {\condition{2}}. In both figures, $R=\{a,b,c\}$ is a minimal {\teq}-retentive set of $T$. Moreover, $R'$ is assumed to be another minimal {\teq}-retentive set of $T$, and $S$ is assumed to be a minimal {\teq}-retentive set of $T[N^-_T(v)]$. Furthermore, the circle drawn in the area labeled with $S$ in the left figure, and the broom $\broom{u}{v'}{w_1,w_2}$ in the right figure are parts of the domination graph of $T[N^-_T(v)]$. Assuming $\{a,b\}\trelation \{v\}$, we arrive at two red arcs in each figure which violate Lemma~\ref{lem:t3minimalretentiveset}.}
\label{fig:teqtrianglecyclebroom}
\end{figure}
\end{proof}

\noindent

Our first main result in this section is summarized as follows.

\begin{theorem}
\label{thm:notthreekleqtwelveteqtournaments}
There are no $(T_3, T_k)$-{\teq}-retentive tournaments for every positive integer $k\leq 12$.
\end{theorem}
\begin{proof}
The main idea of the proof is as follows. Assume that there is a tournament $T$ with two minimal {\teq}-retentive sets $R_1$ and $R_2$ such that $|R_1|=3$ and $|R_2|\leq 12$. Observe that there is a vertex $v\in R_2$ which has at most 5 inneighbors in $R_2$. Hence, every minimal {\teq}-retentive set of $T[N^-_T(v)]$ is of size at most $5$. Then, we show all possibilities of $\mathteq{(T[N^-_{T}(v)])}$ lead to some contradictions, based on some properties of minimal {\teq}-retentive sets of size at most $5$.
The formal proof is as follows.

Assume for the sake of contradiction that there is a $(T_3, T_k)$-{\teq}-retentive tournament $T$ for some integer $1\leq k\leq 12$. Let $R_1=\{a,b,c\}$ be a minimal {\teq}-retentive set of $T$, and let $R_2$ be another minimal {\teq}-retentive set of $T$ of size $k$. Observe that every tournament with at most $12$ vertices has a vertex of indegree at most 5. Let $v$ be a vertex in $T[R_2]$ with minimum indegree. Thus, $|N^-_{T[R_2]}(v)|\leq 5$. According to Lemma~\ref{lem:t3minimalretentiveset}, at least two of $\{a,b,c\}$ dominate $v$. Without loss of generality, assume that $a$ and $b$ dominate $v$. Since $R_2$ is a minimal {\teq}-retentive set of $T$ and $v\in R_2$, it holds that $\mathteq{(T[N^-_T(v)])}\subseteq N^-_{T[R_2]}(v)$. Therefore, every minimal {\teq}-retentive set of $T[N^-_T(v)]$ is of size at most $5$. Observe that it cannot be that $N^-_{T[R_2]}(v)=\emptyset$ (If $N^-_{T[R_2]}(v)=\emptyset$, then due to Lemma~\ref{lem:teqirreducible}, it must be that $R_2=\{v\}$. Then due to Lemma~\ref{lem:teqsourceuniqueminial}, $v$ is the source of $T$. Moreover, $\{v\}$ is the unique minimal {\teq}-retentive set of $T$; contradicting with our assumption that $R_1$ and $R_2$ are two distinct minimal {\teq}-retentive sets of $T$). Let $R(v)$ be a minimal {\teq}-retentive set of $T[N^-_T(v)]$.
We distinguish between several cases with respect to the size of $R(v)$.  According to Lemma~\ref{lem:teqirreducible} and Theorem~\ref{lem:noretentivesetofsizefour}, there are no minimal {\teq}-retentive sets of sizes $2$ and $4$. It remains to consider the following cases.
\caseinproofskip

Case~1. $|R(v)|=1$.

Let $R(v)=\{u\}$. According to Lemma~\ref{lem:captainistheuniqueteqofslavery}, $u$ is the captain of $v$ in $T$.
Therefore, the edge $\edge{u}{v}$ exists in $\capgraph{T}$. Since $\{a,b,c\}$ is a minimal {\teq}-retentive set of size $3$, according to the proof of Lemma~\ref{lem:t3minimalretentiveset}, $\{a,b,c\}$ form a cycle in $\capgraph{T}$. However, this contradicts with Lemma~\ref{lem:dominationgraphproperty}.
\caseinproofskip

Case~2. $|R(v)|=3$.

Let $R(v)=\{a',b',c'\}$. Since $\{a',b',c'\}$ is a minimal {\teq}-retentive set of $T[N^-_T(v)]$ and $\{a,b\}\subseteq N^-_T(v)$, according to Lemma~\ref{lem:t3minimalretentiveset}, at least two of $\{a',b',c'\}$ dominate $a$, and moreover, at least two of $\{a',b',c'\}$ dominate $b$. This implies that at least one of $\{a',b',c'\}$ dominates both $a$ and $b$. However, since $R_1=\{a,b,c\}$ is a minimal {\teq}-retentive set of $T$, according to Lemma~\ref{lem:t3minimalretentiveset}, no vertex of $T$ dominates two of $\{a,b,c\}$; a contradiction.
\caseinproofskip

Case~3. $|R(v)|=5$.

According to Theorem~\ref{thm:teqisomorphctot5}, $T[R(v)]$ is either isomorphic to $T_5(1)$ or to $T_5(6)$ (See \myfig{fig:tfive} for $T_5(1)$ or $T_5(6)$, and \myfig{fig:tfiveonesixcaptaingraphs} for the directed {\captaingraph}s of $T_5(1)$ and $T_5(6)$). We distinguish between the following two cases. Observe that $R(v)=N^-_{T[R_2]}(v)$ (since $|N^-_{T[R_2]}(v)|\leq 5$).


\begin{figure}[h!]
\begin{center}
\includegraphics[width=\textwidth]{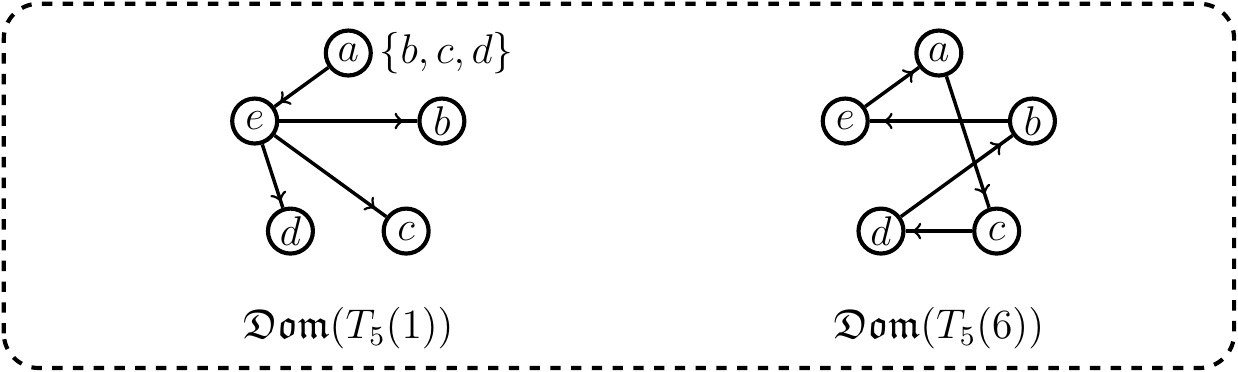}
\end{center}
\caption{Directed {\captaingraph}s of $T_5(1)$ and $T_5(6)$ in \myfig{fig:tfive}. The tri-captain of the vertex $a$ in $T_5(1)$ is $\{b,c,d\}$.}
\label{fig:tfiveonesixcaptaingraphs}
\myvspace{-1pt}
\end{figure}

Subcase 1. $T[R(v)]$ is isomorphic to $T_5(1)$.

For ease of exposition, let $T_5(1)=T[R(v)]$. Moreover, let the vertices in $R(v)$ be labeled as in the tournament $T_5(1)$ in {\myfigs}~\ref{fig:tfive}~and~\ref{fig:tfiveonesixcaptaingraphs}. It is easy to verify that  $\{b,c,d\}=\mathteq{(T[N^-_{T[R(v)]}(a)])}$. Then, since the vertex $a$ has no more than five inneighbors in $R(v)$, and $R(v)$ is a minimal {\teq}-retentive set of $T[N^-_{T}(v)]$, according to Lemma~\ref{lem:localboundedsmallteqtournament}, it holds that $\{b,c,d\}=\mathteq{(T[N^-_{H}(a)])}$, where $H=T[N^-_{T}(v)]$.
Observe that there is a broom $\broom{a}{e}{d,c}$ in $\dcapgraph{T_5(1)}=\dcapgraph{T[R(v)]}$. 
Then, according to Lemma~\ref{lem:teqtrianglecyclebroom} ({\condition{2}}), $R_2$ cannot be a minimal {\teq}-retentive set of $T$; a contradiction.

Subcase 2. $T[R(v)]$ is isomorphic to $T_5(6)$.

For ease of exposition, let $T_5(6)=T[R(v)]$. Moreover, let the vertices in $R(v)$ be labeled as in the tournament $T_5(6)$ in {\myfigs}~\ref{fig:tfive}~and~\ref{fig:tfiveonesixcaptaingraphs}. There is a directed cycle $(a,c,d,b,e)$ in $\dcapgraph{T_5(6)}=\dcapgraph{T[R(v)]}$. 
Then, according to Lemma~\ref{lem:teqtrianglecyclebroom} ({\condition{1}}), $R_2$ cannot be a minimal {\teq}-retentive set of $T$; a contradiction.
\end{proof}

We further prove that there are no $(T_5, T_k)$-{\teq}-retentive tournaments for every positive integer $k\leq 10$. Let's first study some useful properties.

\begin{lemma}
\label{lem:captaingraphagainsttwoteq}
Let $T$ be a tournament, and $R_1$ and $R_2$ be two vertex subsets of $T$ such that $R_1\cap R_2=\emptyset$. If the following conditions hold, $R_1$ and $R_2$ cannot be minimal {\teq}-retentive sets of $T$ simultaneously.

(1) ${\dcapgraph{T[R_1]}}$ is not empty; and

(2) ${\dcapgraph{T[R_2]}}$ has a directed cycle.
\end{lemma}

\begin{proof}
Assume for the sake of contradiction that the above two conditions hold, and both $R_1$ and $R_2$ are minimal {\teq}-retentive sets of $T$. Let $C=(v_0,v_1,...,v_{k-1})$ be a directed cycle in ${\dcapgraph{T[R_2]}}$. Moreover, let $u\trelation v$ be any arbitrary arc in  ${\dcapgraph{T[R_1]}}$. Since  ${\dcapgraph{T[R_1]}}$ is not empty, such an arc exists. Therefore, for all $i=0,1,...,k-1$, $v_i$ is the captain of $v_{\mymod{(i+1)}{k}}$ in $T[R_2]$. Moreover, $u$ is the captain of $v$ in $T[R_1]$. Then since $R_1$ and $R_2$ are minimal {\teq}-retentive sets of $T$, according to Lemma~\ref{lem:inneighboronesouce}, in the tournament $T$, $v_i$ is still the captain of $v_{\mymod{(i+1)}{k}}$ for every $i=0,1,...,k-1$, and $u$ is still the captain of $v$. This implies that both the directed circle $C$ and the arc $u\trelation v$ exist in $\dcapgraph{T}$. However, since $R_1\cap R_2=\emptyset$, this contradicts with Lemma~\ref{lem:dominationgraphproperty}.
\end{proof}


\begin{lemma}
\label{lem:notwoteqonecaptaincycleonetricaptain}
Let $T$ be a tournament, and $R_1$ and $R_2$ be two vertex subsets of $T$ such that $R_1\cap R_2=\emptyset$. If the following conditions hold, $R_1$ and $R_2$ cannot be minimal {\teq}-retentive sets of $T$ simultaneously.

(1) there is a directed cycle in $\dcapgraph{T[R_1]}$; and

(2) there are four vertices $v,a,b,c,\in R_2$ 
such that $\{a,b,c\}=\mathteq{(T[N^-_T(v)])}$.
\end{lemma}

\begin{proof}
Assume for the sake of contradiction that the above two conditions hold, and both $R_1$ and $R_2$ are minimal {\teq}-retentive sets of $T$.
Let $C=(u_0,u_1,...,u_{k-1})$ be a directed cycle in $\dcapgraph{T[R_1]}$.
According to Lemma~\ref{lem:dominationgraphproperty}, $k$ is odd. Furthermore, since $R_1$ is a minimal {\teq}-retentive set of $T$, according to Lemma~\ref{lem:novertexdominatestwoadjacentverticesincapgraingraphs}, we know that $v$ does not dominate two adjacent vertices in the cycle $C$. Since $k$ is odd, this implies that there are two adjacent vertices in the cycle $C$ which dominate the vertex $v$. Without loss of generality, assume that $\{u_{i},u_{j}\}\trelation \{v\}$ for some $i\in \{0,1,...,k-1\}$, where $j\equiv \mymod{(i+1)}{k}$.
Since $\{a,b,c\}=\mathteq{(T[N^-_T(v)])}$ and $u_i, u_{j}\in N^-_T(v)$, according to Lemma~\ref{lem:t3minimalretentiveset}, at least two of $\{a,b,c\}$ dominate $u_i$, and also, at least two of $\{a,b,c\}$ dominate $u_{j}$. This implies that at least one of $\{a,b,c\}$ dominates both $u_i$ and $u_{j}$. However, since $R_1$ is a minimal {\teq}-retentive set of $T$, according to Lemma~\ref{lem:novertexdominatestwoadjacentverticesincapgraingraphs}, no vertex in $T$ dominates two adjacent vertices in the cycle $C$; a contradiction. 
\end{proof}

\bigskip
%

\medskip

\begin{lemma}
\label{lem:twoteqonecaptainbroomonetricaptain}
Let $T$ be a tournament, and $R_1$ and $R_2$ be two vertex subsets of $T$ such that $R_1\cap R_2=\emptyset$.
If the following conditions hold, $R_1$ and $R_2$ cannot be minimal {\teq}-retentive sets of $T$ simultaneously.

(1) There is a broom $\broom{u}{v}{w_1,w_2,w_3}$ in $\dcapgraph{T[R_1]}$;

(2) $\mathteq{(T[N^-_T(u)])}=\{w_1,w_2,w_3\}$; and

(3) there is a vertex $\bar{v}\in R_2$ such that $|\mathteq{(T[N^-_T(\bar{v})])}|=3$.
\end{lemma}

\begin{proof}
We prove the lemma by contradiction. \myfig{fig:againsttwoteqonetriangleonebroom} is helpful for the readers to follow the proof.
Assume for the sake of contradiction that the above two conditions hold, and both $R_1$ and $R_2$ are minimal {\teq}-retentive sets of $T$. Let $\mathteq{(T[N^-_T(\bar{v})])}=\{b_1,b_2,b_3\}$. We distinguish between two cases with respect to the direction of the arc between $v$ and $\bar{v}$.

\begin{figure}[h!]
\begin{center}
\includegraphics[width=\textwidth]{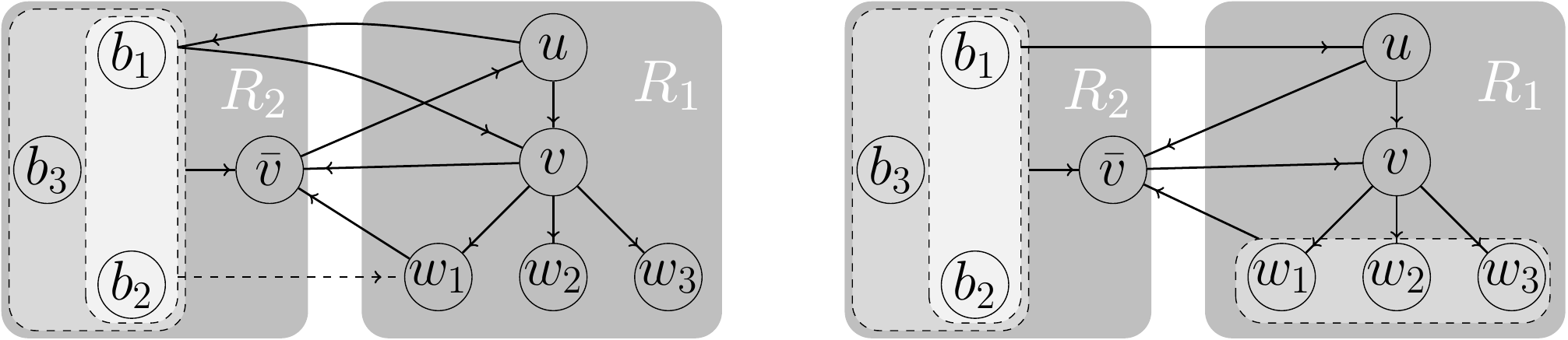}
\end{center}
\caption{An illustration of the proof of Lemma~\ref{lem:twoteqonecaptainbroomonetricaptain}. The figure on the left side illustrates Case~1 and the figure on the right side illustrates Case~2. The dashed arc to $w_1$ in the left figure means that at least one of $\{b_1,b_2\}$ dominates $w_1$. In both figures, $\{b_1,b_2,b_3\}=\mathteq{(T[N^-_T(\bar{v})])}$, $\broom{u}{v}{w_1,w_2,w_3}$ is a broom in $\dcapgraph{T[R_1]}$, and $\{w_1,w_2,w_3\}=\mathteq{(T[N^-_T(u)])}$. To make the figures clear, only arcs and vertices in $T[R_1\cup R_2]$ that are relevant in the proof are drawn. Moreover, we specify the variables $i,j,k$ in the proof here as $i=k=1$ and $j=2$. 
}
\label{fig:againsttwoteqonetriangleonebroom}
\end{figure}
\caseinproofskip

Case~1. $v\trelation \bar{v}$.

Since $\{b_1,b_2,b_3\}=\mathteq{(T[N^-_T(\bar{v})])}$, according to Lemma~\ref{lem:tauthreeequalretentive} and Lemma~\ref{lem:t3minimalretentiveset}, at least two of $\{b_1,b_2,b_3\}$ dominate $v$.
Let's say $\{b_i,b_j\}\trelation \{v\}$ for some $\{i,j\}\subset \{1,2,3\}$.
Since $u$ is the captain of $v$ in $T[R_1]$ (implied by {\condition{1}}), and $R_1$ is a minimal {\teq}-retentive set of $T$, then according to Lemma~\ref{lem:inneighboronesouce}, $u$ is still the captain of $v$ in $T$. This implies that $\{u\}\trelation \{b_i,b_j\}$. Then, it holds that $\bar{v}\trelation u$, since otherwise, $u$ is an inneighbor of $\bar{v}$, and according to Lemma~\ref{lem:tauthreeequalretentive} and Lemma~\ref{lem:t3minimalretentiveset}, at least two of $\{b_1,b_2,b_3\}$ should have dominated $u$. Since $\{w_1,w_2,w_3\}=\mathteq{(T[N^-_T(u)])}$, according to Lemma~\ref{lem:tauthreeequalretentive} and Lemma~\ref{lem:t3minimalretentiveset}, at least one (in fact at least two, but ``one" is enough for our proof) of $\{w_1,w_2,w_3\}$ dominates $\bar{v}$. Let's say $w_k\trelation \{\bar{v}\}$ for some $k\in \{1,2,3\}$. Since $\{b_1,b_2,b_3\}=\mathteq{(T[N^-_T(\bar{v})])}$, according to Lemma~\ref{lem:tauthreeequalretentive} and Lemma~\ref{lem:t3minimalretentiveset}, at least two of $\{b_1,b_2,b_3\}$ dominate $w_k$. Hence, at least one of $\{b_i,b_j\}$ dominates $w_k$. Since $\{b_i,b_j\}\trelation \{v\}$, at least one of $\{b_i,b_j\}$ dominates both $v$ and $w_k$.
However, according to Lemma~\ref{lem:novertexdominatestwoadjacentverticesincapgraingraphs}, no vertex dominates two adjacent vertices in $\dcapgraph{T[R_1]}$; a contradiction.

\caseinproofskip

Case~2. $\bar{v}\trelation {v}$.

Since $u$ is the captain of $v$ in $T[R_1]$ and $R_1$ is a minimal {\teq}-retentive set of $T$, then according to Lemma~\ref{lem:inneighboronesouce}, $u$ is still the captain of $v$ in $T$. This implies that $u\trelation \bar{v}$.  Since $\{b_1,b_2,b_3\}=\mathteq{(T[N^-_T(\bar{v})])}$, according to Lemma~\ref{lem:tauthreeequalretentive} and Lemma~\ref{lem:t3minimalretentiveset}, at least two of $\{b_1,b_2,b_3\}$ dominate $u$. Let's say $\{b_i,b_j\}\trelation \{u\}$ for some $\{i,j\}\subset \{1,2,3\}$.  Moreover, according to Lemma~\ref{lem:novertexdominatestwoadjacentverticesincapgraingraphs}, no vertex in $T$ dominates two adjacent vertices in $\capgraph{T[R_1]}$; thus, $\{w_1,w_2,w_3\}\trelation \{\bar{v}\}$. Again, since  $\{b_1,b_2,b_3\}=\mathteq{(T[N^-_T(\bar{v})])}$, according to Lemma~\ref{lem:tauthreeequalretentive} and Lemma~\ref{lem:t3minimalretentiveset}, at least two of $\{b_1,b_2,b_3\}$ dominate $w_1, w_2$ and $w_3$, respectively. Therefore, there are at least 6 arcs from $\{b_1,b_2,b_3\}$ to $\{w_1, w_2, w_3\}$. On the other hand, since $\{b_i,b_j\}\trelation \{u\}$, and $\{w_1,w_2,w_3\}=\mathteq{(T[N^-_T(u)])}$, according to Lemma~\ref{lem:tauthreeequalretentive} and Lemma~\ref{lem:t3minimalretentiveset} at least two of $\{w_1, w_2, w_3\}$ dominate $b_i$ and $b_j$, respectively. Therefore, there are at least $4$ arcs from $\{w_1, w_2, w_3\}$ to $\{b_1,b_2,b_3\}$. This leads to at least 10 arcs between $\{b_1,b_2,b_3\}$ and $\{w_1,w_2,w_3\}$ in total. However, there are 9 arcs between them; a contradiction.
\end{proof}

For vertices $a,b,c,d$ (it may that two of them are identical) in a tournament, $[a,b]\asymp [c,d]$ means that there is an arc $a\trelation b$ if and only if there is an arc $c\trelation d$. On the other hand, $[a,b]\not\asymp [c,d]$ means that there is an arc $a\trelation b$ if and only if there is an arc $d\trelation c$ in the tournament. Clearly, if $[a,b]\asymp [c,d]$ and $[c,d]\asymp [e,f]$, then $[a,b]\asymp [e,f]$. Moreover, if $[a,b]\asymp [c,d]$ and $[c,d]\not\asymp [e,f]$, then $[a,b]\not\asymp [e,f]$. Furthermore, if $[a,b]\not\asymp [c,d]$ and $[c,d]\not\asymp [e,f]$, then $[a,b]\asymp [e,f]$.

\begin{lemma}
\label{lem:teqisomorhpictoT43}
Let $T$ be a tournament. If $T$ has two distinct minimal {\teq}-retentive sets $R_1$ and $R_2$ such that

(1) there is an arc $u\trelation v$ in $\dcapgraph{T[R_1]}$; and

(2) there is an arc $z\trelation w$ in $\dcapgraph{T[R_2]}$,

\noindent then, 
$[u,w]\asymp [v,z], [u,z]\asymp [v,w]$ and $[u,w]\not\asymp [u,z]$.
\end{lemma}

\begin{proof}
We distinguish between two cases, with respect to the arc between $u$ and $z$ in the tournament $T$. See \myfig{fig:isomorphic43} for an illustration.



\begin{figure}[h!]
\begin{center}
\includegraphics[width=\textwidth]{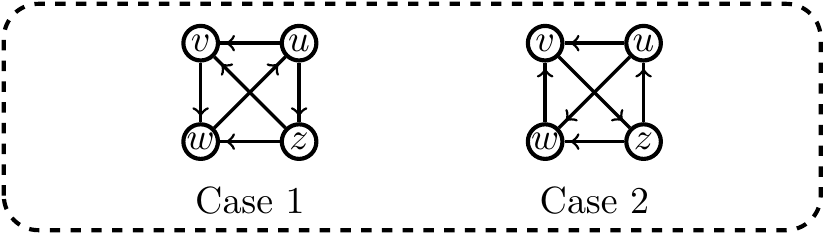}
\end{center}
\caption{This figure illustrates the two cases in the proof of Lemma~\ref{lem:teqisomorhpictoT43}.}
\label{fig:isomorphic43}
\end{figure}

\caseinproofskip
Case~1. $u\trelation z$ in $T$.

Since $R_2$ is a minimal {\teq}-retentive set of $T$ and $z\trelation w$ in $\dcapgraph{T[R_2]}$, then according to~Lemma~\ref{lem:novertexdominatestwoadjacentverticesincapgraingraphs}, it holds that $w\trelation u$. 
Then, since $R_1$ is a minimal {\teq}-retentive set of $T$ and $u\trelation v$ in $\dcapgraph{T[R_1]}$, according to~lemma~\ref{lem:novertexdominatestwoadjacentverticesincapgraingraphs}, it holds that $v\trelation w$. Again, by  Lemma~\ref{lem:novertexdominatestwoadjacentverticesincapgraingraphs}, we can infer that $z\trelation v$. 
\caseinproofskip

Case~2. $z\trelation u$ in $T$.

Since $R_1$ is a minimal {\teq}-retentive set of $T$ and $u\trelation v$ in $\dcapgraph{T[R_1]}$, according to~Lemma~\ref{lem:novertexdominatestwoadjacentverticesincapgraingraphs}, it holds that $v\trelation z$. Then, since $R_2$ is a minimal {\teq}-retentive set of $T$ and $z\trelation w$ in $\dcapgraph{T[R_2]}$, according to~lemma~\ref{lem:novertexdominatestwoadjacentverticesincapgraingraphs}, it holds that $w\trelation v$. Again, by lemma~\ref{lem:novertexdominatestwoadjacentverticesincapgraingraphs}, we can infer that $u\trelation w$. 

It is then easy to verify that in both cases, $[u,w]\asymp [v,z], [u,z]\asymp [v,w]$ and $[u,w]\not\asymp [u,z]$.
\end{proof}

\begin{lemma}
\label{twoteqonecaptainbroomonenonemptycaptain}
Let $T$ be a tournament, and $R_1$ and $R_2$ be two vertex subsets of $T$ such that $R_1\cap R_2=\emptyset$. If the following conditions hold, $R_1$ and $R_2$ cannot be minimal {\teq}-retentive sets of $T$ simultaneously.

(1) there is a broom $\broom{u}{v}{w_1,w_2}$ in $\dcapgraph{T[R_1]}$;

(2) $\mathteq{(T[N^-_T(u)])}=\{w_{1},w_{2},w\}$, where $w\in R_1\setminus \{w_1,w_2\}$; and 

(3) $\dcapgraph{T[R_2]}$ is not empty.
\end{lemma}
\begin{proof}
We prove the lemma by contradiction. Assume that the lemma is not true.
Let $\bar{v}\trelation \bar{u}$ be any arc in $\dcapgraph{T[R_2]}$ (since $\dcapgraph{T[R_2]}$ is not empty, such an arc exists).
We distinguish between two cases with respect to the arc between $\bar{v}$ and $v$. \myfig{fig:conflicttwoteqonebroomonenonempty} is helpful for the readers to follow the proof.

\begin{figure}
\begin{center}
\includegraphics[width=\textwidth]{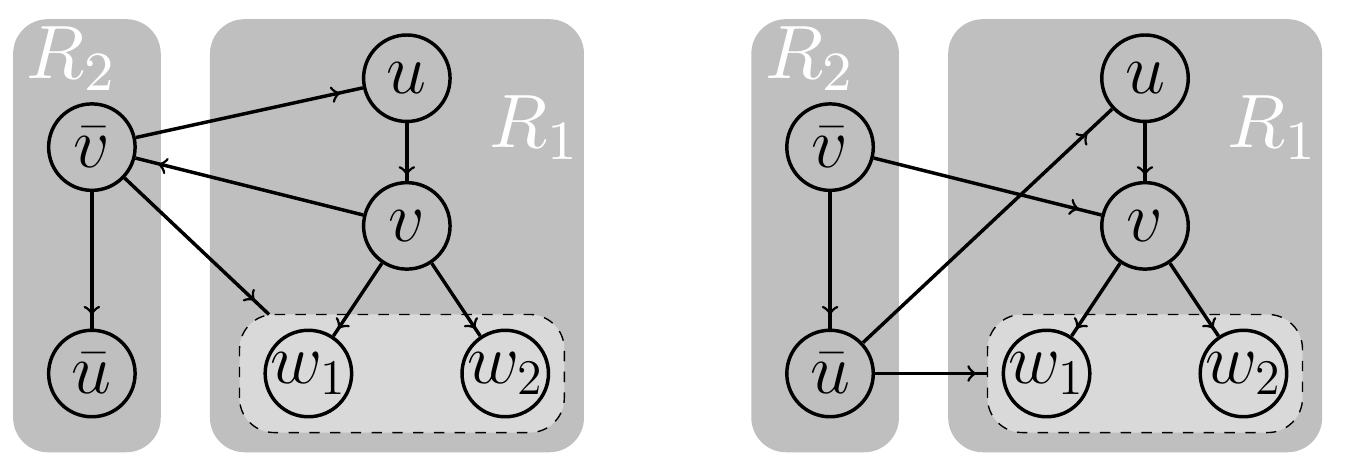}
\end{center}
\caption{An illustration of the proof of Lemma~\ref{twoteqonecaptainbroomonenonemptycaptain}. The figure on the left side illustrates Case~1 and the figure on the right side illustrates Case~2. In both figures, $\broom{u}{v}{w_1,w_2}$ is a broom in $T[R_1]$. To make the figures clear, not all vertices and arcs in $T[R_1\cup R_2]$ are drawn.}
\label{fig:conflicttwoteqonebroomonenonempty}
\end{figure}
\caseinproofskip

Case~1. $v\trelation \bar{v}$.

According to Lemma~\ref{lem:teqisomorhpictoT43}, it holds that $\{\bar{v}\}\trelation \{u,w_1, w_2\}$. However, this contradicts with the fact that $\{w_1,w_2,w\}=\mathteq{(T[N^-_T(u)])}$, since according to Lemma~\ref{lem:t3minimalretentiveset} and Lemma~\ref{lem:tauthreeequalretentive}, at least two of $\{w_1,w_2, w\}$ should have dominated $\bar{v}$.
\caseinproofskip

Case~2. $\bar{v}\trelation {v}$.

According to Lemma~\ref{lem:teqisomorhpictoT43}, it holds $\{\bar{u}\}\trelation \{u,w_1,w_2\}$. 
However, since $\{w_1,w_2,w\}=\mathteq{[T(N^-_T(u))]}$, according to Lemma~\ref{lem:t3minimalretentiveset} and Lemma~\ref{lem:tauthreeequalretentive}, at least two of $\{w_1,w_2,w\}$ should have dominated $\bar{u}$; a contradiction.
\end{proof}

We are ready to show our main result concerning $(T_5, T_k)$-{\teq}-retentive tournaments for small values of $k$.

\begin{theorem}
\label{thm:notfivekleqtenteqtournaments}
There are no $(T_5, T_k)$-{\teq}-retentive tournaments for every positive integer $k\leq 10$.
\end{theorem}
\begin{proof}
We prove this theorem by contradiction. Assume that there is a $(T_5, T_k)$-{\teq}-retentive tournament $T$ for some positive integer $k\leq 10$. Let $R_1$ and $R_2$ be two minimal {\teq}-retentive sets of $T$ such that $T[R_1]$ is isomorphic to some tournament in $T_5$, and $T[R_2]$ is isomorphic to some tournament in $T_k$. We have showed in Theorem~\ref{thm:teqisomorphctot5} that every {\teq}-retentive tournament of size $5$ is either isomorphic to $T_5(1)$ or to $T_5(6)$ in \myfig{fig:tfive}. Therefore, $T[R_1]$ is either isomorphic to $T_5(1)$ or $T_5(6)$. We distinguish between these two cases. The directed {\captaingraph} of $T_5(1)$ and $T_5(6)$ are shown in \myfig{fig:tfiveonesixcaptaingraphs}.
Observe that each tournament with at most 10 vertices has a vertex of indegree at most 4. In the following, let $v$ be a vertex of indegree at most 4 in $T[R_2]$, that is, $|N^-_{T[R_2]}(v)|\leq 4$.
Further observe that it cannot be that $N^-_{T[R_2]}(v)=\emptyset$ (If $N^-_{T[R_2]}(v)=\emptyset$, then due to Lemma~\ref{lem:teqirreducible}, it must be that $R_2=\{v\}$. Then due to Lemma~\ref{lem:teqsourceuniqueminial}, $v$ is the source of $T$. Moreover, $\{v\}$ is the unique minimal {\teq}-retentive set of $T$; contradicting with our assumption that $R_1$ and $R_2$ are two distinct minimal {\teq}-retentive sets of $T$). Hence, in the following, we assume that $N^-_{T[R_2]}(v)\neq \emptyset$; thus, $\mathteq{(T[N^-_T(v)])}\neq \emptyset$.

\caseinproofskip
Case~1. $T[R_1]$ is isomorphic to $T_5(1)$.

For ease of exposition, let $T_5(1)=T[R_1]$. Moreover, let the vertices in $R_1$ be labeled as in the tournament $T_5(1)$ in {\myfigs}~\ref{fig:tfive} and~\ref{fig:tfiveonesixcaptaingraphs}. Observe that there is a broom $\broom{a}{e}{b,c}$ in $\dcapgraph{T[R_1]}$ (see \myfig{fig:tfiveonesixcaptaingraphs}). Moreover, $\{b,c,d\}$ is a tri-captain of $a$ in $T[R_1]$. Since $R_1$ is a minimal {\teq}-retentive set of $T$, according to Lemma~\ref{lem:localboundedsmallteqtournament}, it holds that $\{b,c,d\}=\mathteq{(T[N^-_T(a)])}$. Now let's take a look at $T[R_2]$.
Since $R_2$ is a minimal {\teq}-retentive set of $T$ and $v\in R_2$, we have that $\mathteq{(T[N^-_T(v)])}\subseteq N^-_{T[R_2]}(v)$. According to Lemma~\ref{lem:teqirreducible} and Theorem~\ref{lem:noretentivesetofsizefour}, there are no minimal {\teq}-retentive sets of sizes $2$ and $4$. This implies that $|\mathteq{(T[N^-_T(v)])}|\neq 2$ and $|\mathteq{(T[N^-_T(v)])}|\neq 4$. Therefore, either $|\mathteq{(T[N^-_T(v)])}|=1$ or $|\mathteq{(T[N^-_T(v)])}|=3$. In the former case, due to Lemma~\ref{lem:teqsourceuniqueminial}, there is a vertex $u$ (a source in $T[N^-_{T[{R_2}]}(v)]$) in $N^-_{T[R_2]}(v)$ which dominates every other vertex in $N^-_{T[R_2]}(v)$. Hence $u$ is the captain of $v$ in $T[R_2]$; implying that $\dcapgraph{T[R_2]}$ is not empty. Now it is easy to see that $R_1$ and $R_2$ satisfy all conditions in Lemma~\ref{twoteqonecaptainbroomonenonemptycaptain}, which in turn implies that $R_1$ and $R_2$ cannot be both minimal {\teq}-retentive sets of $T$; a contradiction. In the latter case, $R_1$ and $R_2$ satisfy the conditions in Lemma~\ref{lem:twoteqonecaptainbroomonetricaptain} (notice that $\broom{a}{e}{b,c,d}$ is a broom in $\dcapgraph{T[R_1]}$), which in turn implies that $R_1$ and $R_2$ cannot be both minimal {\teq}-retentive sets of $T$; a contradiction either.
\caseinproofskip

Case 2. $T[R_1]$ is isomorphic to $T_5(6)$.

The proof for this case is similar to that for Case~1. First, we distinguish between  $|\mathteq{(T[N^-_T(v)])}|=1$ and $|\mathteq{(T[N^-_T(v)])}|=3$, as in the above proof. If $|\mathteq{(T[N^-_T(v)])}|=1$, then we can still arrive at the conclusion that $\dcapgraph{T[R_1]}$ is not empty, as discussed in the above proof. Then, since there is a directed cycle in $\dcapgraph{T_5(6)}$ (see the directed cycle $(a,c,d,b,e)$ in the directed domination graph of $T_5(6)$ in  \myfig{fig:tfiveonesixcaptaingraphs}) and $T[R_1]$ is isomorphic to $T_5(6)$, there is a directed cycle in $\dcapgraph{T[R_1]}$. Now, it is clear that $R_1$ and $R_2$ satisfy the conditions in Lemma~\ref{lem:captaingraphagainsttwoteq}, which, however, in turn implies that $R_1$ and $R_2$ cannot be two minimal {\teq}-retentive sets of $T$ simultaneously; a contradiction.
On the other hand, if $|\mathteq{(T[N^-_T(v)])}|=3$, then $R_1$ and $R_2$ will satisfy the conditions of Lemma~\ref{lem:notwoteqonecaptaincycleonetricaptain}, which, however, in turn implies that $R_1$ and $R_2$ cannot be two minimal {\teq}-retentive sets of $T$ simultaneously; a contradiction either. 
\end{proof}

Now we study $(T_6,T_k)$-{\teq}-retentive tournaments. 
The following lemmas are useful for our study.

\begin{figure}
\begin{center}
\begin{minipage}{0.25\textwidth}
\begin{center}
\includegraphics[width=0.8\textwidth]{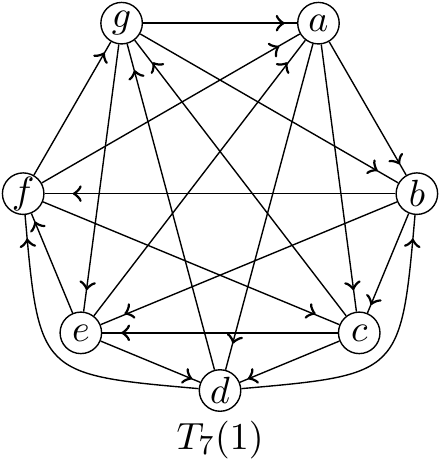}
\end{center}
\end{minipage}\begin{minipage}{0.25\textwidth}
\begin{center}
\begin{tabular}{l|c}
     &   tri-captain \\ \hline

$a$  &   $\{e,f,g\}$ \\ \hline

$b$  &   $\{a,d,g\}$ \\ \hline

$c$  &   $\{a,b,f\}$ \\ \hline

$d$  &   $\{a,c,e\}$ \\ \hline

$e$  &   $\{b,c,g\}$ \\ \hline

$f$  &   $\{b,d,e\}$\\ \hline

$g$  &   $\{c,d,f\}$ \\ \hline
\end{tabular}
\end{center}
\end{minipage}
\caption{All tournaments of size 7 that have non-empty domination graphs are isomorphic to $T_7(1)$. The tri-captain of each vertex is shown in the table on the right side.}
\label{fig:tsevenone}
\end{center}
\end{figure}

\begin{lemma}
\label{lem:againsttwoteqonetsixnineonecaptainnotempty}
Let $T$ be a tournament. Let $R_1$ and $R_2$ be two vertex subsets of $T$ such that $R_1\cap R_2=\emptyset$. Then, $R_1$ and $R_2$ cannot be minimal {\teq}-retentive sets of $T$ simultaneously if the following conditions hold.

(1) $T[R_1]$ is isomorphic to the tournament $T_6(2)$ in \myfig{fig:teqtournamentofsizesix}; and

(2) $\dcapgraph{T[R_2]}$ is not empty.
\end{lemma}
\begin{proof}
Assume for the sake of contradiction that $R_1$ and $R_2$ are two minimal {\teq}-retentive sets of $T$ that satisfy the above two conditions in the lemma. For ease of exposition, assume that $T_6(2)=T[R_1]$. Moreover, let the vertices in $R_1$ be labeled as in the tournament $T_6(2)$ in \myfig{fig:teqtournamentofsizesix}. Let $v\trelation u$ be an arbitrary arc in $\dcapgraph{T[R_2]}$.  We complete the proof by distinguishing between the following two cases. \myfig{fig:againsttwoteqonetsixnineonecaptainnotempty} is helpful for the readers to follow the proof.
\caseinproofskip

Case~1. $u\trelation a$.

It is easy to verify that $\{d,e,f\}$ is a tri-captain of $a$ in $T[R_1]$. Since $R_1$ is a minimal {\teq}-retentive set of $T$, and $a$ has no more than $5$ inneighbors in $R_1$, according to Lemma~\ref{lem:localboundedsmallteqtournament}, it holds that  $\{d,e,f\}=\mathteq{(T[N^-_T(a)])}$. Then, due to Lemma~\ref{lem:tauthreeequalretentive} and Lemma~\ref{lem:t3minimalretentiveset}, at least two of $\{d,e,f\}$ dominate $u$.
Moreover, due to Lemma~\ref{lem:teqisomorhpictoT43}, $[u,e]\not\asymp [u,f]$ and $[u,f]\not\asymp [u,d]$. Therefore, it holds that $e\trelation u$, $u\trelation f$ and $d\trelation u$.
Then, again due to Lemma~\ref{lem:teqisomorhpictoT43}, we have that $u\trelation c$. It is clear that  $\{a,b,f\}$ is a tri-captain of $c$ in $T[R_1]$. Since $R_1$ is a minimal {\teq}-retentive set of $T$, and $c$ has no more than five inneighbors in $R_1$, according to Lemma~\ref{lem:localboundedsmallteqtournament}, $\{a,b,f\}=\mathteq{(T[N^-_{T}(c)])}$. Then, according to Lemma~\ref{lem:tauthreeequalretentive} and Lemma~\ref{lem:t3minimalretentiveset}, at least two of $\{a,b,f\}$ dominate $u$. However, we have showed that $\{u\}\trelation \{a,f\}$; a contradiction.

\caseinproofskip
Case~2. $a\trelation u$.

Due to Lemma~\ref{lem:novertexdominatestwoadjacentverticesincapgraingraphs}, $v\trelation a$. Then, by replacing all occurrences of $u$ with $v$ in the above proof for Case~1, we can prove the lemma for this case.
\begin{figure}[h!]
\begin{center}
\includegraphics[width=\textwidth]{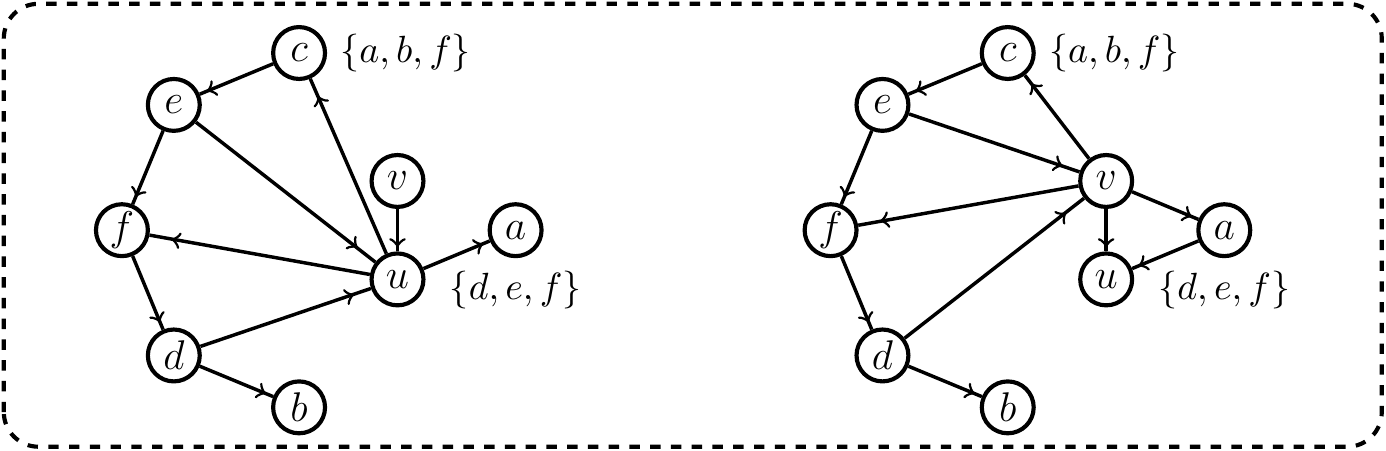}
\end{center}
\caption{An illustration of the proof of Lemma~\ref{lem:againsttwoteqonetsixnineonecaptainnotempty}. The figure on the left hand illustrates Case~1, and the figure on the right side illustrates Case~2. To keep the figures clear, not all arcs are drawn. In both figures, $R_1=\{a,b,c,d,e,f\}$, and $v\trelation u$ is an arc in $\dcapgraph{T[R_2]}$. In each figure, the directed subgraph induced by $R_1$ is $\dcapgraph{T[R_1]}$. A 3-set labeled next to a vertex is the tri-captain of the vertex in $T[R_1]$.}
\label{fig:againsttwoteqonetsixnineonecaptainnotempty}
\end{figure}
\end{proof}

\begin{lemma}
\label{lem:tsixsixtricaptain}
Let $T$ be a tournament, and $R_1$ and $R_2$ be two vertex subsets of $T$ such that $R_1\cap R_2=\emptyset$. If the following conditions hold, $R_1$ and $R_2$ cannot be minimal {\teq}-retentive sets of $T$ simultaneously.

(1) $T[R_1]$ is isomorphic to the tournament $T_6(1)$ in \myfig{fig:teqtournamentofsizesix}; and

(2) there is a vertex $v\in R_2$ such that $|\mathteq{(T[N^-_T(v)])}|=3$
\end{lemma}
\begin{proof}
Assume for the sake of contradiction that $R_1$ and $R_2$ are two minimal {\teq}-retentive sets of $T$ that satisfy the above two conditions in the lemma. For ease of exposition, let $T_6(1)=T[R_1]$.  Moreover, let the vertices in $R_1$ be labeled as in the tournament $T_6(1)$ in \myfig{fig:teqtournamentofsizesix}.
Let $v, x, y, z$ be four vertices in $R_2$ such that $\{x,y,z\}=\mathteq{(T[N^-_T(v)])}$.  We distinguish  between two cases with respect to the arc between the vertex $f$ in $R_1$ and the vertex $v$ in $R_2$. Recall first that if a vertex $a'\in R_1$ is the captain of another vertex $b'\in R_1$ in $T[R_1]$, then $a'$ is also the captain of $b'$ in $T$. This is due to Lemma~\ref{lem:inneighboronesouce} and the assumption that $R_1$ is a minimal {\teq}-retentive set of $T$. 
\myfig{fig:conflicttwoteqonetricaptainonetsixtwo} is helpful for the readers to follow the proof.
\begin{figure}
\begin{center}
\includegraphics[width=\textwidth]{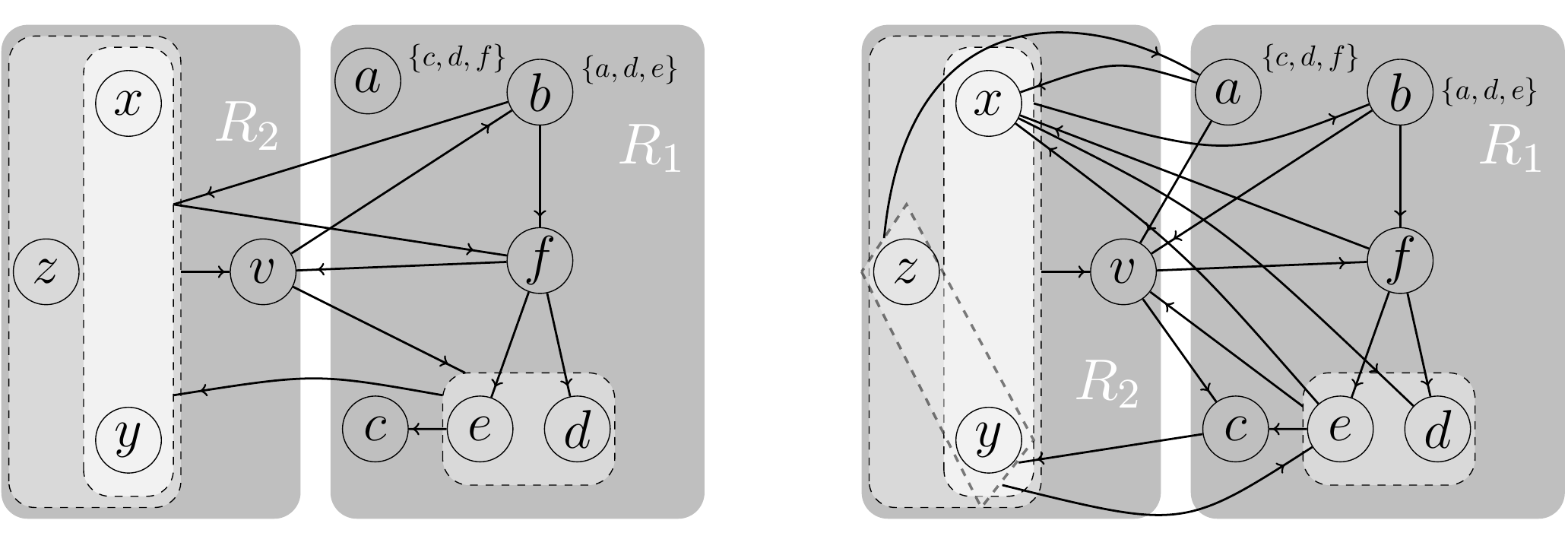}
\end{center}
\myvspace{-15pt}
\caption{An illustration of the proof of Lemma~\ref{lem:tsixsixtricaptain}. The figure on the left side illustrates Case~1 and the figure on the right side illustrates Case~2. To keep the figures clear, not all arcs and vertices in $T[R_1\cup R_2]$ are drawn. In both cases, $\{x,y,z\}=\mathteq{(T[T^-_T(v)])}$. Moreover, the directed graph induced by $\{a,b,c,d,e,f\}$ in each figure is $\dcapgraph{T[R_1]}$. The 3-subset of vertices labeled next to a vertex is the tri-captain of the vertex in $T[R_1]$.}
\label{fig:conflicttwoteqonetricaptainonetsixtwo}
\myvspace{-10pt}
\end{figure}
\caseinproofskip

Case~1. $f\trelation v$.

Since $\{x,y,z\}=\mathteq{(T[N^-_T(v)])}$, according to Lemma~\ref{lem:tauthreeequalretentive} and Lemma~\ref{lem:t3minimalretentiveset}, at least two of $\{x,y,z\}$ dominate $f$. Without loss of generality, assume that $\{x,y\}\trelation \{f\}$. Then, according to Lemma~\ref{lem:novertexdominatestwoadjacentverticesincapgraingraphs},
it holds that  $\{b,e,d\}\trelation \{x,y\}$. This implies that $\{v\}\trelation \{b,e,d\}$, since otherwise, according to Lemma~\ref{lem:tauthreeequalretentive} and Lemma~\ref{lem:t3minimalretentiveset}, for each $i\in \{b,e,d\}$ at least two of $\{x,y,z\}$ should have dominated $i$.
It is easy to check that $\{a,e,d\}$ is a tri-captain of $b$ in $T[R_1]$.
Since $R_1$ is a minimal {\teq}-retentive set of $T$, according to Lemma~\ref{lem:localboundedsmallteqtournament}, $\{a,e,d\}=\mathteq{(T[N^-_{T}(b)])}$. Then, since $v\trelation b$, according to Lemma~\ref{lem:tauthreeequalretentive} and Lemma~\ref{lem:t3minimalretentiveset}, at least two of $\{a,e,d\}$ dominate $v$. However, we have showed above that $\{v\}\trelation \{e,d\}$); a contradiction.
\caseinproofskip

Case~2. $v\trelation f$.

According to Lemma~\ref{lem:novertexdominatestwoadjacentverticesincapgraingraphs}, $\{b,e,d\}\trelation \{v\}$. Then, since $\{x,y,z\}=\mathteq{(T[N^-_T(v)])}$, according to Lemma~\ref{lem:tauthreeequalretentive} and Lemma~\ref{lem:t3minimalretentiveset}, at least two of $\{x,y,z\}$ dominate $b$. Without loss of generality, assume that $\{x,y\}\trelation \{b\}$. With the same reason, at least two of $\{x,y,z\}$ dominate $e$, and moreover, at least two of $\{x,y,z\}$ dominate $d$. These imply that at least one of $\{x,y\}$ dominates $d$. Assume that $x\trelation d$ (the proof applies to the case that $y\trelation d$ if we exchange all occurrences of ``$x$'' and ``$y$'' in the following argument). As showed above in the proof for Case~1, $\{a,e,d\}=\mathteq{(T[N^-_{T}(b)])}$. Then, according to Lemma~\ref{lem:tauthreeequalretentive} and Lemma~\ref{lem:t3minimalretentiveset}, at least two of $\{a,e,d\}$ dominate $x$. Since we have just showed that $x\trelation d$, it must be that $\{a,e\}\trelation \{x\}$. Then, as we have showed above that at least two of $\{x,y,z\}$ dominate $e$, it holds that $\{y,z\}\trelation \{e\}$. Then, according to Lemma~\ref{lem:novertexdominatestwoadjacentverticesincapgraingraphs}, it holds that $\{c\}\trelation \{y,z\}$. This implies that $v\trelation c$, since otherwise, according to Lemma~\ref{lem:tauthreeequalretentive} and Lemma~\ref{lem:t3minimalretentiveset}, at least two of $\{x,y,z\}$ should have dominated $c$. It is easy to verify that $\{c,d,f\}$ is a tri-captain of $a$ in $T[R_1]$. Since $R_1$ is a minimal {\teq}-retentive set of $T$ and $a$ has no more than 5 inneighbors in $R_1$, according to Lemma~\ref{lem:localboundedsmallteqtournament}, it holds that $\{c,d,f\}=\mathteq{(T[N^-_{T}(a)])}$.
Since we have showed that $\{v\}\trelation \{c,f\}$, it must be that $a\trelation v$, since otherwise, according to Lemma~\ref{lem:tauthreeequalretentive} and Lemma~\ref{lem:t3minimalretentiveset} at least two of $\{c,d,f\}$ should have dominated $v$. Then, according to Lemma~\ref{lem:tauthreeequalretentive} and Lemma~\ref{lem:t3minimalretentiveset}, at least two of $\{x,y,z\}$ dominate $a$. Since we have showed above that $a\trelation x$, it must be that $\{y,z\}\trelation \{a\}$. However, since $\{a,d,e\}=\mathteq(T[N^-_T(b)])$ and $y\trelation b$, according to Lemma~\ref{lem:tauthreeequalretentive} and Lemma~\ref{lem:t3minimalretentiveset}, at least two of $\{a,d,e\}$ should have dominated $y$; a contradiction (we have showed above that $\{y\}\trelation \{a,e\}$).
\end{proof}

\begin{figure}
\begin{center}
\includegraphics[width=\textwidth]{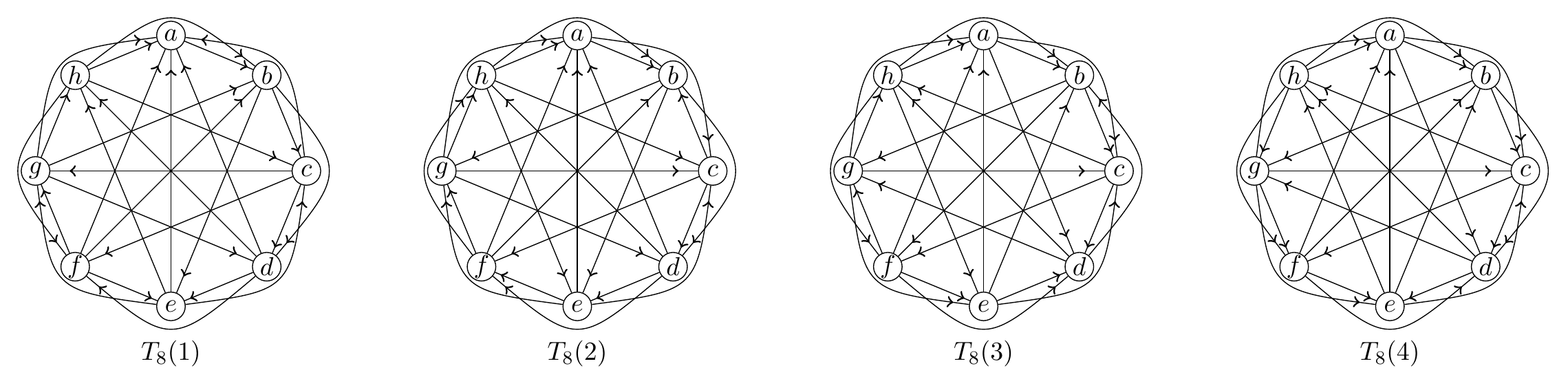}
\end{center}
\myvspace{-18pt}
\caption{All non-isomorphic tournaments of size $8$ whose {\captaingraph}s are not empty.}
\label{fig:teight}
\end{figure}


\begin{table*}
\begin{center}
{{\begin{tabular}{l|c|c|c|c|c|c|c|c}
\hline

    &  $a$ & $b$ & $c$  & $d$  & $e$ & $f$ & $g$ & $h$\\ \hline

$T_8(1)$  & & $\{h,f,g\}$ & $\{e,h,b\}$ & $\{b,c,g\}$ & $\{b,d,f\}$ & $\{c,d,h\}$ & $\{c,e,f\}$ & $\{d,e,g\}$\\ \hline

$T_8(2)$ & $\{d,e,g\}$ & $\{c,f,h\}$&$\{e,g,h\}$&$\{b,g,c\}$&$\{b,d,h\}$&$\{c,d,e\}$&$\{b,e,f\}$&$\{d,f,g\}$\\ \hline

$T_8(3)$&$\{f,g,h\}$&$\{e,d,h\}$&$\{b,e,g\}$&$\{c,f,e\}$&$\{f,g,h\}$&$\{b,c,h\}$&$\{d,b,f\}$&$\{c,d,g\}$\\ \hline

$T_8(4)$&$\{h,g,e\}$&$\{h,f,e\}$&$\{b,g,e\}$&$\{a,c,f\}$&$\{d,g,f\}$&$\{g,c,h\}$&$\{b,d,h\}$&$\{c,d,e\}$\\ \hline
\end{tabular}}}
\end{center}
\myvspace{-15pt}
\caption{A summary of tri-captains of vertices in $T_8(1),T_8(2),T_8(3),T_8(4)$ in \myfig{fig:teight}. The vertex $a$ in $T_8(1)$ has no tri-captain.}
\label{tab:teighttricaptains}
\myvspace{-20pt}
\end{table*}


Let $F=\{Tree_1, Tree_2,...,Tree_k\}$ be a forest. We say two vertices $v$ and $u$ are {\it{\sibling}s} in $F$ if the following two conditions hold:

(1) $v$ and $u$ are in the same tree $Tree_i$ for some $i\in \{1,2,...,k\}$; and

(2) for every vertex $w\in V(Tree_i)\setminus\{v,u\}$, it holds that \[dist(w,v)\equiv dist(w,u) \mod{(2)},\] where $dist(a,b)$ denotes the distance between two vertices $a$ and $b$ in the tree $Tree_i$, and $V(Tree_i)$ is the set of vertices in $Tree_i$.

\begin{lemma}
\label{lem:captaingrapht43property}
Let $T$ be a tournament. Let $R_1$ and $R_2$ be two distinct minimal {\teq}-retentive sets of $T$ such that neither $\capgraph{T[R_1]}$ nor $\capgraph{T[R_2]}$ is empty. Then, both $\capgraph{T[R_1]}$ and $\capgraph{T[R_2]}$ are forests. Moreover, for every two {\sibling}s $v,u\in R_1$ in $\capgraph{T[R_1]}$, and every non-isolated vertex $w\in R_{2}$ in ${\capgraph{T[R_{2}]}}$, it holds that $[w,v]\asymp [w,u]$. Furthermore, for every two vertices $v,u\in R_1$ that are in the same tree in $\capgraph{T[R_1]}$ but are not {\sibling}s in $\capgraph{T[R_1]}$, and for every non-isolated vertex $w\in R_{2}$ in $\capgraph{T[R_{2}]}$, it holds that $[w,v]\not\asymp [w,u]$.
\end{lemma}

\begin{proof}
Let's first prove that both $\capgraph{T[R_1]}$ and $\capgraph{T[R_2]}$ are forests, that is, neither $\capgraph{T[R_1]}$ nor $\capgraph{T[R_2]}$ contains a cycle. Assume for the sack of contradiction that this is not the case. Let $C$ be a cycle in $\capgraph{T[R_i]}$ for some $i\in \{1,2\}$. Since $R_i$ is a minimal {\teq}-retentive set of $T$, according to Lemma~\ref{lem:inneighboronesouce} and the definition of {\captaingraph}s, if there is an edge between two vertices $a$ and $b$ in $\capgraph{T[R_i]}$, this edge still exists in $\capgraph{T}$. Therefore, the cycle $C$ still exits in $\capgraph{T}$. Moreover, since $\capgraph{T[R_{3-i}]}$ is not empty, there is at least one edge $\edge{c}{d}$ in $\capgraph{T[R_{3-i}]}$. Since $R_{i-3}$ is a minimal {\teq}-retentive set of $T$, according to Lemma~\ref{lem:inneighboronesouce}, the edge $\edge{c}{d}$ still exists in $\capgraph{T}$. According to Lemma~\ref{lem:distinctminiteqs}, $R_1\cap R_2=\emptyset$; thus, both the vertex $c$ and the vertex $d$ are not in the circle $C$. However, this contradicts with Lemma~\ref{lem:dominationgraphproperty}.

Let $v,u\in R_1$ be two vertices that are {\sibling}s in $\capgraph{T[R_1]}$, and $w\in R_{2}$ a non-isolated vertex in ${\capgraph{T[R_{2}]}}$. In the following, we prove that $[w,v]\asymp [w,u]$.
To this end, let $y$ be an arbitrary neighbor of $w$ in $\capgraph{T[R_{2}]}$ (since $w$ is a non-isolated vertex in $\capgraph{T[R_{2}]}$, such a vertex $y$ exists). Let $Tree$ be the tree that includes $v$ and $u$ in ${\capgraph{T[R_1]}}$.
It is a folklore that there is a unique path between every two vertices in a tree. Let $Path=(v,x_1,x_2,...,x_k,u)$ be the unique path between $v$ and $u$ in $Tree$. Since $v$ and $u$ are siblings in $\capgraph{T[R_1]}$, $dist(v,x_1)\equiv dist(u,x_1) \mod{(2)}$. Since $dist(v,x_1)=1$ and $dist(u,x_1)=k$, it follows that
$k$ is odd.
For simplicity, in the following we define $v=x_0$ and $u=x_{k+1}$.
According to~Lemma~\ref{lem:teqisomorhpictoT43}, we have that
\[[w,x_j]\not\asymp[w,x_{j+1}]~\text{and}~[w,x_{j+1}]\not\asymp[w,x_{j+2}]\]
\noindent for every $j\in \{0,1,...,k-1\}$ (notice that these hold regardless of the directions of the arc between $v$ and $u$, and the arc between $w$ and $y$).
Therefore, it holds that $[w,x_j]\asymp[w,x_{j+2}]$, for every $j\in \{0,1,...,k-1\}$. Thus, we have 
\[[w,x_0=v]\asymp[w,x_2]\asymp [w,x_4],...,[w,x_{k+1}=u].\]

Now we prove the last part of the lemma. Let $v,u\in R_1$ be now two vertices that are in the same tree in $\capgraph{T[R_1]}$ but are not {\sibling}s, and $w\in R_{2}$ still a non-isolated vertex in ${\capgraph{T[R_{2}]}}$. The proof is similar to the above proof. Let $(v,x_1,x_2,...,x_k,u)$ still denote the unique path between $v$ and $u$ in the tree that includes both $v$ and $u$. Notice that in this case, $k$ is even (since $dist(v,x_1)=1$, $dist(u,x_1)=k$ and $v$ and $u$ are not {\sibling}s). Moreover, $v$ and $x_{k}$ are siblings. Then, according to above proof, we have that $[w,v]\asymp[w,x_{k}]$. Then, according to Lemma~\ref{lem:teqisomorhpictoT43}, $[w,u]\not\asymp[w,x_{k}]$, which implies that $[w,v]\not\asymp [w,u]$.
\end{proof}

\begin{lemma}
\label{lem:tsixnineiterativetricaptain}
Let $T$ be a tournament, and $R_1$ and $R_2$ be two vertex subset of $T$ such that $R_1\cap R_2=\emptyset$. If the following conditions hold, $R_1$ and $R_2$ cannot be minimal {\teq}-retentive sets of $T$ simultaneously.

(1) $T[R_1]$ is isomorphic to the tournament $T_6(2)$ in \myfig{fig:teqtournamentofsizesix};

(2) there is a vertex $v\in R_2$ such that

(2.1) $v$ has a tri-captain $\{x,y,z\}\subset R_2$ in $T$; and

(2.2) at least two of $\{x,y,z\}$ have tri-captains whose vertices are from $R_2$ in $T$.
\end{lemma}
\begin{proof}
{\assumecontradiction} For ease of exposition, assume that $T_6(2)=T[R_1]$. Moreover, let $R_1=\{a,b,c,d,e,f\}$ and let the vertices in $R_1$ be labeled as in the tournament $T_6(2)$ in \myfig{fig:teqtournamentofsizesix}.
We first prove the following claim.

{\bf{Claim.}} If there are four vertices $v',x',y',z'\in R_2$ such that $\{x',y',z'\}$ is the tri-captain of $v'$ in $T$, then $b\trelation v'$.

We prove the claim by contradiction. Assume that $v'\trelation b$. Then, since $d$ is the captain of $b$ in $T[R_1]$ and $R_1$ is a minimal {\teq}-retentive set of $T$, according to Lemma~\ref{lem:novertexdominatestwoadjacentverticesincapgraingraphs}, it holds that $d\trelation v'$. Since $\{x',y',z'\}$ is a tri-captain of $v'$ in $T$, at least two of $\{x',y',z'\}$ dominate $d$.  Since $f$ is the captain of $d$ in $T[R_1]$ and $R_1$ is a minimal {\teq}-retentive set of $T$,  according to Lemma~\ref{lem:novertexdominatestwoadjacentverticesincapgraingraphs}, the two vertices in $\{x',y',z'\}$ which dominate $d$ must be dominated by $f$. This implies that $v'\trelation f$, since otherwise, according to Lemma~\ref{lem:novertexdominatestwoadjacentverticesincapgraingraphs}, at least two of $\{x',y',z'\}$ should have dominated $f$.  Then, since $e$ is the captain of $f$ in $T[R_1]$ and $R_1$ is a minimal {\teq}-retentive set of $T$, according to Lemma~\ref{lem:novertexdominatestwoadjacentverticesincapgraingraphs}, $e\trelation v'$. Analogous to what we proved for $v'\trelation f$ as above, we can show that $v'\trelation c$. Since $\{a,b,f\}$ is a tri-captain of $c$ in $T[R_1]$, and $c$ has no more than 5 inneighbors in $R_1$, according to Lemma~\ref{lem:localboundedsmallteqtournament}, we have that $\{a,b,f\}=\mathteq{(T[N^-_T(c)])}$. Then, according to Lemma~\ref{lem:tauthreeequalretentive} and Lemma~\ref{lem:t3minimalretentiveset}, at least two of $\{a,b,f\}$ dominate $v'$. However, we have just showed that $\{v'\}\trelation \{b,f\}$; a contradiction. This completes the proof for the claim.

According to the above claim, if at least two of $\{x,y,z\}$ have tri-captains whose vertices are from $R_2$ in $T$, say $x$ and $y$ do, then $\{b\}\trelation \{x,y\}$. However, since $\{x,y,z\}$ is the tri-captain of $v$ in $T$, at least two of $\{x,y,z\}$ should have dominated $b$; a contradiction.
\end{proof}

Now we are ready to show another main result of this paper.

\begin{theorem}
\label{thm:notsixsixseverneightteqtournaments}
There are no $(T_6,T_6\cup T_7\cup T_8)$-retentive tournaments.
\end{theorem}
\myvspace{-5pt}
\begin{proof}
Due to Theorem~\ref{thm:betas}, if a tournament $T$ has a minimal {\teq}-retentive set $R$ of size~6, then $T[R]$ is either isomorphic to the tournament $T_6(1)$ or to the tournament $T_6(2)$ in \myfig{fig:teqtournamentofsizesix}. Hence, it suffices to prove that there are no $(\{T_6(1),T_6(2)\},\{T_6(1),T_6(2)\}\cup T_7\cup T_8)$-retentive tournaments. We break down the proof of the theorem into the following cases.
\caseinproofskip

Case~1. There are no $(T_6(1),H)$-retentive tournaments where $\capgraph{H}$ is not empty.

Assume for the sake of contradiction that this is not the case. Let $T$ be a tournament which has two minimal {\teq}-retentive sets $R_1$ and $R_2$ such that $T[R_1]$ is isomorphic to the tournament $T_6(1)$ in \myfig{fig:teqtournamentofsizesix} and $T[R_2]$ is isomorphic to a tournament $H$ such that $\capgraph{H}$ is not empty. For ease of exposition, assume that $T_6(1)=T[R_1]$ (thus $R_1=\{a,b,c,d,e,f\}$ and the vertices are labeled as in \myfig{fig:teqtournamentofsizesix}). Observe that there is a broom $\broom{b}{f}{e,d}$ in $\dcapgraph{T[R_1]}$. Moreover, $\{e,d,a\}$ is a tri-captain of $b$ in $T[R_1]$. 
Then, according to Lemma~\ref{lem:localboundedsmallteqtournament}, $\{e,d,a\}=\mathteq{(T[N^-_{T}(b)])}$. However, according to Lemma~\ref{twoteqonecaptainbroomonenonemptycaptain}, $R_1$ and $R_2$ cannot be minimal {\teq}-retentive sets of $T$ simultaneously; a contradiction.

We wrote a C++ program to filter out all non-isomorphic tournaments of sizes $6,7,8$ such that their domination graphs are not empty.
After filtering out these tournaments, there remain no tournament of size 6, one tournament of size 7 and four tournaments of size 8. The tournament of size 7 is isomorphic to the tournament $T_7(1)$ in \myfig{fig:tsevenone}, and the four tournaments of size 8 are isomorphic to the tournaments $T_8(1),T_8(2),T_8(3),T_8(4)$ in \myfig{fig:teight}, respectively.
Due to our results from the C++ program and the above proof for Case~1, it is sufficient to consider the following cases.
\caseinproofskip

Case~2. There are no $(T_6(1),T_7(1))$-retentive tournaments.

Assume for the sake of contradiction that this is not the case. Let $T$ be a tournament which has two minimal {\teq}-retentive sets $R_1$ and $R_2$ such that $T[R_1]$ is isomorphic to the tournament $T_6(1)$ in \myfig{fig:teqtournamentofsizesix}, and $T[R_2]$ is isomorphic to the tournament $T_7(1)$ in \myfig{fig:tsevenone}. For ease of exposition, assume that $T_7(1)=T[R_2]$. Moreover, let the vertices in $R_2=\{a,b,c,d,e,f,g\}$ be labeled as in the tournament $T_7(1)$ in \myfig{fig:tsevenone}. It is easy to verify that $\{e,f,g\}$ is the tri-captain of $a$ in $T[R_2]$, and $a$ has no more than $5$ inneighbors in $R_2$. Then, due to Lemma~\ref{lem:localboundedsmallteqtournament}, $\{e,f,g\}=\mathteq{(T[N^-_{T}(a)])}$. However, this cannot be the case according to Lemma~\ref{lem:tsixsixtricaptain}.
\caseinproofskip

Case 3. There are no $(T_6(1),  T_8(i))$-{\teq}-retentive tournaments for all $i\in \{1,2,3,4\}$.

One can check that $\{b,c,d,e,f,g,h\}$ is a 3-locally bounded minimal {\teq}-retentive set in each of the tournaments $T_8(1), T_8(2), T_8(3)$ (see \mytable{tab:teighttricaptains} for a summary of tri-captains of all vertices). Moreover, every of $\{b,c,d,e,f,g,h\}$ has at most 5 inneighbors in $T_8(i)$ for every $i\in \{1,2,3\}$. Then, due to Lemma~\ref{lem:threeloallyboudnedminimalteqagainst}, none of $\{T_8(1), T_8(2), T_8(3)\}$ is a {\teq}-retentive tournament. Therefore, there are no $(\{T_6(1)\}, \{T_8(1),T_8(2),T_8(3)\})$-{\teq}-retentive tournaments. It remains to show that there are no  $(T_6(1),  T_8(4))$-{\teq}-retentive tournaments. It is easy to verify that $\{h,g,e\}$ is the tri-captain of the vertex $a$ in $T_8(4)$. Moreover, the vertex $a$ has no more than 5 inneighbors in $T_8(4)$. According to Lemma~\ref{lem:localboundedsmallteqtournament}, we have that $\{h,g,e\}=\mathteq{(T[N^-_{T}(a)])}$. Then, according to Lemma~\ref{lem:tsixsixtricaptain}, there is no tournament $T$ which admits two minimal {\teq}-retentive sets $R_1$ and $R_2$ such that $T[R_1]$ is isomorphic to $T_6(1)$ and $T[R_2]$ is isomorphic to $T_8(4)$; a contradiction.
\caseinproofskip

Case 4. There are no $(T_6(2),H)$-retentive tournaments where ${\capgraph{H}}$ is not empty.

Due to Lemma~\ref{lem:againsttwoteqonetsixnineonecaptainnotempty}, there is no tournament $T$ which has two minimal {\teq}-retentive sets $R_1$ and $R_2$ such that $T[R_1]$ is isomorphic to $T_6(2)$ and $\capgraph{T[R_2]}$ is not empty. This completes the proof for this case.

As discussed above in the proof for Case~1, all tournaments of size 7 which have empty domination graphs are isomorphic to $T_7(1)$. Moreover, all tournaments of size 8 which have empty domination graphs are isomorphic to one of $\{T_8(1),T_8(2),T_8(3),T_8(4)\}$. Furthermore, due to the proof for Case~3, none of $\{T_8(1),T_8(2),T_8(3)\}$ is a {\teq}-retentive tournament. Hence, it suffices to consider the following case.
\caseinproofskip

Case~5. There are no $(\{T_6(2)\},\{T_7(1),T_8(4)\})$-retentive tournaments.

Since every vertex in $T_7(1)$ and every vertex in $T_8(4)$ has a tri-captain (see \myfig{fig:tsevenone} and \mytable{tab:teighttricaptains} for further details), according to Lemma~\ref{lem:tsixnineiterativetricaptain}, there is no tournament $T$ which admits two minimal {\teq}-retentive sets $R_1$ and $R_2$ such that $T[R_1]$ is isomorphic to $T_6(2)$ and $T[R_2]$ is isomorphic to any of $\{T_7(1),T_8(4)\}$. This completes the proof for this case.
\end{proof}

Now, we study $(T_7, T_7)$-{\teq}-retentive tournaments.

\begin{theorem}\label{thm:notsevenseventeqtournaments}
There are no $(T_7, T_7)$-{\teq}-retentive tournaments.
\end{theorem}

The remainder of this paper is devoted to the proof of Theorem~\ref{thm:notsevenseventeqtournaments}.
As discussed in Section~\ref{sec:numberofteqretentivetournaments} (before Theorem~\ref{thm:betas}), each {\teq}-retentive tournament of size $7$ is isomorphic to one of the 26 tournaments $T_7(i)$, $i\in \{1,2,...,26\}$, shown in Appendix. Observe that, due to symmetry, if there are no $(H_1,H_2)$-{\teq}-retentive tournaments, then there are no $(H_2,H_1)$-{\teq}-retentive tournaments. Therefore, to prove Theorem~\ref{thm:notsevenseventeqtournaments} it suffices to prove that there are no $(T_7(i), T_7(j))$-{\teq}-retentive tournaments for every two integers $1\leq i\leq j\leq 26$.
%
%
We break down our proof into numerous lemmas, with each stating the non-existence of $(\mathcal{H}_1,\mathcal{H}_2)$-retentive tournaments, for $\mathcal{H}_1$ and $\mathcal{H}_2$ being subsets of $\{T_7(i) \mid i\in \{1,2,...,26\}\}$. Moreover, the union of these lemmas cover all possibilities of $T_7(i)$ and $T_7(j)$ mentioned above.  In Appendix, the directed domination graph of each tournament $T_7(i)$, $i\in \{1,2,...,26\}$, is also given. Moreover, if a vertex has a tri-captain, its tri-captain is labeled next to the vertex. These information are helpful for the readers to follow the proofs of the lemmas.

\begin{lemma}
\label{lem:againsttwoteqonetsevenoneonecaptainnotempty}
Let $T$ be a tournament, and $R_1$ and $R_2$ be two vertex subsets of $T$ such that $R_1\cap R_2=\emptyset$. If the following conditions hold, then $R_1$ and $R_2$ cannot be minimal {\teq}-retentive sets of $T$ simultaneously.

(1) $T[R_1]$ is isomorphic to the tournament $T_7(1)$; and

(2) $\dcapgraph{T[R_2]}$ is not empty.
\end{lemma}

\begin{proof}
{\assumecontradiction} For ease of exposition, assume that $T_7(1)=T[R_1]$; thus, $R_1=\{a,b,c,d,e,f,g\}$. 
Let $v\trelation u$ be an arbitrary arc in $\dcapgraph{T[R_2]}$.  We complete the proof by distinguishing between the following two cases, with respect to the arc between the vertices $u$ and $e$. We shall show that each case leads to a contradiction. Before proceeding further, let's first observe a useful  property. The tri-captain of each vertex in $T_7(1)$ is shown in Appendix. Since each vertex has indegree at most 5 in $T_7(1)$ and $R_1$ is a minimal {\teq}-retentive set of $T$, according to Lemma~\ref{lem:localboundedsmallteqtournament}, it holds that

$\{e,f,g\}=\mathteq{(T[N^-_{T}(a)])}$,

$\{a,d,g\}=\mathteq{(T[N^-_{T}(b)])}$,

$\{a,b,f\}=\mathteq{(T[N^-_{T}(c)])}$,

$\{a,c,e\}=\mathteq{(T[N^-_{T}(d)])}$,

$\{b,c,g\}=\mathteq{(T[N^-_{T}(e)])}$,

$\{b,d,e\}=\mathteq{(T[N^-_{T}(f)])}$,

$\{c,d,f\}=\mathteq{(T[N^-_{T}(g)])}$.

\caseinproofskip
Case~1. $u\trelation e$.

Since $\{b,c,g\}=\mathteq{(T[N^-_{T}(e)])}$, according to Lemma~\ref{lem:tauthreeequalretentive} and Lemma~\ref{lem:t3minimalretentiveset}, at least two of $\{b,c,g\}$ dominate $u$. We consider all subcases.

Subcase 1. $\{b,c\}\trelation \{u\}$

According to Lemma~\ref{lem:novertexdominatestwoadjacentverticesincapgraingraphs}, $\{v\}\trelation \{b,c\}$.
Since $\{a,b,f\}=\mathteq{(T[N^-_{T}(c)])}$, according to Lemma~\ref{lem:tauthreeequalretentive} and Lemma~\ref{lem:t3minimalretentiveset}, at least two of $\{a,b,f\}$ dominate $v$. Since $v\trelation b$, it holds that $\{a,f\}\trelation \{v\}$. Then, according to Lemma~\ref{lem:novertexdominatestwoadjacentverticesincapgraingraphs}, $\{u\}\trelation \{a,f\}$.
Since $\{e,f,g\}=\mathteq{(T[N^-_{T}(a)])}$, according to Lemma~\ref{lem:tauthreeequalretentive} and Lemma~\ref{lem:t3minimalretentiveset}, at least two of $\{e,f,g\}$ dominate $u$. However, we have showed above that $\{u\}\trelation \{e,f\}$; a contradiction.

Subcase 2. $\{b,g\}\trelation \{u\}$

According to Lemma~\ref{lem:novertexdominatestwoadjacentverticesincapgraingraphs}, $\{v\}\trelation \{b,g\}$.
Since $\{a,d,g\}=\mathteq{(T[N^-_{T}(b)])}$, according to Lemma~\ref{lem:tauthreeequalretentive} and Lemma~\ref{lem:t3minimalretentiveset}, at least two of $\{a,d,g\}$ dominate $v$. Since $v\trelation g$, it holds that $\{a,d\}\trelation \{v\}$. Then, according to Lemma~\ref{lem:novertexdominatestwoadjacentverticesincapgraingraphs}, $\{u\}\trelation \{a,d\}$.
Since $\{a,c,e\}=\mathteq{(T[N^-_{T}(d)])}$, according to Lemma~\ref{lem:tauthreeequalretentive} and Lemma~\ref{lem:t3minimalretentiveset}, at least two of $\{a,c,e\}$ dominate $u$. However, we have showed above that $\{u\}\trelation \{a,e\}$; a contradiction.

Subcase 3. $\{c,g\}\trelation \{u\}$

According to Lemma~\ref{lem:novertexdominatestwoadjacentverticesincapgraingraphs}, $\{v\}\trelation \{c,g\}$.
Since $\{c,d,f\}=\mathteq{(T[N^-_{T}(g)])}$, according to Lemma~\ref{lem:tauthreeequalretentive} and Lemma~\ref{lem:t3minimalretentiveset}, at least two of $\{c,d,f\}$ dominate $v$. Since $v\trelation c$, it holds that $\{d,f\}\trelation \{v\}$. Then, according to Lemma~\ref{lem:novertexdominatestwoadjacentverticesincapgraingraphs}, $\{u\}\trelation \{d,f\}$.
Since $\{b,d,e\}=\mathteq{(T[N^-_{T}(f)])}$, according to Lemma~\ref{lem:tauthreeequalretentive} and Lemma~\ref{lem:t3minimalretentiveset}, at least two of $\{b,d,e\}$ dominate $u$. However, we have showed above that $\{u\}\trelation \{d,e\}$; a contradiction.

\caseinproofskip
Case~2. $e\trelation u$.

According to Lemma~\ref{lem:novertexdominatestwoadjacentverticesincapgraingraphs}, it holds that $v\trelation e$.
Then, by exchanging all occurrences of $u$ and $v$ in the above proof for Case~1, we can show that this case leads to some contradictions. 
\end{proof}

Since all tournaments of size $7$ with empty domination graphs are isomorphic to $T_7(1)$, the above lemma implies that there are no $(T_7(1),T_7(i))$-{\teq}-retentive tournaments for all integer $2\leq i\leq 26$. In the following, we show that there are no $(T_7(1),T_7(1))$-{\teq}-retentive tournaments.

\begin{lemma}
\label{lem:againsttwoteqbothtsevenone}
Let $T$ be a tournament, and $R_1$ and $R_2$ be two vertex subsets of $T$ such that $R_1\cap R_2=\emptyset$. If both $T[R_1]$ and $T[R_2]$ are isomorphic to $T_7(1)$, then $R_1$ and $R_2$ cannot be minimal {\teq}-retentive sets of $T$ simultaneously.
\end{lemma}

\begin{proof}
We prove the lemma by contradiction. Assume that $R_1$ and $R_2$ are two distinct minimal {\teq}-retentive sets of $T$ such that both $T[R_1]$ and $T[R_2]$ are isomorphic to the tournament $T_7(1)$. 
Let $\{a(i),b(i),c(i),d(i),e(i),f(i),g(i)\}=R_i$ for $i=1,2$, where every vertex $v(i)\in \{a(i),...,g(i)\}$ corresponds to the vertex $v$ in $T_7(1)$.
Due to the proof of Lemma~\ref{lem:againsttwoteqonetsevenoneonecaptainnotempty}, for both $i=1,2$ it holds that

$\{e(i),f(i),g(i)\}=\mathteq{(T[N^-_{T}(a(i))])}$,

$\{a(i),d(i),g(i)\}=\mathteq{(T[N^-_{T}(b(i))])}$,

$\{a(i),b(i),f(i)\}=\mathteq{(T[N^-_{T}(c(i))])}$,

$\{a(i),c(i),e(i)\}=\mathteq{(T[N^-_{T}(d(i))])}$,

$\{b(i),c(i),g(i)\}=\mathteq{(T[N^-_{T}(e(i))])}$,

$\{b(i),d(i),e(i)\}=\mathteq{(T[N^-_{T}(f(i))])}$,

$\{c(i),d(i),f(i)\}=\mathteq{(T[N^-_{T}(g(i))])}$.

One can check easily that for any two tournament equilibrium sets $A$ and $B$ listed above (the sets on the left side), $|A\cap B|\leq 1$. 

For a vertex $v$ and a vertex subset $S$ of $T$, let $N^+(v,S)$ denote the set of outneighbors of $v$ in $S$.
The number of arcs between the vertices in $R_1$ and the vertices in $R_2$ is \[\sum_{v\in R_1}N^+(v,R_2)+\sum_{v\in R_2}N^+(v,R_1).\]
Since $|R_1|=|R_2|=7$, there are exactly $7\times 7=49$ arcs between them; that is, \[\sum_{v\in R_1}N^+(v,R_2)+\sum_{v\in R_2}N^+(v,R_1)=49.\] This implies that there is at least one vertex $v\in R_i$ such that $N^+(v,R_{3-i})\geq 4$ for some $i\in \{1,2\}$. Due to symmetry, assume that $i=1$. Moreover, let $x,y,z,w$ be any four outneighbors of $v$ in $R_2$.
Then, for each $u\in \{x,y,z,w\}$ it holds that \[|\{x,y,z,w\}\cap \mathteq{(T[N^-_{T}(u)])}|\leq 1;\] since otherwise, according to Lemma~\ref{lem:tauthreeequalretentive} and Lemma~\ref{lem:t3minimalretentiveset}, at least two of $\mathteq{(T[N^-_{T}(u)])}$ should have dominated $v$.\comments{further explanation?}
Then, since $|\mathteq{(T[N^-_{T}(u)])}|=3$ for every $u\in \{x,y,z,w\}$ as discussed above, it holds that \[|\mathteq{(T[N^-_{T}(u)])}\cap (R_2\setminus \{x,y,z,w\})|\geq 2.\] Moreover, since $|R_2\setminus \{x,y,z,w\}|=3$, according to Pigeonhole principle, there are two distinct vertices $\{u_1,u_2\}\subset \{x,y,z,w\}$ such that \[|\mathteq{(T[N^-_{T}(u_1)])}\cap \mathteq{(T[N^-_{T}(u_2)])}\cap (R_2\setminus \{x,y,z,w\})|\geq 2.\]
However, this cannot be the case, since $|\mathteq{(T[N^-_{T}(u_1)])}\cap \mathteq{(T[N^-_{T}(u_2)])}|\leq  1$ for every possible $u_1,u_2\in R_2$, as we have observed above; a contradiction.
\end{proof}

Combining Lemmas~\ref{lem:againsttwoteqonetsevenoneonecaptainnotempty} and~\ref{lem:againsttwoteqbothtsevenone}, we can conclude that there are no $(T_7(1),T_7(i))$-{\teq}-retentive tournaments for all integer $1\leq i\leq 26$. In the following, we prove that there are no $(\{T_7(2),T_7(3)\},\bigcup_{2\leq i\leq 26}\{T_7(i)\})$-{\teq}-retentive tournaments. Let's first study some useful properties.

\begin{lemma}
\label{claim:a}
Let $T$ be a tournament and $R$ a minimal {\teq}-retentive set of $T$. If there is a vertex $v\in R$ such that $T[N^-_{T[R]}(v)]$ is isomorphic to the tournament $T_5(1)$ or the tournament $T_5(6)$ in \myfig{fig:tfive}, then $N^-_{T[R]}(v)$ is the unique minimal {\teq}-retentive set of $T[N^-_T(v)]$, that  is, $N^-_{T[R]}(v)=\mathteq{(T[N^-_T(v)])}$.
\end{lemma}
\begin{proof}
Let $v$ be a vertex as stated in the lemma such that $T[N^-_{T[R]}(v)]$ is isomorphic to $T_5(1)$ (resp. $T_5(6)$). For ease of exposition, let $T_5(1)=T[R]$ (resp. $T_5(6)=T[R]$).
Since $R$ is a minimal {\teq}-retentive set of $T$, it holds that $\mathteq{(T[N^-_T(v)])}\subseteq N^-_{T[R]}(v)$. Since $T[N^-_{T[R]}(v)]$ is isomorphic to $T_5(1)$ (resp. $T_5(6)$), there is no source in $T[N^-_{T[R]}(v)]$. Then, according to Lemma~\ref{lem:teqsourceuniqueminial}, $|\mathteq{(T[N^-_T(v)])}|\neq 1$. Furthermore, since there is no directed triangle in $T_5(1)$ (resp. $T_5(6)$) such that every vertex in $T_5(1)$ (resp. $T_5(6)$) not in the triangle is dominated by at least two vertices in the triangle, according to Lemma~\ref{lem:t3minimalretentiveset}, there is no minimal {\teq}-retentive set of size $3$ in $T[N^-_T(v)]$. Then, due to Lemma~\ref{lem:tauthreeequalretentive}, it holds that $|\mathteq{(T[N^-_T(v)])}|\neq 3$. According to Lemma~\ref{lem:teqirreducible} and Theorem~\ref{lem:noretentivesetofsizefour}, there are no minimal {\teq}-retentive sets of size~$2$ and~$4$. This leaves only the case that $N^-_{T[R]}(v)$ is the unique minimal {\teq}-retentive set of $T[N^-_T(v)]$. Hence, it holds that $\mathteq{(T[N^-_T(v)])}=N^-_{T[R]}(v)$.
\end{proof}

\begin{lemma}
\label{lem:againsttwoteqsbroomtfiveoneandsixcaptainnotempty}
Let $T$ be a tournament, and $R_1$ and $R_2$ be two vertex subsets of $T$ such that $R_1\cap R_2=\emptyset$. If the following conditions hold, then $R_1$ and $R_2$ cannot be minimal {\teq}-retentive sets of $T$ simultaneously.

(1) There is a broom $\broom{u}{v}{w_1,...,w_5}$ in $\dcapgraph{T[R_1]}$ such that

(1.1) $T[w_1,...,w_5]$ is isomorphic to the tournament $T_5(1)$ or the tournament $T_5(6)$ in \myfig{fig:tfive};

(1.2) $N^-_{T[R_1]}(u)=\{w_1,...,w_5\}$; and

(2) $\dcapgraph{T[R_2]}$ is not empty.
\end{lemma}

\begin{proof}
We prove the lemma by contradiction. Assume that the above conditions hold, and $R_1$ and $R_2$ are two minimal {\teq}-retentive sets of $T$. Let $\bar{v}\trelation \bar{u}$ be any arbitrary arc in $\dcapgraph{T[R_2]}$. We distinguish between the following two cases.

\caseinproofskip
Case~1. $\bar{u}\trelation v$.

According to Lemma~\ref{lem:teqisomorhpictoT43}, $\{\bar{v}\}\trelation \{u,w_1,...,w_5\}$. Moreover, due to {\conditions{1.1}{1.2}} and Lemma~\ref{claim:a}, we know that $\{w_1,...,w_5\}$ is the unique minimal {\teq}-retentive set of $T[N^-_T(u)]$. The directed domination graphs of both $T_5(1)$ and $T_5(6)$ are not empty (see \myfig{fig:tfiveonesixcaptaingraphs}). Let $w_i\trelation w_j$ with $i,j\in \{1,2,3,4,5\}$ be any arbitrary arc in $\dcapgraph{T[w_1,...,w_5]}$ (since $T[w_1,...,w_5]$ is isomorphic to $T_5(1)$ or $T_5(6)$, such an arc exists). Then, due to Lemma~\ref{lem:inneighboronesouce}, $w_i$ is still the captain of $w_j$ in $T[N^-_T(u)]$. However, this contradicts with $\{\bar{v}\}\trelation \{w_i,w_j\}$.
\caseinproofskip

Case~2. $v\trelation \bar{u}$.

In this case, according to Lemma~\ref{lem:novertexdominatestwoadjacentverticesincapgraingraphs}, $\bar{v}\trelation v$. Then, we can achieve the same contradiction as in the above proof for Case~1 by exchanging all appearances of $\bar{v}$ and $\bar{u}$.
\end{proof}

Observe that in both $\dcapgraph{T_7(2)}$ and $\dcapgraph{T_7(3)}$ there is a broom $\broom{a}{g}{b,c,d,e,f}$ that satisfies {\condition{1}} in the above lemma. Therefore, according to the above lemma, there are no $(\{T_7(2),T_7(3)\},\bigcup_{2\leq i\leq 26}\{T_7(i)\})$-{\teq}-retentive tournaments.

In the following, we prove that there are no $(T_7(i),T_7(j))$-{\teq}-retentive tournaments for several different values of $i$ and $j$. To this end, we derive some useful properties. 

\begin{lemma}
\label{lem:generaltwoteqonecaptainbroomonenonemptycaptain}
Let $T$ be a tournament, and $R_1$ and $R_2$ be two vertex subsets of $T$ such that $R_1\cap R_2=\emptyset$ and $\capgraph{T[R_1]}$ is a forest. If the following conditions hold, $R_1$ and $R_2$ cannot be minimal {\teq}-retentive sets of $T$ simultaneously.

(1) there are four non-isolated vertices $u,v,w_1,w_2$ in $\capgraph{T[R_1]}$ such that

(1.1) $u,v,w_1,w_2$ are in the same tree in $\capgraph{T[R_1]}$;

(1.2) $|\mathteq{(T[N^-_T(u)])}|=3$ and $\{w_1,w_2\}\subset \mathteq{(T[N^-_T(u)])}$;

(1.3) $u$ and $w_1$, and $u$ and $w_2$ are siblings in $\capgraph{T[R_1]}$;

(1.4) $u$ and $v$ are not siblings in $\capgraph{T[R_1]}$; and

(2) $\dcapgraph{T[R_2]}$ is not empty.
\end{lemma}
\begin{proof}
We prove the lemma by contradiction. Assume that the above conditions hold, and $R_1$ and $R_2$ are two minimal {\teq}-retentive sets of $T$.
Let $\bar{v}\trelation \bar{u}$ be any arc in $\dcapgraph{T[R_2]}$ (since $\dcapgraph{T[R_2]}$ is not empty, such an arc exists).
We distinguish between two cases with respect to the arc between $\bar{v}$ and $v$. 

\caseinproofskip

Case~1. $v\trelation \bar{v}$.

Since $v$ and $u$ are not siblings in $\capgraph{T[R_1}$, according to Lemma~\ref{lem:captaingrapht43property}, it holds that $\bar{v}\trelation u$. Since $u$ and  $w_1$, and $u$ and $w_2$ are siblings, according to Lemma~\ref{lem:captaingrapht43property}, it holds that $\bar{v}\trelation \{w_1,w_2\}$. However, since $|\mathteq{(T[N^-(u)])}|=3$ and $\{w_1,w_2\}\subset \mathteq{(T[N^-(u)])}$, according to Lemma~\ref{lem:tauthreeequalretentive} and  Lemma~\ref{lem:t3minimalretentiveset}, at least one of $\{w_1,w_2\}$ should have dominated $\bar{v}$; a contradiction.
\caseinproofskip

Case~2. $\bar{v}\trelation {v}$.

Since $v$ and $u$ are not siblings in $\capgraph{T[R_1}$, according to Lemma~\ref{lem:captaingrapht43property}, it holds that $u\trelation \bar{v}$. Then, according to Lemma~\ref{lem:novertexdominatestwoadjacentverticesincapgraingraphs}, $\bar{u}\trelation u$. Now, we can prove the lemma for this case by replacing all occurrences of $\bar{v}$ by $\bar{u}$ in the above proof (from second sentence) for Case 1. 
\end{proof}

Observe that each $T_7(i)$ for $i=4,5,6,7,11,13,14,16,17,21$ satisfies {\condition{1}} for $R_1$ in Lemma~\ref{lem:generaltwoteqonecaptainbroomonenonemptycaptain} (see Appendix and \mytable{table:notsevenmany} for further details). Then, according to the above lemma, there are no $(\{T_7(i)\},\bigcup_{2\leq i\leq 26}\{T_7(i)\})$-{\teq}-retentive tournaments for all $i=4,5,6,7,11,13,14,16,17,21$.

\begin{table}
\begin{center}
\begin{tabular}{l|c|c|c|c}

           &     $u$  &  $v$   & $w_1$   &   $w_2$\\ \hline

$T_7(4)$   &     $a$  &  $f$     & $c$    &    $d$ \\ \hline

$T_7(5)$   &    $a$  &  $g$     & $c$    &    $d$   \\ \hline

$T_7(6)$   &     $a$  &  $g$     & $c$    &    $d$ \\ \hline

$T_7(7)$   &     $f$  &  $g$     & $c$    &    $d$  \\ \hline

$T_7(11)$   &    $c$  &  $f$     & $e$    &    $g$   \\ \hline

$T_7(13)$   &     $b$  &  $f$     & $g$    &    $e$ \\ \hline

$T_7(14)$   &     $e$  &  $f$     & $g$    &    $d$  \\ \hline

$T_7(16)$   &    $c$  &  $f$     & $g$    &    $d$   \\ \hline

$T_7(17)$   &     $d$  &  $f$     & $c$    &    $e$ \\ \hline

$T_7(21)$   &     $b$  &  $f$     & $g$    &    $e$ \\ \hline
\end{tabular}
\end{center}
\caption{This table shows the correspondences between vertices in $T_7(i)$ for $i=4,5,6,7,11,13,14,16,17,21$ and the vertices in $R_1$ in the proof of Lemma~\ref{lem:generaltwoteqonecaptainbroomonenonemptycaptain}. Replacing all appearances of the vertices in $R_1$ in the proof with their corresponding vertices in each $T_7(i)$ can prove that there are no $(T_7(i),H)$-{\teq}-retentive tournaments, where $\capgraph{H}\neq \emptyset$.}
\label{table:notsevenmany}
\end{table}

In the following, we show that there are no $(\{T_7(8)\},\bigcup_{2\leq i\leq 26}\{T_7(i)\})$-{\teq}-retentive tournaments.

\begin{lemma}
\label{lem:againsttwoteqonetseveneightonecaptainnotempty}
Let $T$ be a tournament, and $R_1$ and $R_2$ be two vertex subsets of $T$ such that $R_1\cap R_2=\emptyset$. If the following conditions hold, $R_1$ and $R_2$ cannot be minimal {\teq}-retentive sets of $T$ simultaneously.

(1) $T[R_1]$ is isomorphic to the tournament $T_7(8)$; and

(2) $\dcapgraph{T[R_2]}$ is not empty.
\end{lemma}

\begin{proof}
We prove the lemma by contradiction. Suppose that the above two conditions hold, and $R_1$ and $R_2$ are both minimal {\teq}-retentive sets of $T$. For ease of exposition, assume that $T_7(8)=T[R_1]$; thus, $R_1=\{a,b,c,d,e,f,g\}$. Moreover, let the vertices in $R_1$ be labeled as in the tournament $T_7(8)$. Let $v\trelation u$ be an arbitrary arc in $\dcapgraph{T[R_2]}$.  We complete the proof by distinguishing between the following two cases. Before proceeding further, let's first study a property. First, it is easy to check that $\{d,f,g\}=N^-_{T[R_1]}(b)$ is a tri-caption of $b$ in $T[R_1]$, $\{b,e,g\}=N^-_{T[R_1]}(c)$ is a tri-captain of $c$ in $T[R_1]$, and $\{a,b,f\}=N^-_{T[R_1]}(e)$ is a tri-captain of $e$ in $T[R_1]$.
Then, since $R_1$ is a minimal {\teq}-retentive set of $T$, according to Lemma~\ref{lem:localboundedsmallteqtournament}, it holds that
$\{d,f,g\}=\mathteq{(T[N^-_{T}(b)])}$,
$\{b,e,g\}=\mathteq{(T[N^-_{T}(c)])}$, and
$\{a,b,f\}=\mathteq{(T[N^-_{T}(e)])}$.
\caseinproofskip

Case~1. $u\trelation b$.

Since $\{d,f,g\}=\mathteq{(T[N^-_{T}(b)])}$, according to Lemma~\ref{lem:tauthreeequalretentive} and Lemma~\ref{lem:t3minimalretentiveset}, at least two of $\{d,f,g\}$ dominate $u$. Since $f\trelation g$ is an arc in $\dcapgraph{T[R_1]}$ (see $\dcapgraph{T_7(8)}$ in Appendix), according to Lemma~\ref{lem:teqisomorhpictoT43}, $[u,g]\not\asymp [u,f]$. Therefore, only one of $\{g,f\}$ dominates $u$, implying that $d\trelation u$. Then, since $e\trelation d$ is an arc in $\dcapgraph{T[R_1]}$, according to Lemma~\ref{lem:teqisomorhpictoT43}, $u\trelation e$. Since $\{a,b,f\}=\mathteq{(T[N^-_{T}(e)])}$, according to Lemma~\ref{lem:tauthreeequalretentive} and Lemma~\ref{lem:t3minimalretentiveset}, at least two of $\{a,b,f\}$ dominate $u$. Since we have assumed $u\trelation b$, it holds that $\{a,f\}\trelation \{u\}$. Then, since $c\trelation f$ is an arc in $\dcapgraph{T[R_1]}$, according to Lemma~\ref{lem:teqisomorhpictoT43}, it holds that $u\trelation c$. Since $\{b,e,g\}=\mathteq{(T[N^-_{T}(c)])}$, according to Lemma~\ref{lem:tauthreeequalretentive} and Lemma~\ref{lem:t3minimalretentiveset}, at least two of $\{b,e,g\}$ dominate $u$. However, we have showed above that $\{u\}\trelation \{b,e\}$; a contradiction.

\caseinproofskip
Case~2. $b\trelation u$.

According to Lemma~\ref{lem:novertexdominatestwoadjacentverticesincapgraingraphs}, it holds that $v\trelation b$. We can then achieve a contradiction by replacing all appearances of $u$ with $v$ in the above argument for Case~1.
\end{proof}


In the following, we derive some further useful properties on {\teq}-retentive sets. Then, we use these properties to show that there are no $(\{T_7(i)\},\bigcup_{2\leq i\leq 26}\{T_7(i)\})$-{\teq}-retentive tournaments for several different values of $i$.

\begin{lemma}
\label{lem:againsttwoteqonexywonecaptainnotempty}
Let $T$ be a tournament, and $R_1$ and $R_2$ be two vertex subsets of $T$ such that $R_1\cap R_2=\emptyset$ and $\capgraph{T[R_1]}$ is a forest. If the following conditions hold, $R_1$ and $R_2$ cannot be minimal {\teq}-retentive sets of $T$ simultaneously

%
%
%
%
%
%
%
%
%

(1) There are four vertices $x,w_1,w_2,w_3$ in $R_1$ such that

(1.1) $\{w_1,w_2,w_3\}$ is a tri-captain of $x$ in $T[R_1]$ and $|N^-_{T[R_1]}(x)|\leq 5$; and

(1.2) $w_1$ and $w_2$ are in the same tree in $\capgraph{T[R_1]}$ but are not {\sibling}s in $\capgraph{T[R_1]}$.

(2) There are three vertices $y,w_4,w_5$ in $R_1$ (it is possible that $w_4$ and $w_1$ are identical) such that

(2.1) $\{x,w_4,w_5\}$ is a tri-captain of $y$ in $T[R_1]$ and $|N^-_{T[R_1]}(y)|\leq 5$; and

(2.2) $w_3, w_4, y$ are in the same tree in $\capgraph{T[R_1]}$, and $y$ and $w_3$, and $w_4$ and $w_3$ are not {\sibling}s in $\capgraph{T[R_1]}$.

(3) $\dcapgraph{T[R_2]}$ is not empty.
\end{lemma}

\begin{proof}
Assume for the sake of contradiction that all the above conditions hold, and $R_1$ and $R_2$ are both minimal {\teq}-retentive sets of $T$. Let $v\trelation u$ be any arbitrary arc in $\dcapgraph{T[R_2]}$. We complete the proof by distinguishing between the following two cases.
According to Conditions~(1.1)~and~(2.1) in the lemma, and according to Lemma~\ref{lem:localboundedsmallteqtournament}, it holds that $\{w_1,w_2,w_3\}=\mathteq{(T[N^-_{T}(x)])}$, and $\{x,w_4,w_5\}=\mathteq{(T[N^-_{T}(y)])}$.
\caseinproofskip

Case~1. $u\trelation x$.

Since $\{w_1,w_2,w_3\}=\mathteq{(T[N^-_{T}(x)])}$, according to Lemma~\ref{lem:tauthreeequalretentive} and Lemma~\ref{lem:t3minimalretentiveset}, at least two of $\{w_1,w_2,w_3\}$ dominate $u$. According to {\condition{1.2}} and Lemma~\ref{lem:captaingrapht43property}, $[w_1,u]\not\asymp [w_2,u]$. Therefore, only one of $\{w_1,w_2\}$ dominates $u$, implying that $w_3\trelation u$. Then, according to {\condition{2.2}} and Lemma~\ref{lem:captaingrapht43property}, $\{u\}\trelation \{w_4,y\}$. Since $\{x,w_4,w_5\}=\mathteq{(T[N^-_{T}(y)])}$, according to Lemma~\ref{lem:tauthreeequalretentive} and Lemma~\ref{lem:t3minimalretentiveset}, at least two of $\{x,w_4,w_5\}$ dominate $u$. However, we have showed above that $\{u\}\trelation \{w_4,x\}$; a contradiction.

\caseinproofskip
Case~2. $x\trelation u$.

According to Lemma~\ref{lem:novertexdominatestwoadjacentverticesincapgraingraphs}, it holds that $v\trelation x$. Then we can prove the lemma for this case by  replacing all appearances of $u$ with $v$ in the above proof for Case~1.
\end{proof}

The above lemma implies that there are no $(\{T_7(i)\},\bigcup_{2\leq i\leq 26}\{T_7(i)\})$-{\teq}-retentive tournaments for all $i=9,10,12,15$, since each of such $T_7(i)$ satisfies the conditions for $R_1$ as stated in the above lemma. See \mytable{table:x} for further details.

\begin{table}
\begin{center}
\begin{tabular}{l|c|c|c|c|c|c|c}

           &     $x$  &  $w_1$   & $w_2$   &   $w_3$    &    $y$   &   $w_4$   &   $w_5$\\ \hline

$T_7(9)$   &     $e$  &  $f$     & $g$    &    $c$     &     $d$   &   $f$     &   $a$ \\ \hline

$T_7(10)$   &    $b$  &  $e$     & $g$    &    $c$     &     $a$   &   $f$     &   $d$ \\ \hline

$T_7(12)$   &     $b$  &  $f$     & $g$    &    $e$     &     $c$   &   $g$     &   $a$ \\ \hline

$T_7(15)$   &     $b$  &  $f$     & $g$    &    $e$     &     $c$   &   $f$     &   $a$ \\ \hline
\end{tabular}
\end{center}
\caption{This table shows the correspondences between vertices in $T_7(i)$ for $i=9,10,12,15$ and the vertices in $R_1$ in the proof of Lemma~\ref{lem:againsttwoteqonexywonecaptainnotempty}. Replacing all appearances of the vertices in $R_1$ in the proof with their corresponding vertices in each $T_7(i)$ for $i=9,10,12,15$ can prove that there are no $(T_7(i),H)$-{\teq}-retentive tournaments, where $\capgraph{H}\neq \emptyset$.}
\label{table:x}
\end{table}

In the following, we prove that there are no $(\{T_7(i)\},\bigcup_{2\leq i\leq 26}\{T_7(i)\})$-{\teq}-retentive tournaments for every positive integer $i\in \{18,19,20,22,23,24,25,26\}$.

\begin{lemma}
\label{lem:againsttwoteqonetseveneighteenonecaptainnotempty}
Let $T$ be a tournament, and $R_1$ and $R_2$ be two vertex subsets of $T$ such that $R_1\cap R_2=\emptyset$. If the following conditions hold, $R_1$ and $R_2$ cannot be minimal {\teq}-retentive sets of $T$ simultaneously.

(1) $T[R_1]$ is isomorphic to one of the tournaments in $\{T_7(18),T_7(19),T_7(20),T_7(22),T_7(23),T_7(24),T_7(25),T_7(26)\}$ in Appendix; and

(2) $\dcapgraph{T[R_2]}$ is not empty.
\end{lemma}

\begin{proof}
Assume for the sake of contradiction that the above two conditions hold, and $R_1$ and $R_2$ are both minimal {\teq}-retentive sets of $T$. We break down the proof of the lemma into 7 cases, based on which tournament $T_7(i)$, $i\in \{18,19,20,22,23,24,25,26\}$, $T[R_1]$ is isomorphic to. In each case, for ease of exposition, assume that $T_7(i)=T[R_1]$; thus, $R_1=\{a,b,c,d,e,f,g\}$. Moreover, let the vertices in $R_1$ be labeled as in the tournament $T_7(i)$ in Appendix. Furthermore, let $v\trelation u$ be an arbitrary arc in $\dcapgraph{T[R_2]}$.

\bigskip
{\casebegin}{\bf{$T[R_1]$ is isomorphic to $T_7(18)$}}

We complete the proof by distinguishing between the following two cases. Before proceeding further, let's study some useful properties. First, it is easy to check that in $T[R_1]$, $\{e,f,g\}$ is a tri-captain of $a$, $\{e,f,g\}$ is a tri-captain of $b$, $\{a,b,d\}$ is a tri-captain of $c$, and $\{c,d,g\}$ is a tri-captain of $e$ (see the tournament $T_7(18)$ in Appendix).
Moreover, each of $\{a,b,c,e\}$ has no more than five inneighbors in $R_1$. Then,
according to Lemma~\ref{lem:localboundedsmallteqtournament}, it holds that

$\{e,f,g\}=\mathteq{(T[N^-_{T}(a)])}$,

$\{e,f,g\}=\mathteq{(T[N^-_{T}(b)])}$,

$\{a,b,d\}=\mathteq{(T[N^-_{T}(c)])}$, and

$\{c,d,g\}=\mathteq{(T[N^-_{T}(e)])}$.

\caseinproofskip
Case~1. $u\trelation e$.

Since $\{c,d,g\}=\mathteq{(T[N^-_{T}(e)])}$, according to Lemma~\ref{lem:tauthreeequalretentive} and Lemma~\ref{lem:t3minimalretentiveset}, at least two of $\{c,d,g\}$ dominate $u$. Since $c\trelation g$ is an arc in $\dcapgraph{T[R_1]}$, according to Lemma~\ref{lem:teqisomorhpictoT43}, $[c,u]\not\asymp [g,u]$. Thus, at most one of $\{c,g\}$ dominates $u$. This implies that $d\trelation u$. Then, according to Lemma~\ref{lem:captaingrapht43property}, it holds that $u \trelation f, c\trelation u$ and $u\trelation g$. Moreover, $v\trelation d, f\trelation v$ and $v\trelation c$. Since $\{a,b,d\}=\mathteq{(T[N^-_{T}(c)])}$, according to Lemma~\ref{lem:tauthreeequalretentive} and Lemma~\ref{lem:t3minimalretentiveset}, at least two of $\{a,b,d\}$ dominate $v$. Since we have showed above that $v\trelation d$, it holds that $\{a,b\}\trelation \{v\}$. Then, according to Lemma~\ref{lem:novertexdominatestwoadjacentverticesincapgraingraphs}, it holds that $\{u\}\trelation \{a,b\}$. Since $\{e,f,g\}=\mathteq{(T[N^-_{T}(a)])}=\mathteq{(T[N^-_{T}(b)])}$, according to Lemma~\ref{lem:tauthreeequalretentive} and Lemma~\ref{lem:t3minimalretentiveset}, at least two of $\{e,f,g\}$ dominate $u$. However, we have showed above that $\{u\}\trelation \{e,f,g\}$; a contradiction.

\caseinproofskip
Case~2. $e\trelation u$.

According to Lemma~\ref{lem:novertexdominatestwoadjacentverticesincapgraingraphs}, it holds that $v\trelation e$. Then, we can achieve a contradiction by exchanging all appearances of $v$ and all appearances of $u$ in the above proof for Case~1.
%
%
%
%
%
%
%
%

\bigskip
{\casebegin}{\bf{$T[R_1]$ is isomorphic to $T_7(19)$}}

We complete the proof by distinguishing between the following two cases. Before proceeding further, let's first study some useful properties. First, it is easy to verify that in the subtournament $T[R_1]$,
 $\{c,f,g\}$ is a tri-captain of $a$,
$\{b,e,g\}$ is a tri-captain of $c$,
$\{a,b,f\}$ is a tri-captain of $e$, and
$\{b,c,d\}$ is a tri-captain of $f$.
Moreover, each of $\{a,c,f,e\}$ has no more than five inneighbors in $R_1$.
Then according to Lemma~\ref{lem:localboundedsmallteqtournament}, it holds that

$\{c,f,g\}=\mathteq{(T[N^-_{T}(a)])}$,

$\{b,e,g\}=\mathteq{(T[N^-_{T}(c)])}$,

$\{a,b,f\}=\mathteq{(T[N^-_{T}(e)])}$,

$\{b,c,d\}=\mathteq{(T[N^-_{T}(f)])}$.

\caseinproofskip
Case~1. $u\trelation a$.

Since $\{c,f,g\}=\mathteq{(T[N^-_{T}(a)])}$, according to Lemma~\ref{lem:tauthreeequalretentive} and Lemma~\ref{lem:t3minimalretentiveset}, at least two of $\{c,f,g\}$ dominate $u$. According to Lemma~\ref{lem:teqisomorhpictoT43}, $[f,u]\not\asymp [g,u]$. Thus, at most one of $\{f,g\}$ dominates $u$. This implies that $c\trelation u$. Then, according to Lemma~\ref{lem:novertexdominatestwoadjacentverticesincapgraingraphs}, it holds that $v\trelation c$. Since $\{b,e,g\}=\mathteq{(T[N^-_{T}(c)])}$, according to Lemma~\ref{lem:tauthreeequalretentive} and Lemma~\ref{lem:t3minimalretentiveset}, at least two of $\{b,e,g\}$ dominate $v$. According to Lemma~\ref{lem:teqisomorhpictoT43}, $[g,v]\not\asymp [b,v]$. Thus, at most one of $\{g,b\}$ dominates $v$. This implies that $e\trelation v$. Then, according to Lemma~\ref{lem:teqisomorhpictoT43}, it holds that $v\trelation d, d\trelation u$ and $u\trelation e$. Since $\{a,b,f\}=\mathteq{(T[N^-_{T}(e)])}$, according to Lemma~\ref{lem:tauthreeequalretentive} and Lemma~\ref{lem:t3minimalretentiveset}, at least two of $\{a,b,f\}$ dominate $u$. Since we have showed above that $u\trelation a$, it holds that $\{b,f\}\trelation \{u\}$. Then, according to Lemma~\ref{lem:novertexdominatestwoadjacentverticesincapgraingraphs}, $\{v\}\trelation \{b,f\}$. Since $\{b,c,d\}=\mathteq{(T[N^-_{T}(f)])}$, according to Lemma~\ref{lem:tauthreeequalretentive} and Lemma~\ref{lem:t3minimalretentiveset}, at least two of $\{b,c,d\}$ dominate $v$. However, we have showed above that $\{v\}\trelation \{b,d\}$; a contradiction.

\caseinproofskip
Case~2. $a\trelation u$.

According to Lemma~\ref{lem:novertexdominatestwoadjacentverticesincapgraingraphs}, it holds that $v\trelation a$. Then, we can achieve a contradiction by exchanging all appearances of $v$ and $u$ in the above proof for Case~1.
%
%
%
%
%
%

\bigskip
{\casebegin}{\bf{$T[R_1]$ is isomorphic to $T_7(20)$}}

We complete the proof by distinguishing between the following two cases. Before proceeding further, let's first study some useful properties. First, it is easy to check that in the subtournament $T[R_1]$,
 $\{b,f,g\}=N^-_{T[R_1]}(c)$ is a tri-captain of $c$,
$\{c,e,g\}=N^-_{T[R_1]}(d)$ is a tri-captain of $d$,
$\{a,b,c\}=N^-_{T[R_1]}(e)$ is a tri-captain of $e$, and
$\{b,d,e\}=N^-_{T[R_1]}(f)$ is a tri-captain of $f$.
Then according to Lemma~\ref{lem:localboundedsmallteqtournament}, it holds that

$\{b,f,g\}=\mathteq{(T[N^-_{T}(c)])}$,

$\{c,e,g\}=\mathteq{(T[N^-_{T}(d)])}$,

$\{a,b,c\}=\mathteq{(T[N^-_{T}(e)])}$, and

$\{b,d,e\}=\mathteq{(T[N^-_{T}(f)])}$.

\caseinproofskip
Case~1. $u\trelation d$.

Since $\{c,e,g\}=\mathteq{(T[N^-_{T}(d)])}$, according to Lemma~\ref{lem:tauthreeequalretentive} and Lemma~\ref{lem:t3minimalretentiveset}, at least two of $\{c,e,g\}$ dominate $u$. According to Lemma~\ref{lem:teqisomorhpictoT43}, $[e,u]\not\asymp [g,u]$. Thus, at most one of $\{e,g\}$ dominates $u$. This implies that $c\trelation u$. Then, according to Lemma~\ref{lem:novertexdominatestwoadjacentverticesincapgraingraphs}, $v\trelation c$. Since $\{b,f,g\}=\mathteq{(T[N^-_{T}(c)])}$, according to Lemma~\ref{lem:tauthreeequalretentive} and Lemma~\ref{lem:t3minimalretentiveset}, at least two of $\{b,f,g\}$ dominate $v$. According to Lemma~\ref{lem:teqisomorhpictoT43}, $[g,v]\not\asymp [b,v]$. Thus, at most one of $\{g,b\}$ dominates $v$. This implies that $f\trelation v$. Then, according to Lemma~\ref{lem:novertexdominatestwoadjacentverticesincapgraingraphs}, $u\trelation f$. Since $\{b,d,e\}=\mathteq{(T[N^-_{T}(f)])}$, according to Lemma~\ref{lem:tauthreeequalretentive} and Lemma~\ref{lem:t3minimalretentiveset}, at least two of $\{b,d,e\}$ dominates $u$. Since $u\trelation d$, it holds that $\{b,e\}\trelation \{u\}$. Then, according to Lemma~\ref{lem:novertexdominatestwoadjacentverticesincapgraingraphs}, it holds that $\{v\}\trelation \{e,b\}$. Since $\{a,b,c\}=\mathteq{(T[N^-_{T}(e)])}$, according to Lemma~\ref{lem:t3minimalretentiveset}, at least two of $\{a,b,c\}$ dominate $v$. However, we have showed above that $\{v\}\trelation \{b,c\}$; a contradiction.
\caseinproofskip

Case~2. $d\trelation u$.

According to Lemma~\ref{lem:novertexdominatestwoadjacentverticesincapgraingraphs}, it holds that $v\trelation d$. Then, we can achieve a contradiction by exchanging all appearances of $v$ and $u$ in the above proof for Case~1.
%
%
%
%
%
%

\bigskip
{\casebegin}{\bf{$T[R_1]$ is isomorphic to $T_7(22)$}}

We complete the proof by distinguishing between the following two cases. Before proceeding further, let's study some useful properties. First, it is easy to verify that in the subtournament $T[R_1]$,
 $\{d,f,g\}$ is a tri-captain of $a$,
$\{a,e,g\}$ is a tri-captain of $b$,
$\{b,f,g\}$ is a tri-captain of $c$, and
$\{b,d,e\}$ is a tri-captain of $f$.
Moreover, each of $\{a,b,c,f\}$ has no more than five inneighbors in $R_1$.
Then according to Lemma~\ref{lem:localboundedsmallteqtournament}, it holds that

$\{d,f,g\}=\mathteq{(T[N^-_{T}(a)])}$,

$\{a,e,g\}=\mathteq{(T[N^-_{T}(b)])}$,

$\{b,f,g\}=\mathteq{(T[N^-_{T}(c)])}$, and

$\{b,d,e\}=\mathteq{(T[N^-_{T}(f)])}$.

\caseinproofskip
Case~1. $u\trelation b$.

Since $\{a,e,g\}=\mathteq{(T[N^-_{T}(b)])}$, according to Lemma~\ref{lem:tauthreeequalretentive} and Lemma~\ref{lem:t3minimalretentiveset}, at least two of $\{a,e,g\}$ dominate $u$. According to Lemma~\ref{lem:teqisomorhpictoT43}, $[e,u]\not\asymp [g,u]$. Thus, at most one of $\{e,g\}$ dominates $u$. This implies that $a\trelation u$. Then, according to Lemma~\ref{lem:novertexdominatestwoadjacentverticesincapgraingraphs}, $v\trelation a$. Since $\{d,f,g\}=\mathteq{(T[N^-_{T}(a)])}$, according to Lemma~\ref{lem:tauthreeequalretentive} and Lemma~\ref{lem:t3minimalretentiveset}, at least two of $\{d,f,g\}$ dominate $v$. According to Lemma~\ref{lem:teqisomorhpictoT43}, $[g,v]\not\asymp [d,v]$. Thus, at most one of $\{g,d\}$ dominates $v$. This implies that $f\trelation v$. Then, according to Lemma~\ref{lem:novertexdominatestwoadjacentverticesincapgraingraphs}, $u\trelation f$. Since $\{b,d,e\}=\mathteq{(T[N^-_{T}(f)])}$, according to Lemma~\ref{lem:tauthreeequalretentive} and Lemma~\ref{lem:t3minimalretentiveset}, at least two of $\{b,d,e\}$ dominates $u$. Since $u\trelation b$, it holds that $\{d,e\}\trelation \{u\}$. Then, according to Lemma~\ref{lem:teqisomorhpictoT43}, it holds that $u\trelation c$. Since $\{b,f,g\}=\mathteq{(T[N^-_{T}(c)])}$, according to Lemma~\ref{lem:tauthreeequalretentive} and Lemma~\ref{lem:t3minimalretentiveset}, at least two of $\{b,f,g\}$ dominate $u$. However, we have showed above that $\{u\}\trelation \{b,f\}$; a contradiction.
\caseinproofskip

Case~2. $b\trelation u$.

According to Lemma~\ref{lem:novertexdominatestwoadjacentverticesincapgraingraphs}, it holds that $v\trelation b$. Then, we can achieve a contradiction by exchanging all appearances of $v$ and $u$ in the above proof for Case~1.
%
%
%
%
%
%
%

\bigskip
{\casebegin}{\bf{$T[R_1]$ is isomorphic to $T_7(23)$}}

We complete the proof by distinguishing between the following two cases. Before proceeding further, let's first observe some useful properties. First, it is easy to verify that in the subtournament $T[R_1]$,
$\{b,e,g\}=N^-_{T[R_1]}(c)$ is a tri-captain of $c$,
$\{a,b,d\}=N^-_{T[R_1]}(e)$ is a tri-captain of $e$, and
$\{c,d,e\}=N^-_{T[R_1]}(f)$ is a tri-captain of $f$.
Then, according to Lemma~\ref{lem:localboundedsmallteqtournament}, it holds that

$\{b,e,g\}=\mathteq{(T[N^-_{T}(c)])}$,

$\{a,b,d\}=\mathteq{(T[N^-_{T}(e)])}$,

$\{c,d,e\}=\mathteq{(T[N^-_{T}(f)])}$.

\caseinproofskip
Case~1. $u\trelation c$.

Since $\{b,e,g\}=\mathteq{(T[N^-_{T}(c)])}$, according to Lemma~\ref{lem:tauthreeequalretentive} and Lemma~\ref{lem:t3minimalretentiveset}, at least two of $\{b,e,g\}$ dominate $u$. According to Lemma~\ref{lem:teqisomorhpictoT43}, $[e,u]\not\asymp [g,u]$. Thus, at most one of $\{e,g\}$ dominates $u$. This implies that $b\trelation u$. Then, according to Lemma~\ref{lem:teqisomorhpictoT43}, it holds that $u\trelation f, f\trelation v$ and $v\trelation b$. Since $\{c,d,e\}=\mathteq{(T[N^-_{T}(f)])}$, according to Lemma~\ref{lem:tauthreeequalretentive} and Lemma~\ref{lem:t3minimalretentiveset}, at least two of $\{c,d,e\}$ dominate $u$. Since $u\trelation c$, it holds that $\{d,e\}\trelation \{u\}$. Then, according to Lemma~\ref{lem:novertexdominatestwoadjacentverticesincapgraingraphs}, $\{v\}\trelation \{d,e\}$. Since $\{a,b,d\}=\mathteq{(T[N^-_{T}(e)])}$, according to Lemma~\ref{lem:tauthreeequalretentive} and Lemma~\ref{lem:t3minimalretentiveset}, at least two of $\{a,b,d\}$ dominate $v$. However, we have showed above that $\{v\}\trelation \{b,d\}$; a contradiction.

\caseinproofskip
Case~2. $c\trelation u$.

According to Lemma~\ref{lem:novertexdominatestwoadjacentverticesincapgraingraphs}, it holds that $v\trelation c$. Then, we can achieve a contradiction by exchanging all appearances of $v$ and $u$ in the above proof for Case~1.
%
%
%
%
%
%
%
%
%

\bigskip
{\casebegin}{\bf{$T[R_1]$ is isomorphic to $T_7(24)$}}

We complete the proof by distinguishing between the following two cases. It is easy to verify that in the subtournament $T[R_1]$,
 $\{a,b,c\}$ is a tri-captain of $d$, and $\{g,e,f\}$ is a tri-captain of $a$, $b$ and $c$.
Moreover, each of $\{a,b,c,d\}$ has at most five inneighbors in $R_1$.
Then according to Lemma~\ref{lem:localboundedsmallteqtournament}, it holds that

$\{a,b,c\}=\mathteq{(T[N^-_{T}(d)])}$ and  $\{g,e,f\}=\mathteq{(T[N^-_{T}(a)])}=\mathteq{(T[N^-_{T}(b)])}=\mathteq{(T[N^-_{T}(c)])}$.

\caseinproofskip
Case~1. $u\trelation d$.

According to Lemma~\ref{lem:teqisomorhpictoT43}, it holds that $d\trelation v, \{v\}\trelation \{e,f,g\}$ and $ \{e,f,g\}\trelation \{u\}$. Then, it must hold that $\{a,b,c\}\trelation \{v\}$, since otherwise, according to Lemma~\ref{lem:tauthreeequalretentive} and Lemma~\ref{lem:t3minimalretentiveset}, at least two of $\{e,f,g\}$ should have dominated $v$. Then, due to Lemma~\ref{lem:novertexdominatestwoadjacentverticesincapgraingraphs}, it holds that $\{u\}\trelation \{a,b,c\}$. However, since $\{a,b,c\}=\mathteq{(T[N^-_{T}(d)])}$ and $u\trelation d$, according to Lemma~\ref{lem:tauthreeequalretentive} and Lemma~\ref{lem:t3minimalretentiveset}, at least two of $\{a,b,c\}$ should have dominated $u$; a contradiction.
\caseinproofskip

Case~2. $d\trelation u$.

According to Lemma~\ref{lem:teqisomorhpictoT43}, it holds that $v\trelation d$. Then, we can achieve a contradiction by exchanging all appearances of $v$ and $u$ in the above proof for Case~1.
%
%
%
%
%
%
%

\bigskip
{\casebegin}{\bf{$T[R_1]$ is isomorphic to $T_7(25)$}}

We complete the proof by distinguishing between the following cases with respect to the arcs between $\{u\}$ and $\{c,d\}$. Before proceeding further, let's observe some useful properties. First, it is easy to verify that in the subtournament $T[R_1]$,
 $\{a,c,f\}=N^-_{T[R_1]}(d)$ is a tri-captain of $d$,
$\{a,b,e\}=N^-_{T[R_1]}(c)$ is a tri-captain of $c$,
$\{a,d,g\}=N^-_{T[R_1]}(b)$ is a tri-captain of $b$, and
$\{e,f,g\}=N^-_{T[R_1]}(a)$ is a tri-captain of $a$.
Then, according to Lemma~\ref{lem:localboundedsmallteqtournament}, it holds that

$\{a,c,f\}=\mathteq{(T[N^-_{T}(d)])}$,

$\{a,b,e\}=\mathteq{(T[N^-_{T}(c)])}$,

$\{a,d,g\}=\mathteq{(T[N^-_{T}(b)])}$, and

$\{e,f,g\}=\mathteq{(T[N^-_{T}(a)])}$.
\caseinproofskip

Case~1. $u\trelation d$ and $u\trelation c$.

According to Lemma~\ref{lem:teqisomorhpictoT43}, $\{v\}\trelation \{g,e\}$. Since $\{a,c,f\}=\mathteq{(T[N^-_{T}(d)])}$ and $u\trelation d$, then according to Lemma~\ref{lem:tauthreeequalretentive} and Lemma~\ref{lem:t3minimalretentiveset}, at least two of $\{a,c,f\}$ dominate $u$. Since $u\trelation c$, it holds that $\{f,a\}\trelation \{u\}$. Due to Lemma~\ref{lem:novertexdominatestwoadjacentverticesincapgraingraphs}, $v\trelation a$. Since $\{e,f,g\}=\mathteq{(T[N^-_{T}(a)])}$, according to Lemma~\ref{lem:tauthreeequalretentive} and Lemma~\ref{lem:t3minimalretentiveset}, at least two of $\{e,f,g\}$ dominate $v$. However, we have just showed above that $\{v\}\trelation \{g,e\}$; a contradiction.
\caseinproofskip

Case~2. $u\trelation d$ and $c\trelation u$.

According to Lemma~\ref{lem:teqisomorhpictoT43}, it holds that
$\{d,g\}\trelation \{v\}$, $\{v\}\trelation \{c,e\}$, $e\trelation u$ and $u\trelation g$.
Since $\{a,b,e\}=\mathteq{(T[N^-_{T}(c)])}$ and $v\trelation c$, according to Lemma~\ref{lem:tauthreeequalretentive} and Lemma~\ref{lem:t3minimalretentiveset}, at least two of $\{a,b,e\}$ dominate $v$. Since we have just showed that $v\trelation e$, it holds that $\{a,b\}\trelation \{v\}$. Then, according to Lemma~\ref{lem:novertexdominatestwoadjacentverticesincapgraingraphs}, we have that $u\trelation a$. Since $\{e,f,g\}=\mathteq{(T[N^-_{T}(a)])}$, according to Lemma~\ref{lem:tauthreeequalretentive} and Lemma~\ref{lem:t3minimalretentiveset}, at least two of $\{e,f,g\}$ dominate $u$. Since we have showed that $u\trelation g$ and $e\trelation u$, it holds that $f\trelation u$. Then, according to Lemma~\ref{lem:teqisomorhpictoT43}, $u\trelation b$. Since $\{a,d,g\}=\mathteq{(T[N^-_{T}(b)])}$, according to Lemma~\ref{lem:tauthreeequalretentive} and Lemma~\ref{lem:t3minimalretentiveset}, at least two of $\{a,d,g\}$ dominate $u$. However, we have showed above that $\{u\}\trelation \{g,d\}$; a contradiction.
\caseinproofskip

Case~3. $d\trelation u$ and $u\trelation c$.

According to Lemma~\ref{lem:teqisomorhpictoT43}, it holds that
$\{e,c\}\trelation \{v\}$, $\{v\}\trelation \{d,g\}$, $g\trelation u$ and $u\trelation e$.
Since $\{a,b,e\}=\mathteq{(T[N^-_{T}(c)])}$, according to Lemma~\ref{lem:tauthreeequalretentive} and Lemma~\ref{lem:t3minimalretentiveset}, at least two of $\{a,b,e\}$ dominate $u$. Since we have just showed that $u\trelation e$, it holds that $\{a,b\}\trelation \{u\}$. Then, according to Lemma~\ref{lem:novertexdominatestwoadjacentverticesincapgraingraphs}, we have that $v\trelation a$. Since $\{e,f,g\}=\mathteq{(T[N^-_{T}(a)])}$, according to Lemma~\ref{lem:tauthreeequalretentive} and Lemma~\ref{lem:t3minimalretentiveset}, at least two of $\{e,f,g\}$ dominate $v$. Since we have showed that $v\trelation g$ and $e\trelation v$, it holds that $f\trelation v$. Then, according to Lemma~\ref{lem:teqisomorhpictoT43}, $v\trelation b$. Since $\{a,d,g\}=\mathteq{(T[N^-_{T}(b)])}$, according to Lemma~\ref{lem:tauthreeequalretentive} and Lemma~\ref{lem:t3minimalretentiveset}, at least two of $\{a,d,g\}$ dominate $v$. However, we have showed above that $\{v\}\trelation \{d,g\}$; a contradiction.
\caseinproofskip

Case~4. $d\trelation u$ and $c\trelation u$.

According to Lemma~\ref{lem:teqisomorhpictoT43}, it holds that
$\{e,g\}\trelation \{v\}$, $\{v\}\trelation \{c,d\}$, and $\{u\}\trelation \{e,g\}$.
Since $\{a,c,f\}=\mathteq{(T[N^-_{T}(d)])}$ and $v\trelation d$, according to Lemma~\ref{lem:tauthreeequalretentive} and Lemma~\ref{lem:t3minimalretentiveset}, at least two of $\{a,c,f\}$ dominate $v$. Since $v\trelation c$, it holds that $\{a,f\}\trelation \{v\}$.
Then, according to Lemma~\ref{lem:novertexdominatestwoadjacentverticesincapgraingraphs}, we have that $u\trelation a$. Since $\{e,f,g\}=\mathteq{(T[N^-_{T}(a)])}$, according to Lemma~\ref{lem:tauthreeequalretentive} and Lemma~\ref{lem:t3minimalretentiveset}, at least two of $\{e,f,g\}$ dominate $u$. However, we have showed above that $\{u\}\trelation \{g,e\}$; a contradiction.

\bigskip
{\casebegin}{\bf{$T[R_1]$ is isomorphic to $T_7(26)$}}

Observe that there is a directed cycle $(a,d,e,c,f,b,g)$ in $\dcapgraph{T_7(26)}$. If $\dcapgraph{T[R_2]}$ is not empty, then due to Lemma~\ref{lem:captaingraphagainsttwoteq}, $R_1$ and $R_2$ cannot be both minimal {\teq}-retentive sets of $T$.
\end{proof}

The summary of the above lemmas clearly proves Theorem~\ref{thm:notsevenseventeqtournaments}.
Furthermore, due to Theorems~\ref{thm:notthreekleqtwelveteqtournaments}-\ref{thm:notsevenseventeqtournaments},
we have the following corollary.

\begin{corollary}
Schwartz's Conjecture holds in all tournaments of size at most $14$.
\end{corollary}

\section{Schwartz's Conjecture}
We have derived numerous properties of {\teq}-retentive sets in the previous sections, and achieved many interesting results regarding {\teq}-retentive tournaments with the help of these properties. In this section, we extend the applications of these properties by examining Schwartz's Conjecture in several classes of tournaments.
A tournament is {\it{locally-transitive}} if the outneighbourhood
and the inneighbourhood of every vertex are transitive~\cite{DBLP:journals/combinatorics/BabaiC00,CohenPPlocalltournaments}. Clearly, every vertex in a locally transitive tournament has a captain---the source in the subtournament induced by its inneighbors.


\begin{theorem}
\label{thm:schwartzholdsinlocallyboundedtournaments}
Schwartz's Conjecture holds in all locally-transitive tournaments.
\end{theorem}
\begin{proof}
We prove the theorem by contradiction. Assume that a locally-transitive tournament $T$ has two minimal {\teq}-retentive sets $R_1$ and $R_2$.
It is clear that both $T[R_1]$ and $T[R_2]$ are locally-transitive. Now consider the directed domination graphs $\dcapgraph{T[R_1]}$ and $\dcapgraph{T[R_2]}$. Since $T[R_1]$ and $T[R_2]$ are locally transitive, for every vertex $v\in R_1$ (resp. $v'\in R_2$), $v$ (resp. $v'$) has a captain in $T[R_1]$ (resp. $T[R_2]$). This implies that there is a directed cycle $C_1$ in $\dcapgraph{T[R_1]}$ and a directed cycle $C_2$ in $\dcapgraph{T[R_2]}$. Since $R_1$ and $R_2$ are both minimal {\teq}-retentive sets of $T$, according to Lemma~\ref{lem:inneighboronesouce}, if a vertex $a\in R_i$ ($i=1,2$) is the captain of a vertex $b\in R_i$ in $T[R_i]$, the vertex a is still the captain of the vertex $b$ in $T$. Therefore, both $C_1$ and $C_2$ exist in $\dcapgraph{T}$. However, this contradicts with Lemma~\ref{lem:dominationgraphproperty}.
\end{proof}

\begin{theorem}
\label{thm:captaincycleminimalteqretentiveset}
Schwartz's Conjecture holds in all tournaments $T$ such that $\dcapgraph{T}$ contains a directed cycle $C$ with $V(C)=V(T)$.
\end{theorem}
\begin{proof}
Let $C=(v_0,v_1,...,v_{n-1})$. We claim that $V(C)$ is the unique minimal {\teq}-retentive set of $T$. Since $v_i$ is the captain of $v_{\mymod{(i+1)}{n}}$ in $T$ for all $i=0,1,...,n-1$, due to Lemma~\ref{lem:captainistheuniqueteqofslavery}, it holds $\{v_i\}=\mathteq{(T[N^-_T(v_{\mymod{(i+1)}{n}})])}$ for every $i=0,1,...,n-1$. Therefore, $V(C)$ is a {\teq}-retentive set of $T$. According to Lemma~\ref{lem:captainistheuniqueteqofslavery}, if a vertex $v_{\mymod{(i+1)}{n}}$ is in a minimal {\teq}-retentive set of $T$, so does its captain $v_{i}$. Therefore, $V(C)$ is the unique minimal {\teq}-retentive set of $T$.
\end{proof}



\section{Conclusion}
The tournament equilibrium set is one of the most attractive tournament solutions. Despite that Schwartz's Conjecture does not hold in general~\cite{DBLP:journals/scw/BrandtCKLNSST13,Brandt2013b}, researchers believe that the tournament equilibrium set satisfies many fairness properties for most of the practical purposes~\cite{handbookofcomsoc2015Cha3Brandt,Brandt2010,MSYangIJCAI2015}. In this paper, we proposed and studied several significant problems pertaining to the tournament equilibrium set, aiming at deriving properties that are useful for a comprehensive understanding of the tournament equilibrium set. Our main contributions are summarized as follows:
(1) We introduced {\captaingraph}s into the study of the tournament equilibrium set, and showed that domination graphs of tournaments are very useful in the study.
(2) We initiated the study of {\teq}-retentive tournaments and $(H_1,H_2)$-{\teq}-retentive tournaments. This study is helpful for researchers to have a comprehensive understanding of the intrinsic structures of minimal {\teq}-retentive sets. Moreover, the study of $(H_1,H_2)$-{\teq}-retentive tournaments is useful in checking Schwartz's Conjecture on special classes of tournaments.
(3) We gave upper bounds for $\beta_n$ for $4\leq n\leq 10$, the number of all non-isomorphic {\teq}-retentive tournaments of size $n$. In particular, we showed that $\beta_4=0, \beta_5=\beta_6=2$ and $\beta_7=26$. Our findings tell that minimal {\teq}-retentive sets of small size are considerably rare. 
(4) Our results on $(H_1,H_2)$-{\teq}-retentive tournaments imply that Schwartz's Conjecture holds in all tournaments of size at most 14. Moreover, we proved that Schwartz's Conjecture holds in several other classes of tournaments; adding further positive evidence for the practicability of the tournament equilibrium set.
(5) We derived a lot of properties of the tournament equilibrium set, which we believe are very helpful in further research on the tournament equilibrium set.

Another major contribution of this paper is that we uncover many new significant problems related to the tournament equilibrium set. We believe that solutions to these problems will bring us new technique tools for the study of the tournament equilibrium set. In the following, we discuss a few of them.

We have proved that there are no $(T_3,T_k)$-{\teq}-retentive tournaments for every positive integer $k\leq 12$ (see Theorem~\ref{thm:notthreekleqtwelveteqtournaments}). We conjecture that this holds for every positive integer $k$, that is, if a tournament has a minimal {\teq}-retentive set of size 3, then it has no further minimal {\teq}-retentive sets.
\begin{conjecture}\label{conjec:triangleunique}
If a tournament $T$ has a minimal {\teq}-retentive set $R$ of size 3, then $R$ is the unique minimal {\teq}-retentive set of $T$.
\end{conjecture}

We showed that if the directed {\captaingraph} of a tournament has a directed cycle $C$ such that $V(C)=V(T)$ then $V(C)$ is the unique minimal {\teq}-retentive set of $T$ (see Theorem~\ref{thm:captaincycleminimalteqretentiveset}).
We conjecture that this is true if we drop the condition that $V(C)=V(T)$.

\begin{conjecture}\label{conjec:captaincycleuniqueminimalteqset}
Let $T$ be a tournament. If there is a directed cycle $C$ in $\dcapgraph{T}$, then $V(C)$ is the unique {\teq}-minimal retentive set of $T$.
\end{conjecture}

Since a minimal {\teq}-retentive set of size $3$ induces a directed cycle in the directed {\captaingraph}, Conjecture~\ref{conjec:captaincycleuniqueminimalteqset} implies Conjecture~\ref{conjec:triangleunique}.

According to Lemma~\ref{lem:teqsourceuniqueminial}, a singleton set is a (the unique) minimal {\teq}-retentive set of a tournament if and only if the vertex in the set dominates every other vertex. A captain vertex together with any of its slaves dominate all the other vertices. This encourages us to propose the following conjecture.

\begin{conjecture}\label{conjec:captaingraph}
Let $T$ be a tournament and $A$ be the set of all captain vertices of $T$. Then $A\subseteq \mathteq{(T)}$.
\end{conjecture}

\section*{Acknowledgement}
The author would like to thank the anonymous reviewers of AAAI 2016 and ECAI 2016 for their constructive comments.

\bibliographystyle{plain}

\newpage
\section*{Appendix}
This appendix shows all the 26 non-isomorphic {\teq}-retentive tournaments of size $7$, as well as their directed domination graphs.
In the domination  graphs, if a vertex has a tri-captain in the tournament, its tri-captain is labeled next to the vertex.
Observe that all vertices which have tri-captains have indegree at most 5 in all these 26 tournaments.
Therefore, according to Lemma~\ref{lem:localboundedsmallteqtournament}, if a tournament $T$ has a minimal {\teq}-retentive set $R$ such that $T[R]$ is isomorphic to one of these 26 tournaments, then the tri-captain showed in the appendix of a vertex $v\in R$ is the unique minimal {\teq}-retentive set of the subtournament induced by all inneighbors of $v$ in $T$.
\label{sec:appb}
\begin{center}
\includegraphics[width=0.45\textwidth]{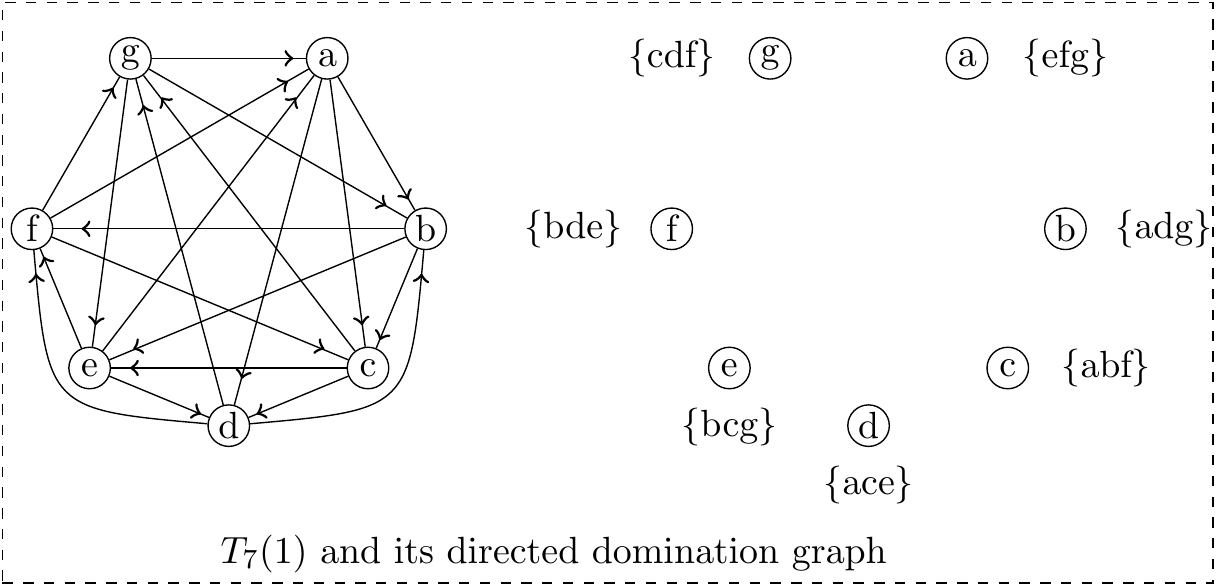}
\includegraphics[width=0.45\textwidth]{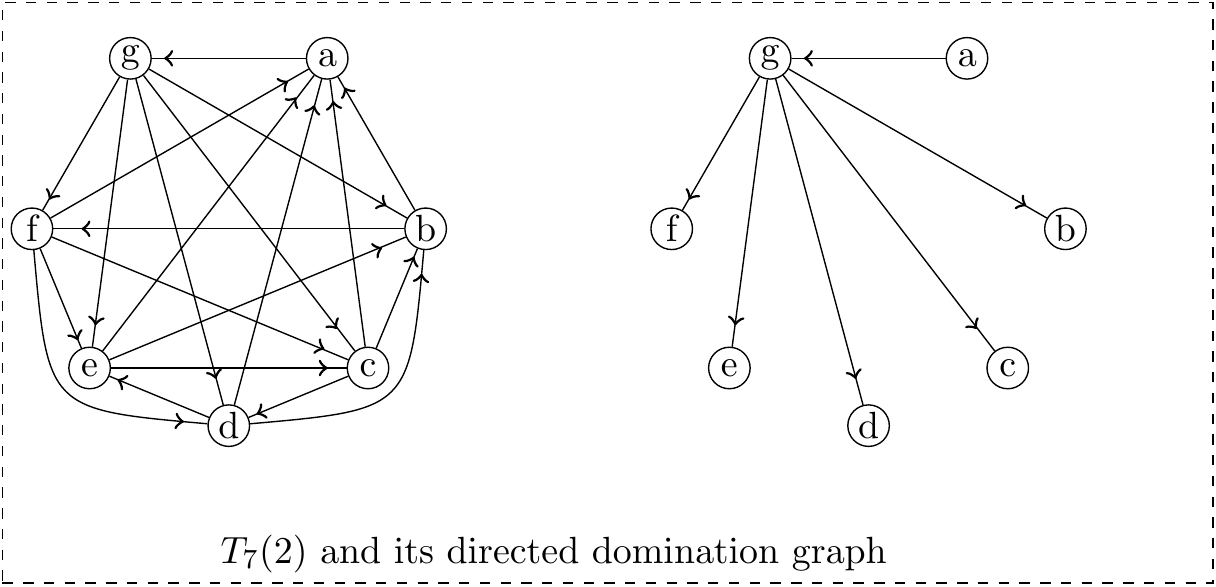}
\includegraphics[width=0.45\textwidth]{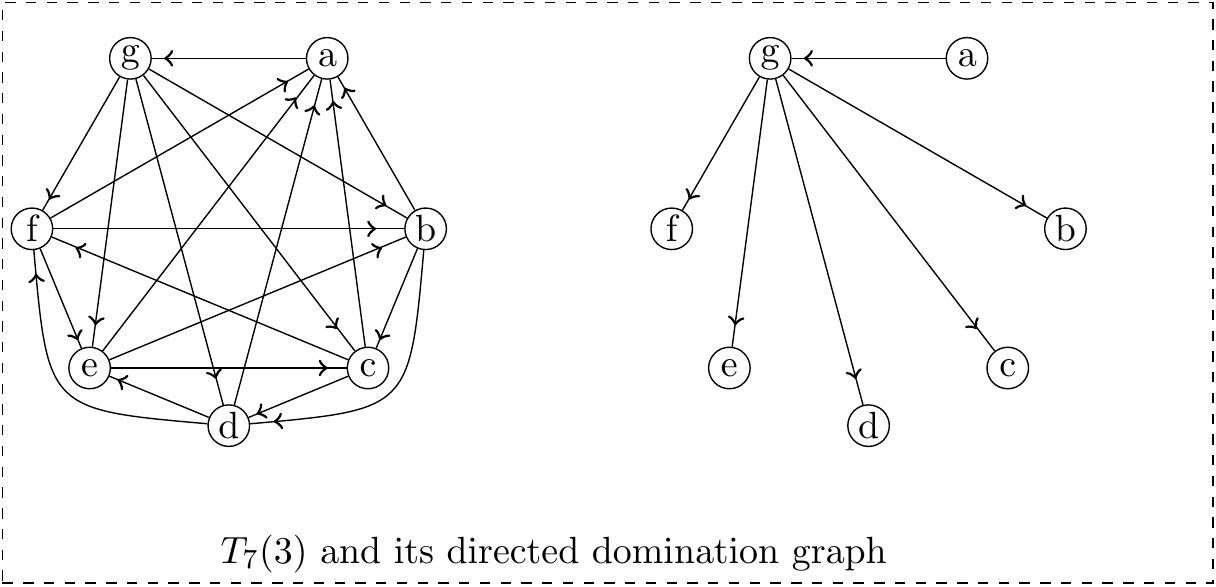}
\includegraphics[width=0.45\textwidth]{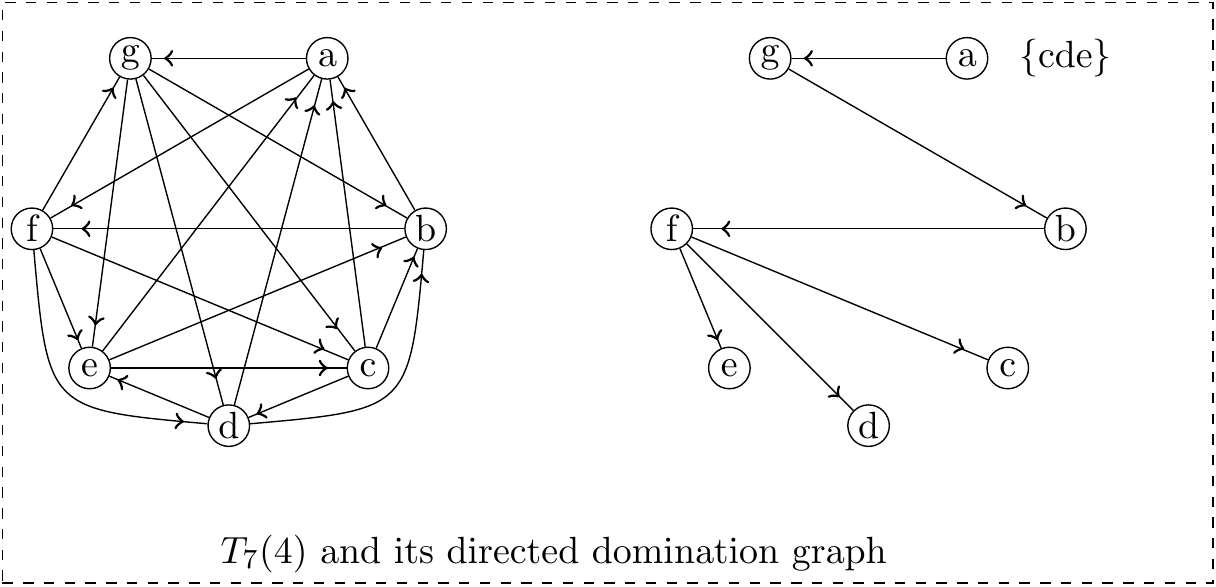}
\includegraphics[width=0.45\textwidth]{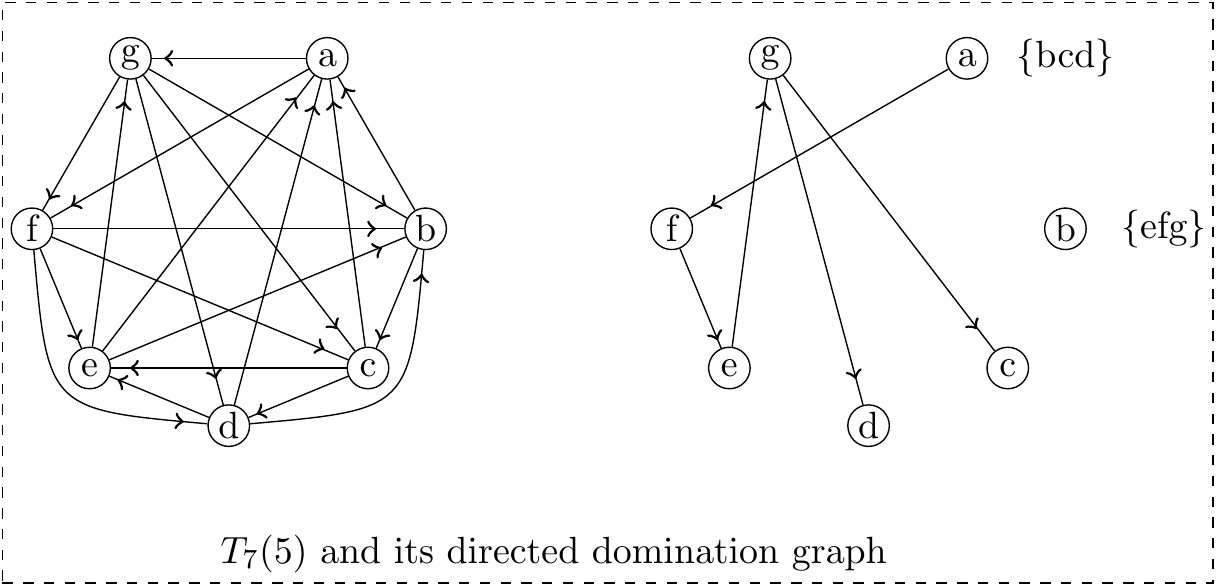}
\includegraphics[width=0.45\textwidth]{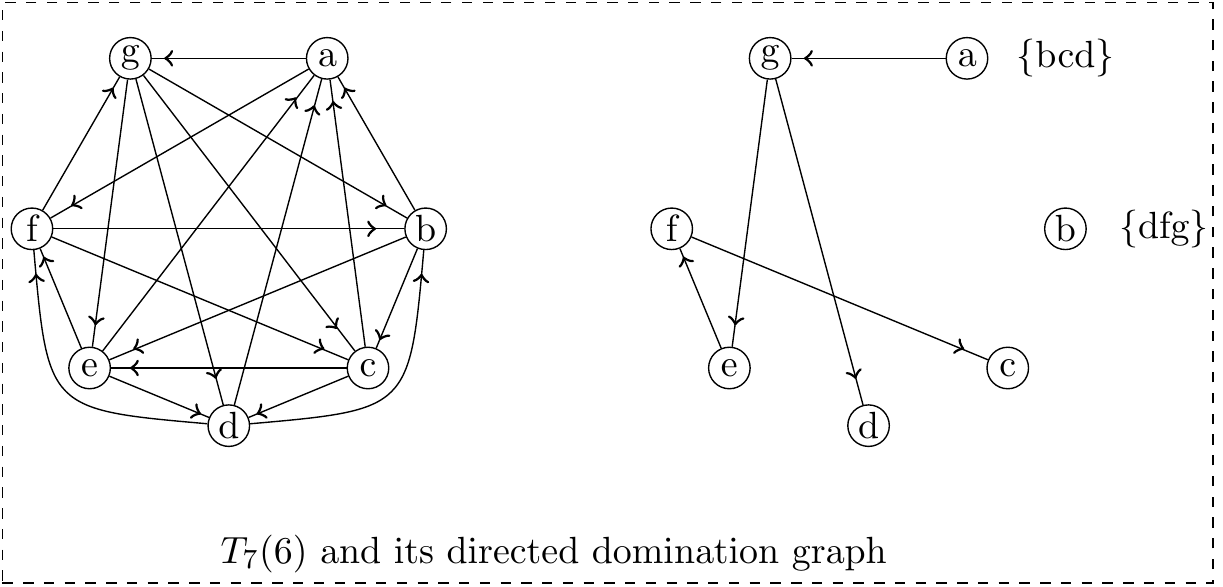}
\includegraphics[width=0.45\textwidth]{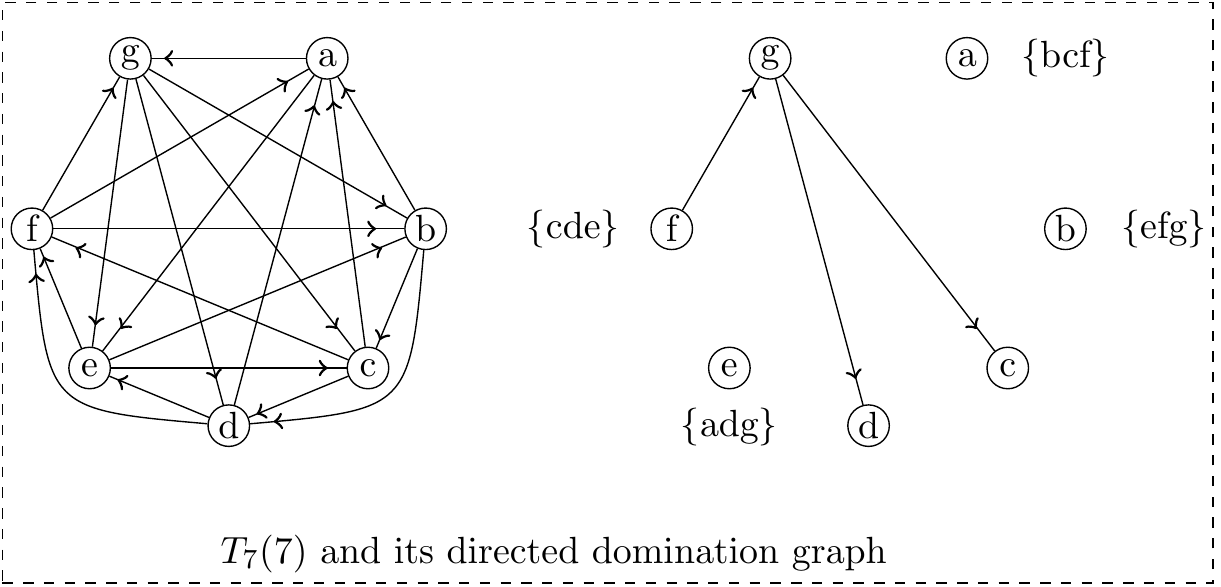}
\includegraphics[width=0.45\textwidth]{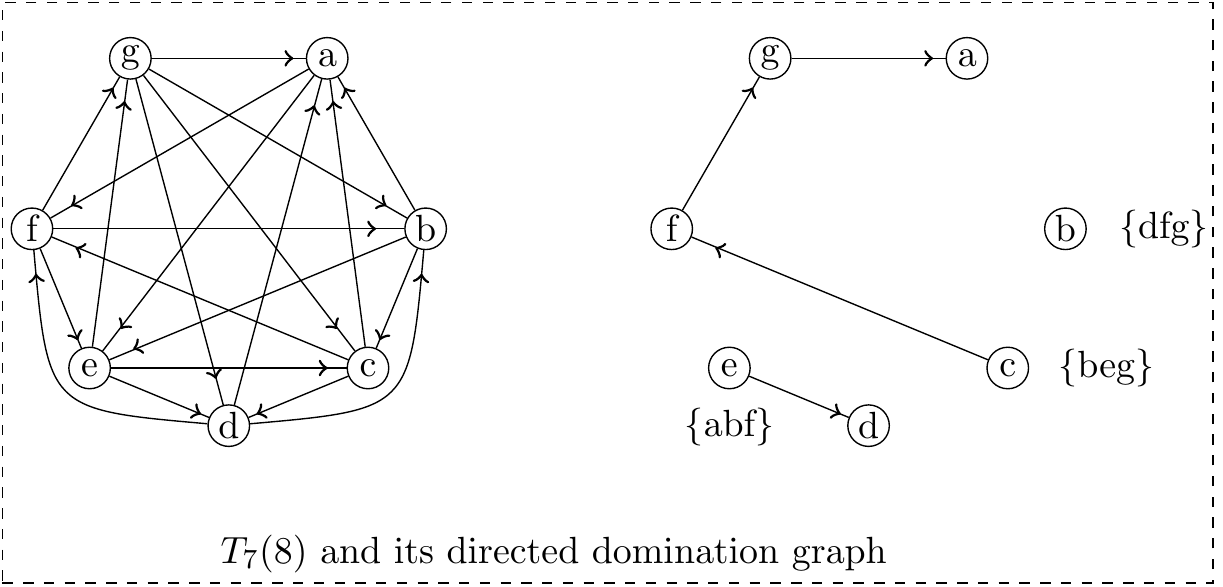}
\includegraphics[width=0.45\textwidth]{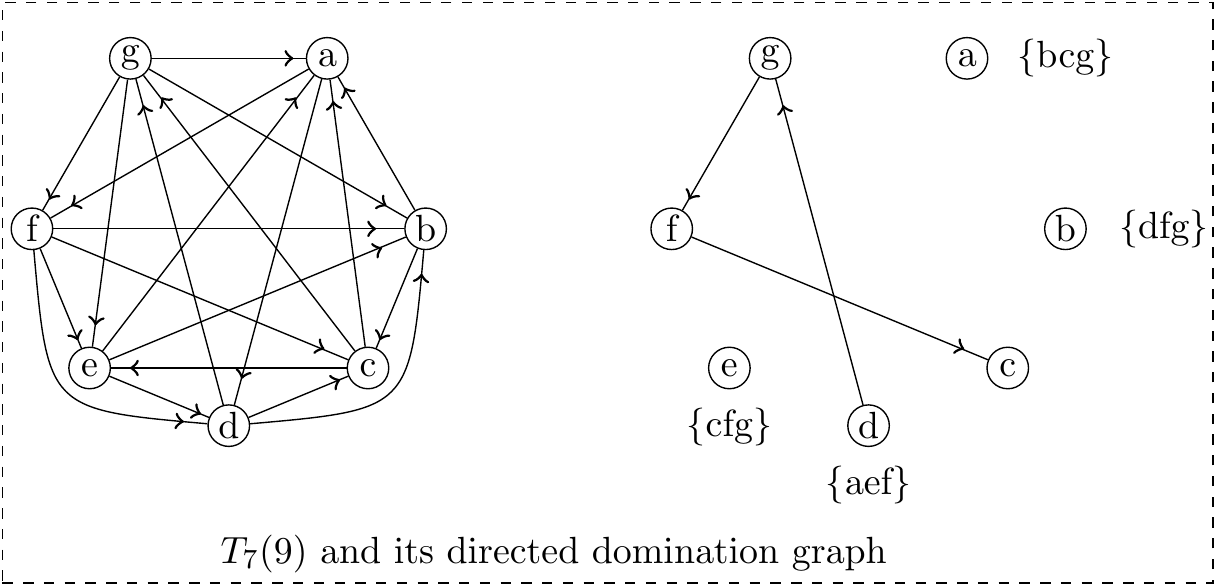}
\includegraphics[width=0.45\textwidth]{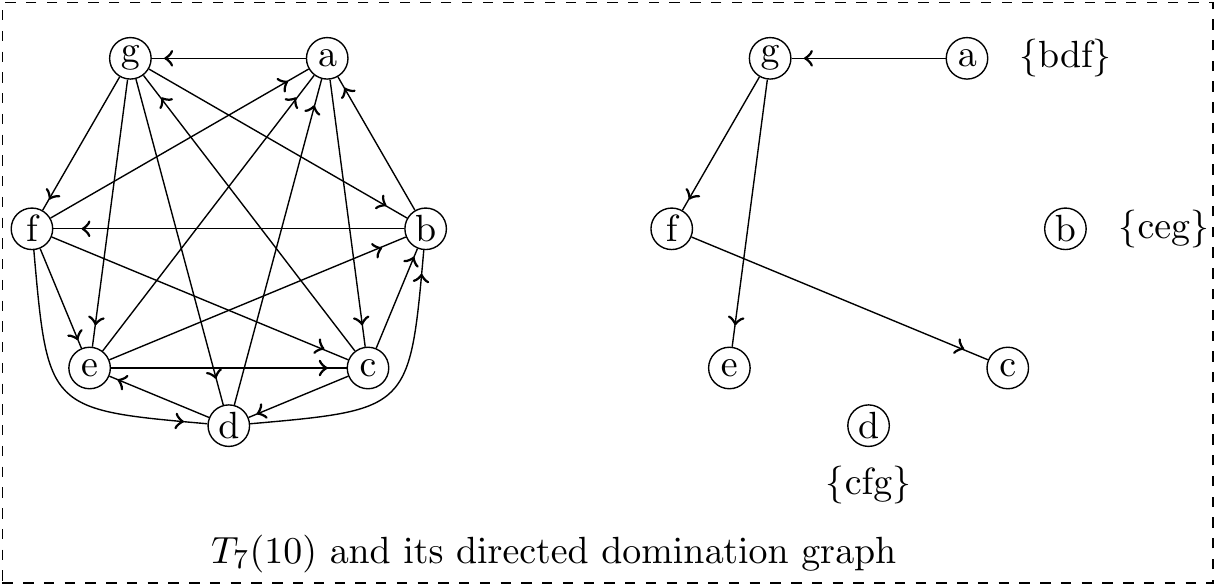}
\includegraphics[width=0.45\textwidth]{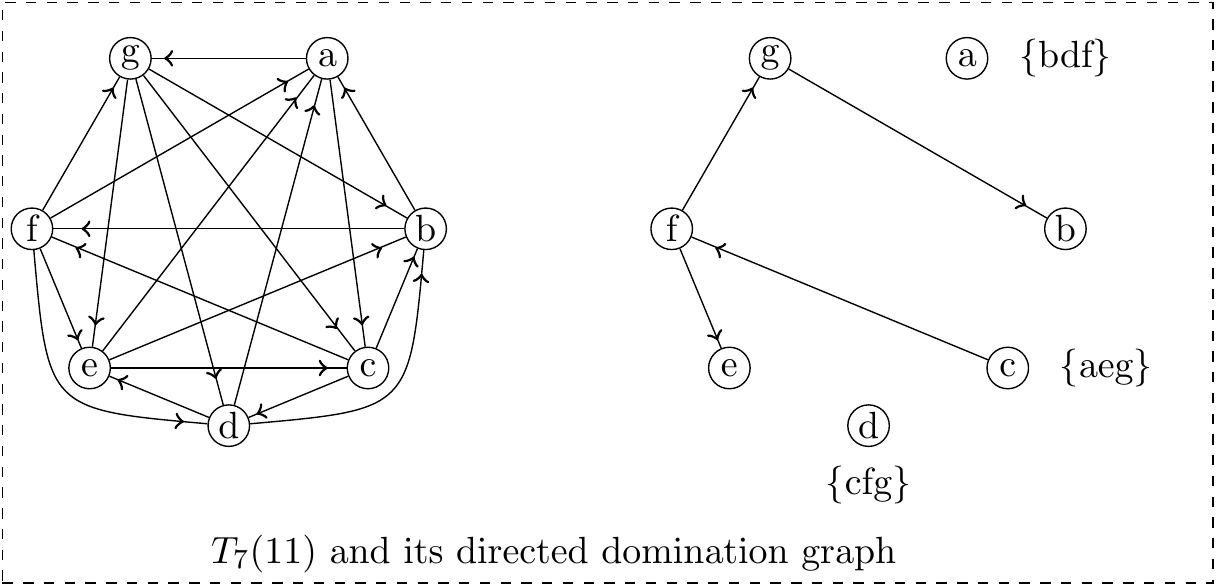}
\includegraphics[width=0.45\textwidth]{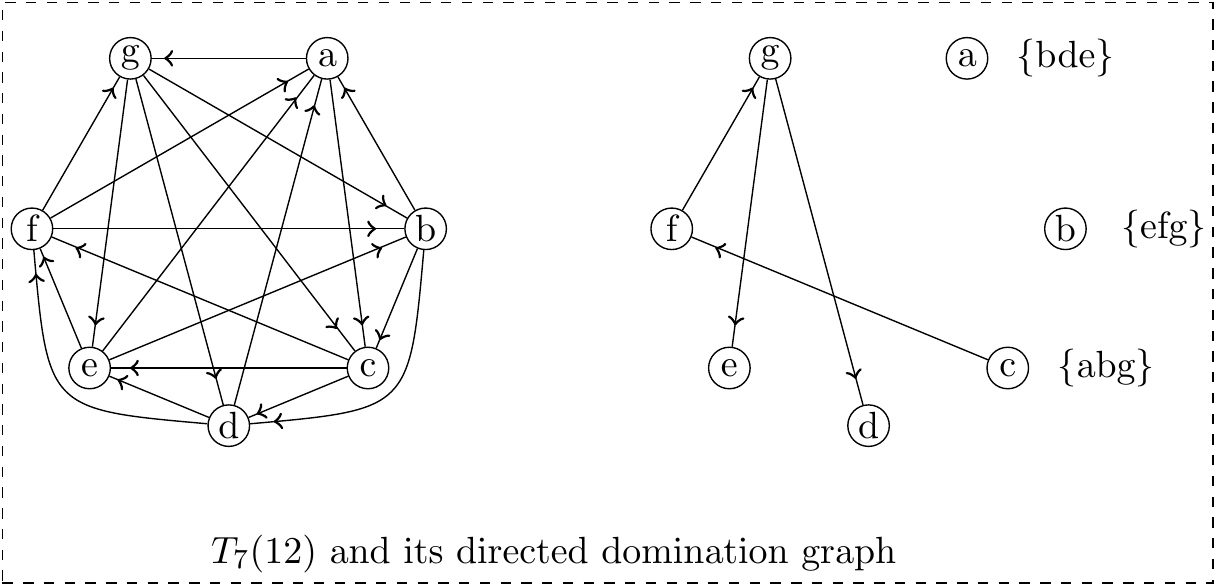}
\includegraphics[width=0.45\textwidth]{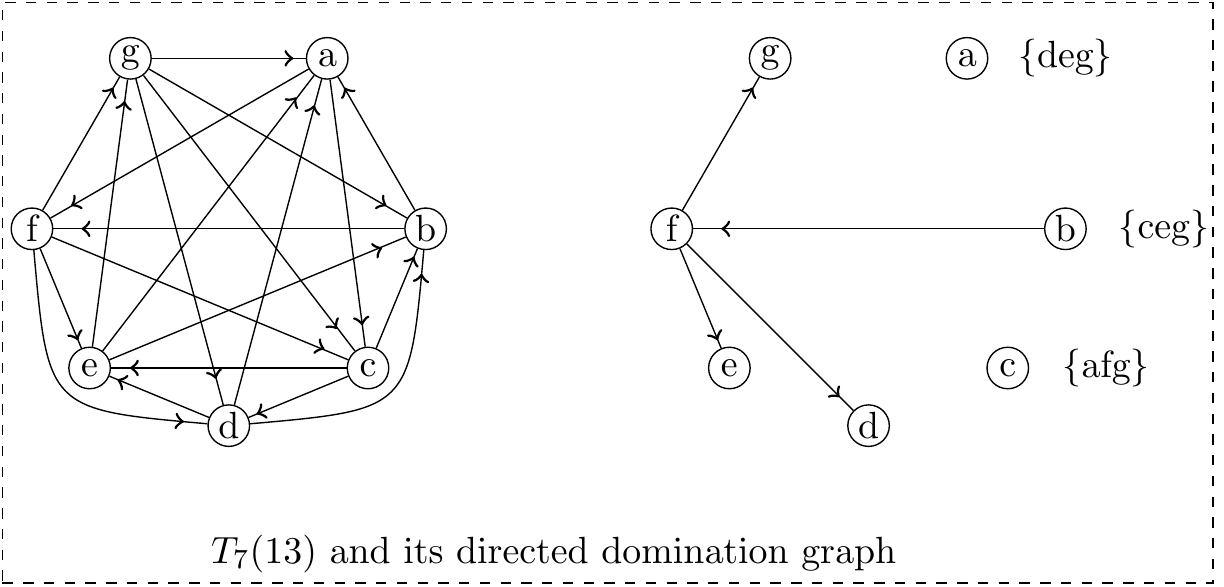}
\includegraphics[width=0.45\textwidth]{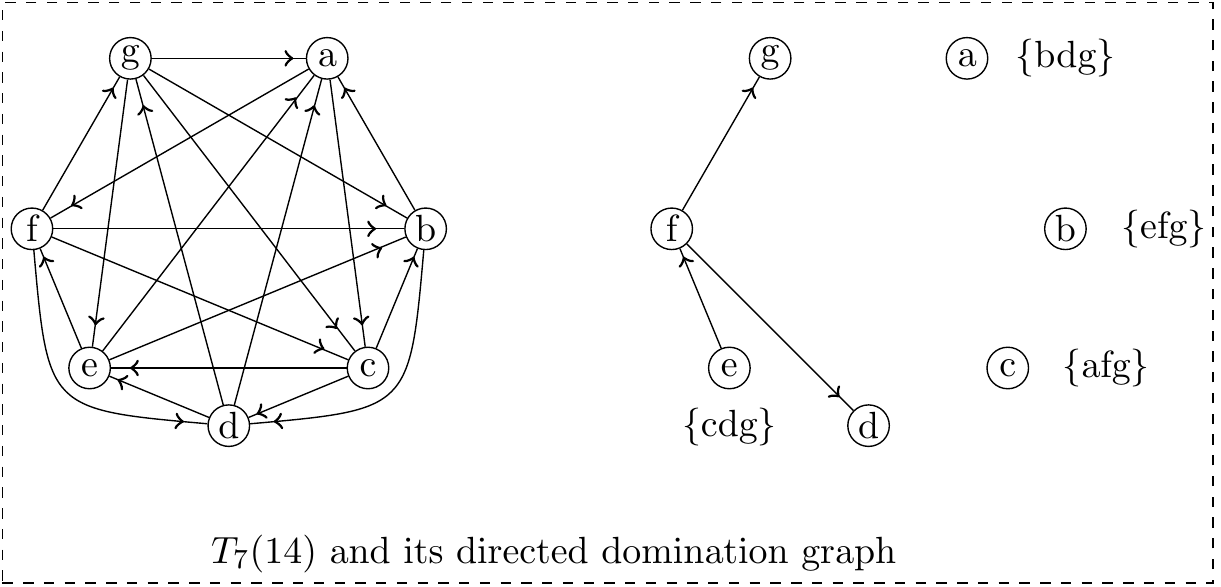}
\includegraphics[width=0.45\textwidth]{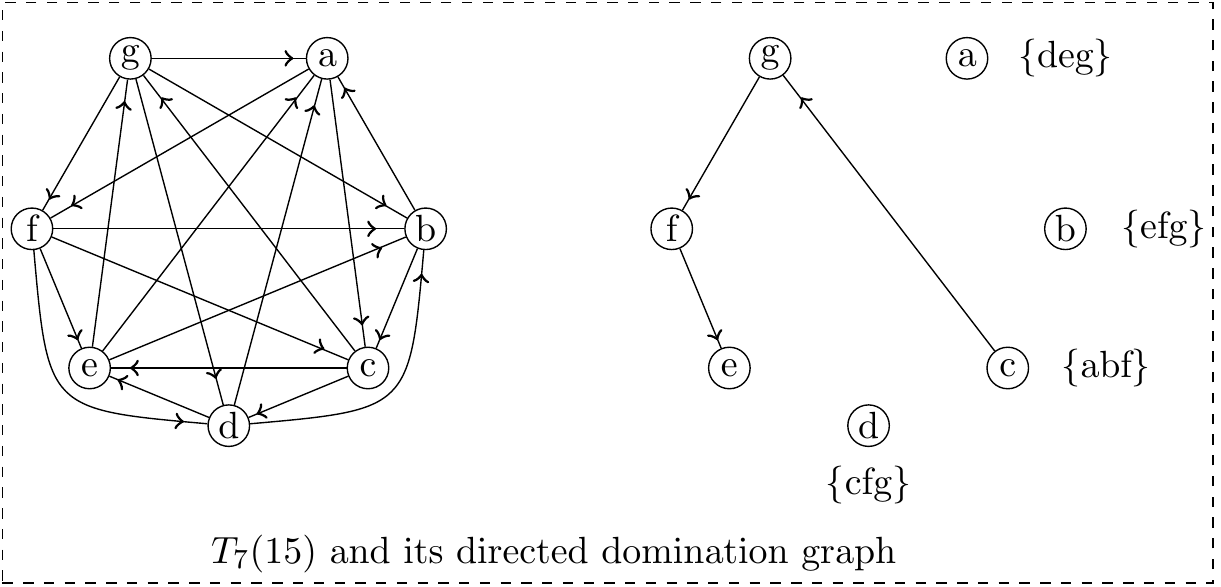}
\includegraphics[width=0.45\textwidth]{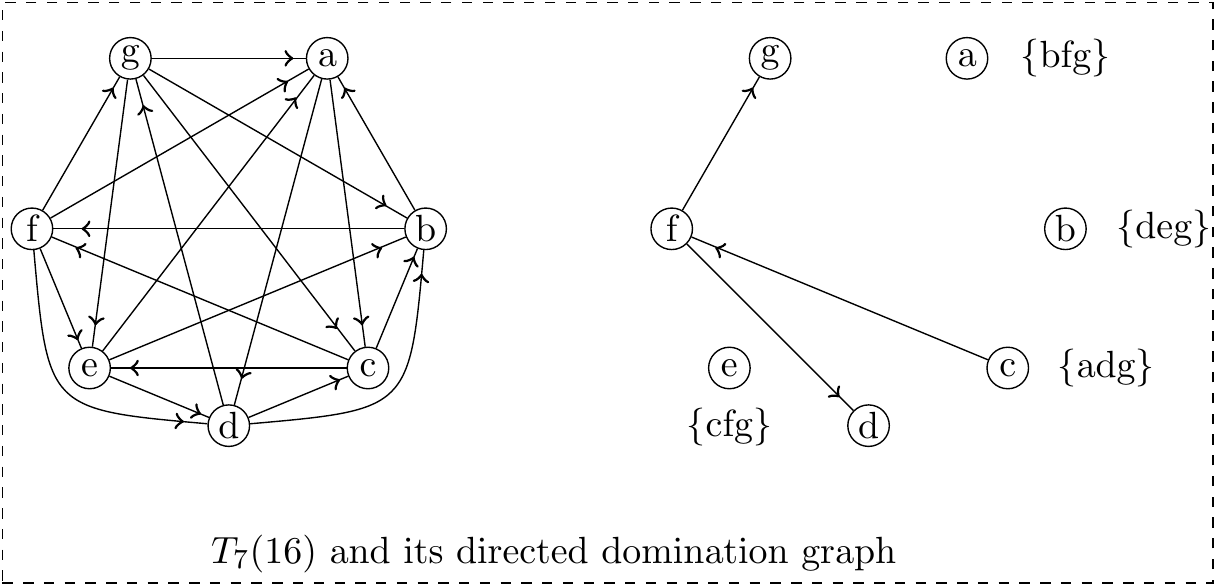}
\includegraphics[width=0.45\textwidth]{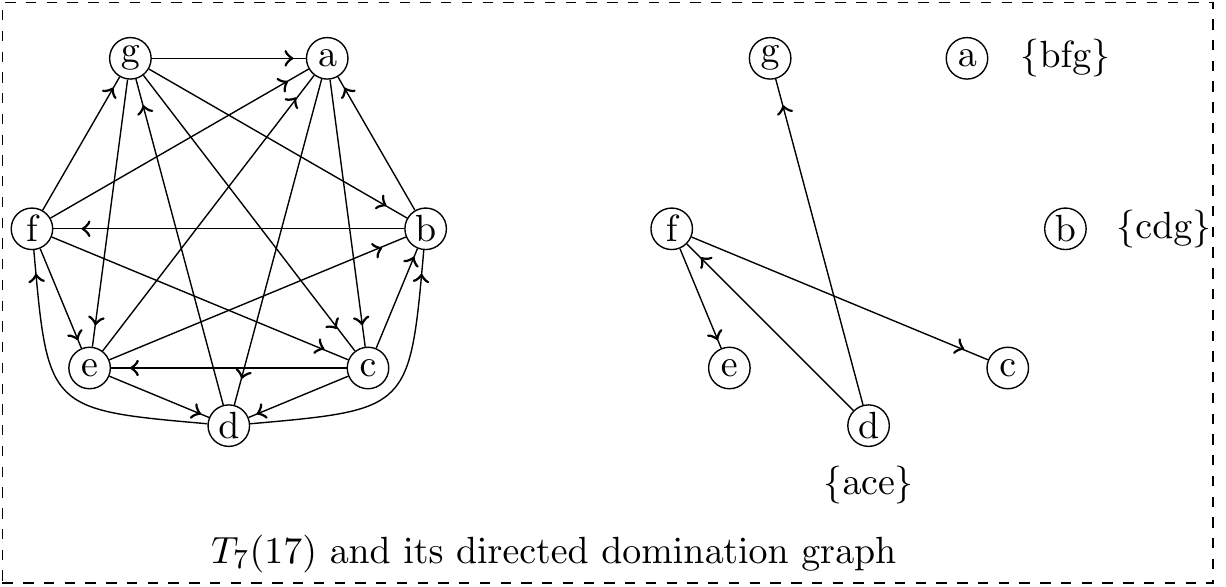}
\includegraphics[width=0.45\textwidth]{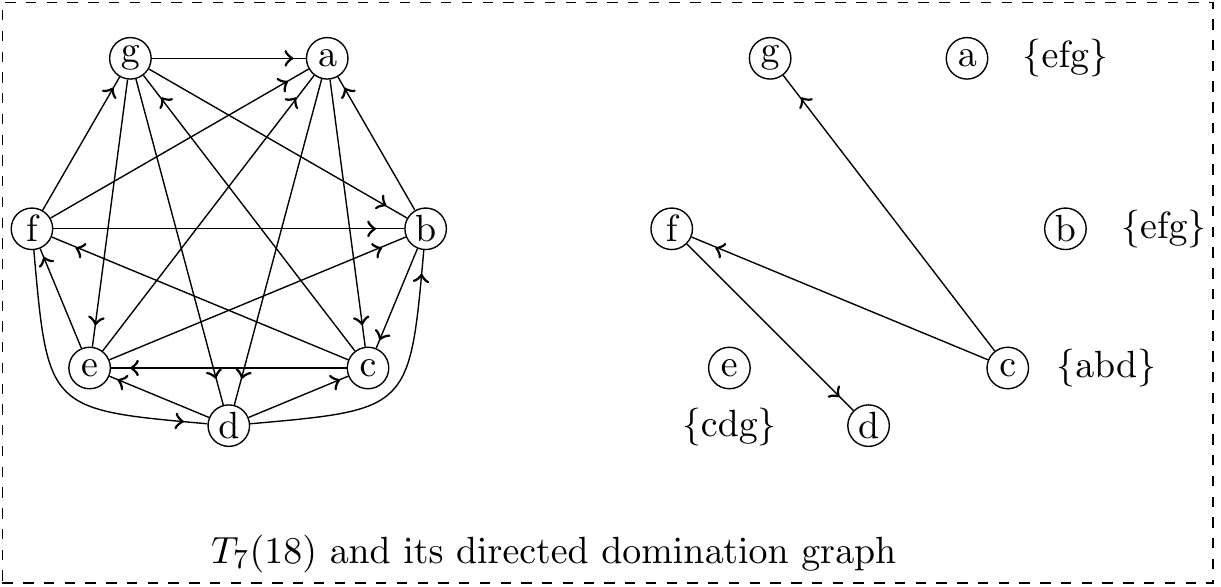}
\includegraphics[width=0.45\textwidth]{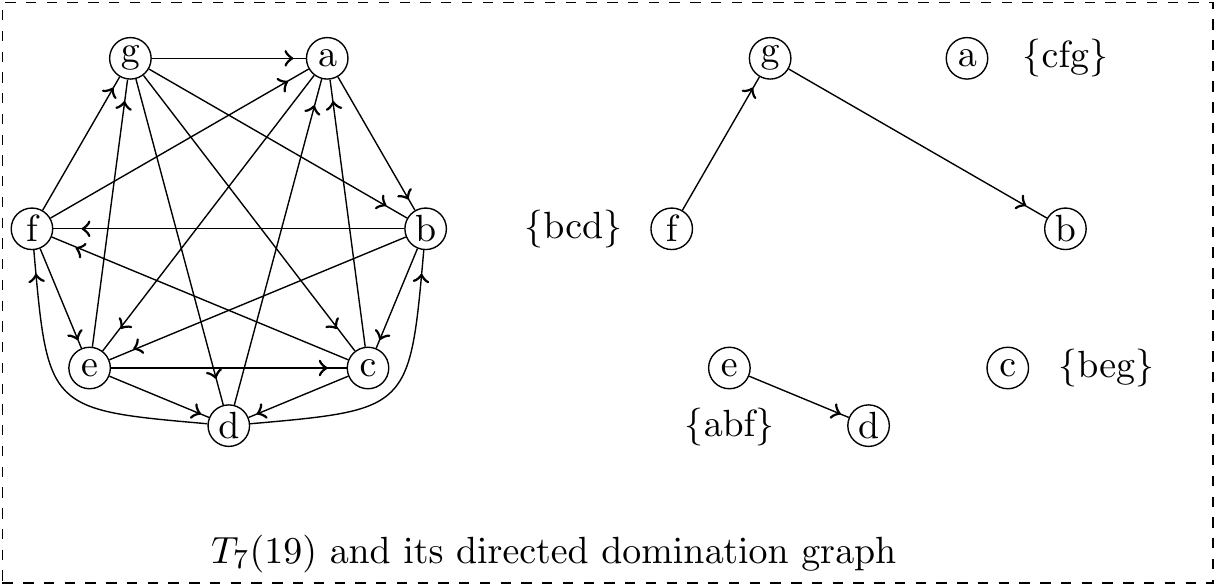}
\includegraphics[width=0.45\textwidth]{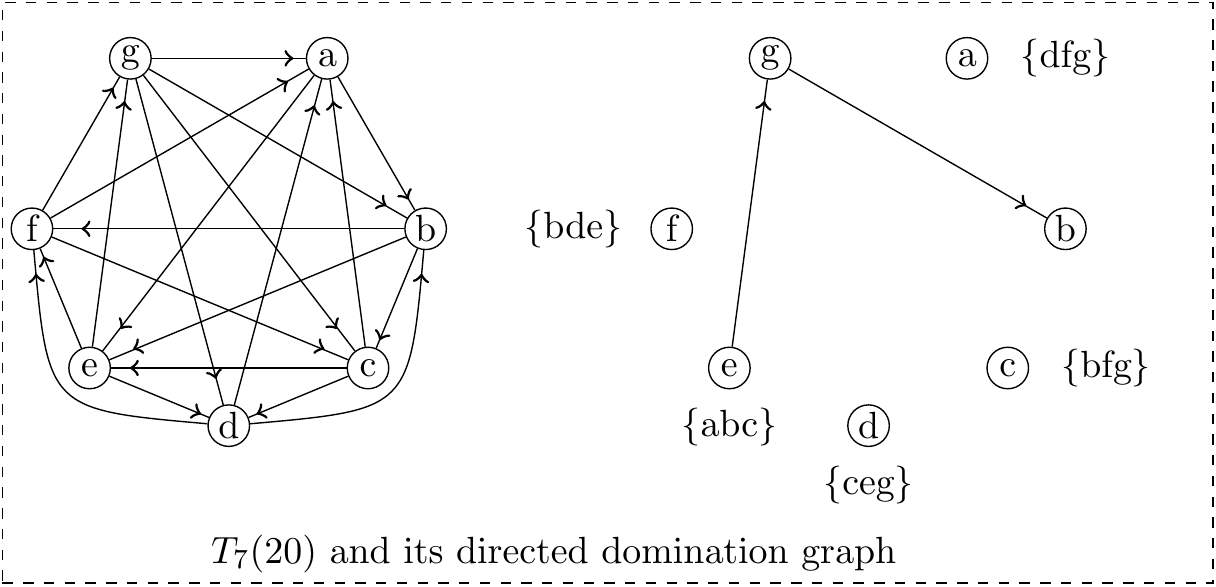}
\includegraphics[width=0.45\textwidth]{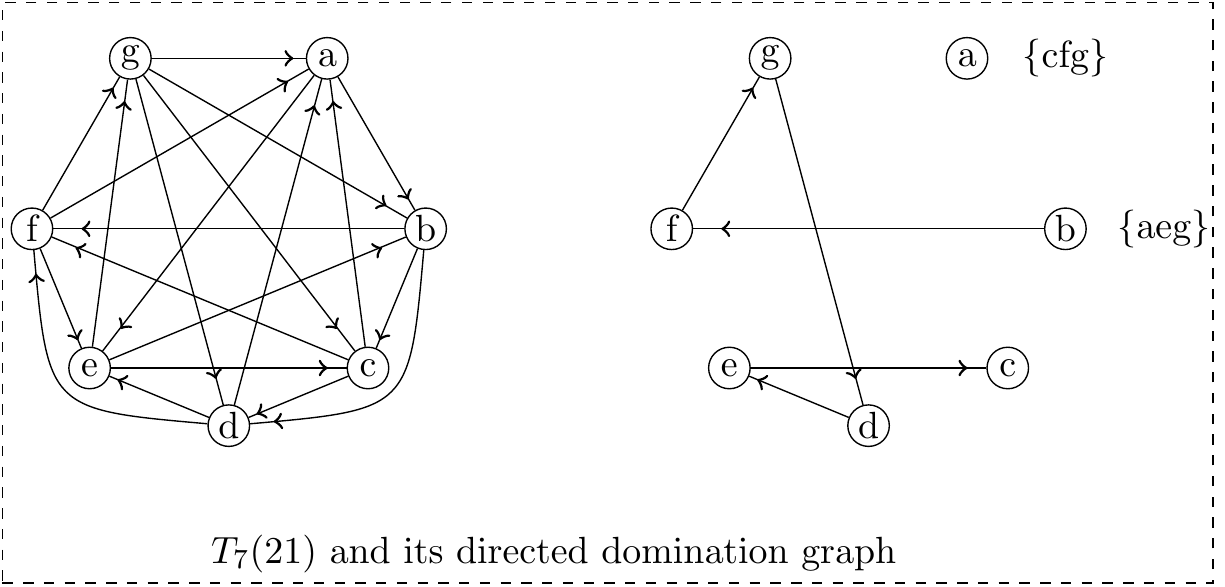}
\includegraphics[width=0.45\textwidth]{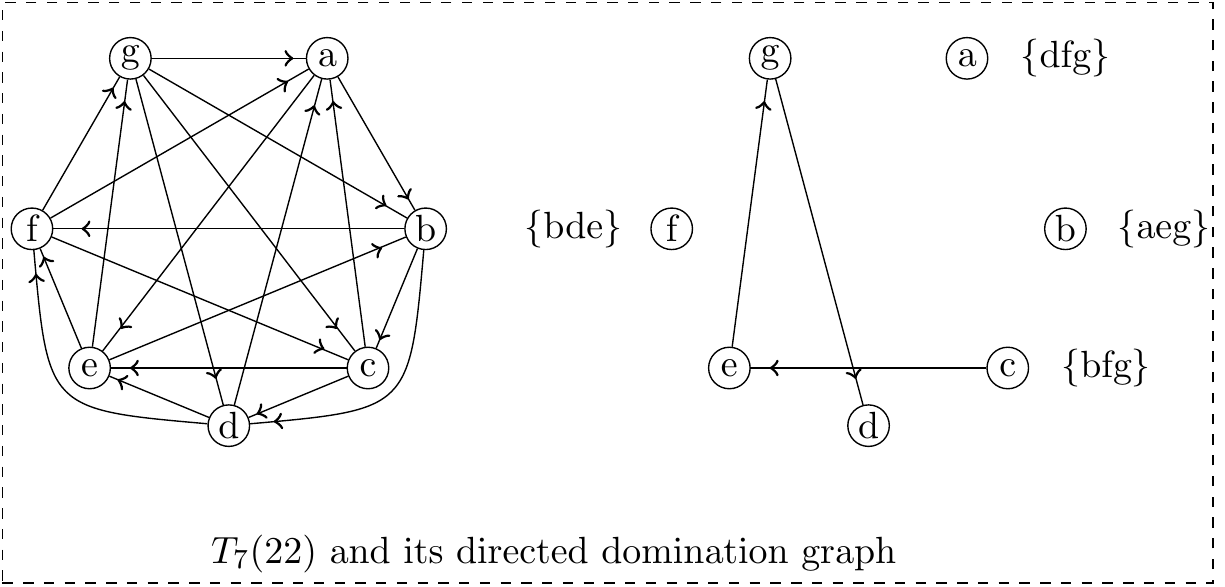}
\includegraphics[width=0.45\textwidth]{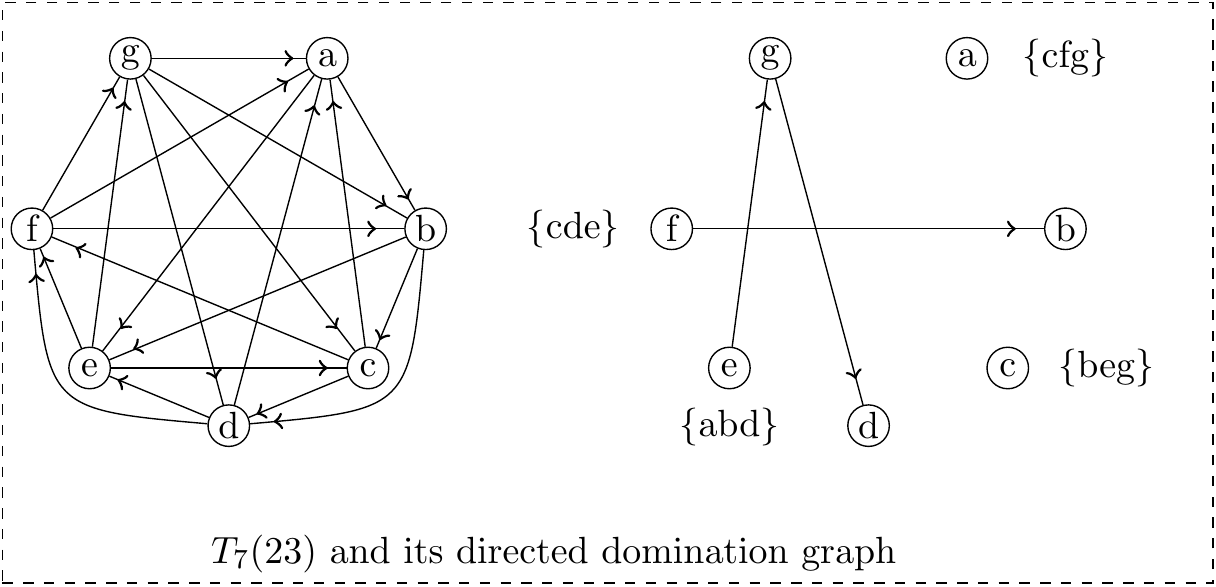}
\includegraphics[width=0.45\textwidth]{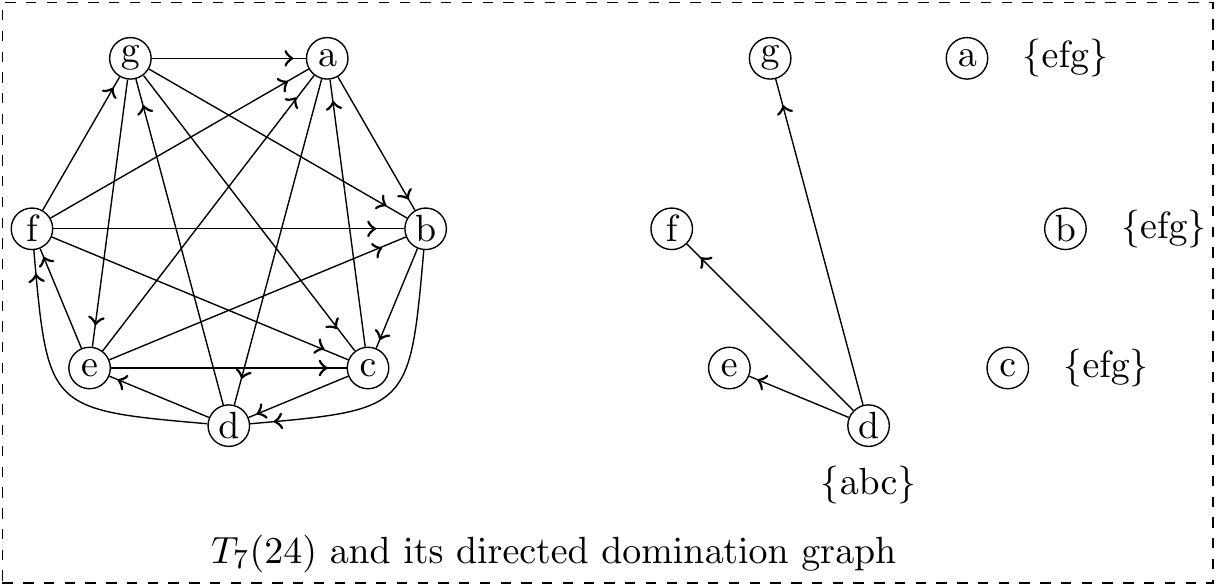}
\includegraphics[width=0.45\textwidth]{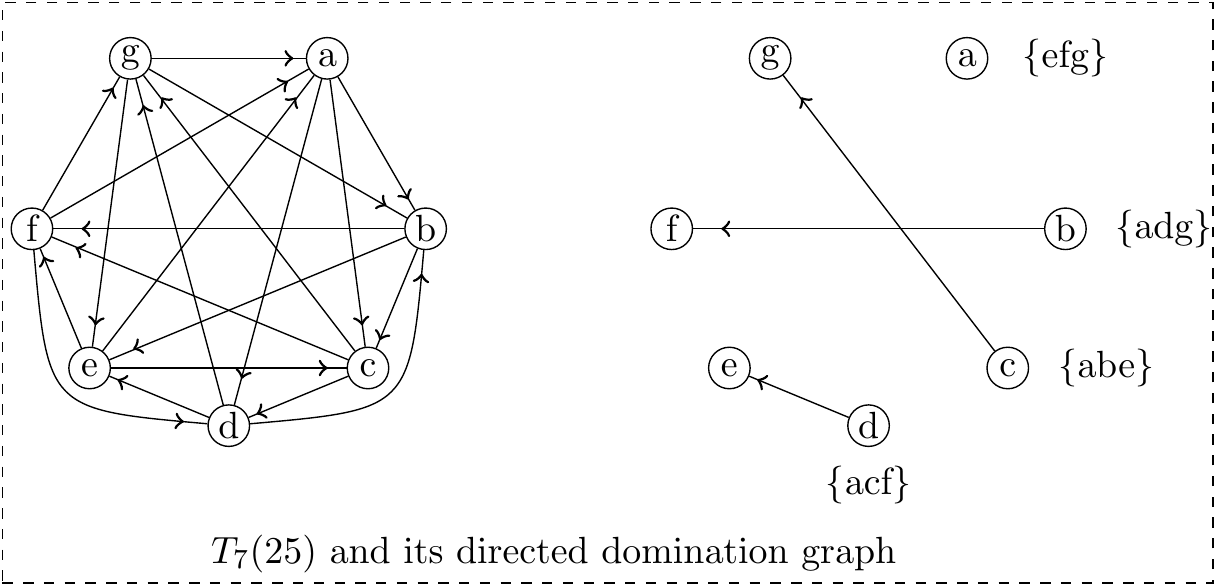}
\includegraphics[width=0.45\textwidth]{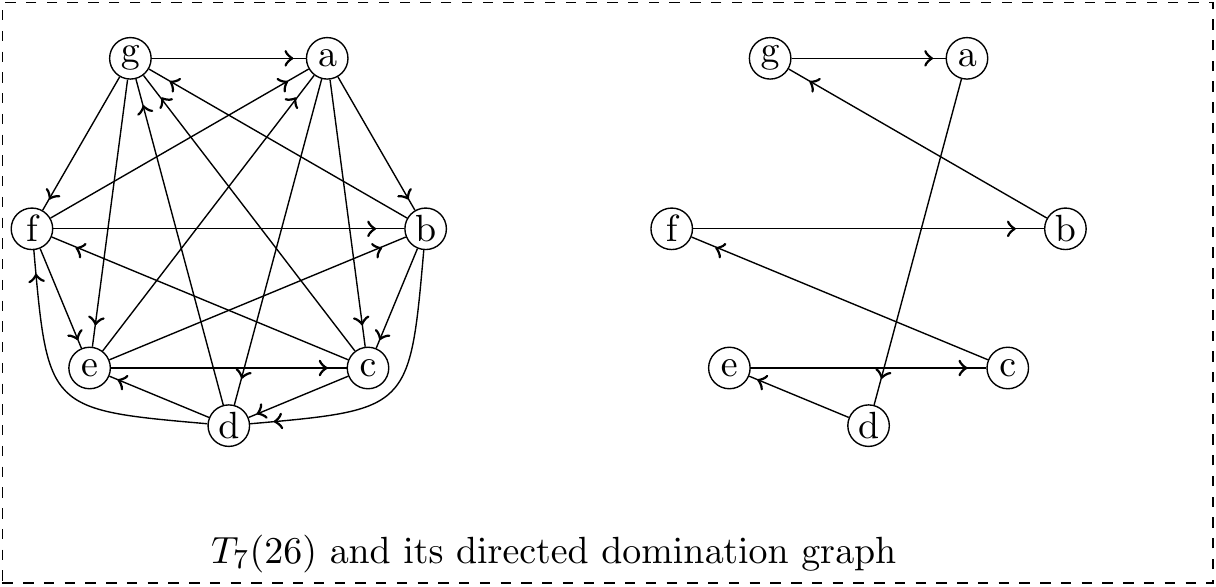}
\end{center}
\end{document}